\documentclass[11pt,reqno]{amsart}
\usepackage{amssymb,amscd,amsfonts,amsbsy}
\usepackage{enumerate}
\usepackage{amsmath,amsthm,amssymb,amsfonts}
\usepackage{mathrsfs}
\usepackage{epsf,epsfig}

\usepackage{color}
\usepackage[colorlinks=true, pdfstartview=FitV, linkcolor=blue, citecolor=blue, urlcolor=blue]{hyperref}

\numberwithin{equation}{section}

\newcommand{\meanint}{{\int{\mkern-19mu}-}}

\newtheorem{proposition}{Proposition}[section]
\newtheorem{theorem}[proposition]{Theorem}
\newtheorem{lemma}[proposition]{Lemma}
\newtheorem{corollary}[proposition]{Corollary}

\newtheorem{remark}[proposition]{Remark}

\newtheorem{definition}[proposition]{Definition}

\begin{document}

\title[Characterizing Regularity of Domains via Riesz Transforms]
{Characterizing Regularity of Domains via the Riesz Transforms on
their Boundaries}

\author{Dorina Mitrea, Marius Mitrea, and Joan Verdera}

\date{January 24, 2016}

\subjclass[2010]{Primary 42B20; Secondary 15A66, 42B37.
\newline
{\it Key words}: Singular integral operator, Riesz transform,
uniform rectifiability, H\"older space, Lyapunov domain, Clifford
algebra, Cauchy-Clifford operator, BMO, VMO, Reifenberg flat, SKT
domain, Besov space}

\begin{abstract}
Under mild geometric measure theoretic assumptions on an open subset
$\Omega$ of ${\mathbb{R}}^n$, we show that the Riesz transforms on
its boundary are continuous mappings on the H\"older space
${\mathscr{C}}^\alpha(\partial\Omega)$ if and only if $\Omega$ is a
Lyapunov domain of order $\alpha$ (i.e., a domain of class
${\mathscr{C}}^{1+\alpha}$). In the category of Lyapunov domains we also
establish the boundedness on H\"older spaces of
singular integral operators with kernels of the form
$P(x-y)/|x-y|^{n-1+l}$, where $P$ is any odd homogeneous polynomial
of degree $l$ in ${\mathbb{R}}^n$. This family of
singular integral operators, which may be thought of as generalized
Riesz transforms, includes the boundary layer potentials associated
with basic PDE's of mathematical physics, such as the Laplacian, the
Lam\'e system, and the Stokes system. We also consider the limiting
case $\alpha=0$ (with ${\rm VMO}(\partial\Omega)$ as the natural
replacement of ${\mathscr{C}}^\alpha(\partial\Omega)$), and discuss
an extension to the scale of Besov spaces.
\end{abstract}

\maketitle

\allowdisplaybreaks

\section{Introduction}
\setcounter{equation}{0} \label{S-1}

Let $\Omega\subset{\mathbb{R}}^n$ be an open set. Singular integral
operators mapping functions on $\partial\Omega$ into functions
defined either on $\partial\Omega,$ or in $\Omega$, arise naturally
in many branches of mathematics and engineering. From the work of
G.~David and S.~Semmes (cf. \cite{DaSe91}, \cite{DS2}) we know that
uniformly rectifiable ({\rm UR}) sets make up the
most general context in which Calder\'on-Zygmund like operators are
bounded on Lebesgue spaces $L^p$, with $p\in(1,\infty)$ (see
Theorem~\ref{Main-T2-BBBB} in the body of the paper
for a concrete illustration of the scope of this theory). David and Semmes
have also proved that, under the background assumption of Ahlfors regularity,
uniform rectifiability is implied by the simultaneous $L^2$ boundedness
of all integral convolution type operators on $\partial\Omega$, whose
kernels are smooth, odd, and satisfy standard growth conditions (cf.
\cite[Definition~1.20, p.\,11]{DS2}). In fact, a remarkable recent
result proved by F. Nazarov, X. Tolsa, and A. Volberg in \cite{NTV}
states that the $L^2$-boundedness of the Riesz transforms alone
yields uniform rectifiability. The corresponding result in the
plane was proved much earlier in \cite{MMV}.

The above discussion points to uniform rectifiability as being
intimately connected with the boundedness of a large class of
Calder\'on-Zygmund like operators on Lebesgue spaces. This being
said, uniform rectifiability is far too weak to guarantee by itself
analogous boundedness properties in other functional analytic
contexts, such as the scale of H\"older spaces
${\mathscr{C}}^\alpha$, with $\alpha\in(0,1)$.

The goal of this paper is to identify the category of domains for
which the Riesz transforms are bounded on H\"older spaces as the
class of Lyapunov domains (cf. Definition~\ref{lipdom}), and also
show that, in fact, a much larger family of singular integral
operators (generalizing the Riesz transforms) act naturally in this
setting. On this note we wish to remark that the trade-mark property
of Lyapunov domains is the H\"older continuity of their outward unit
normals. Alternative characterizations, of a purely geometric
flavor, may be found in \cite{ABMMZ}. The issue of boundedness of
singular integral operators on H\"older spaces has a long history,
with early work focused on Cauchy-type operators in the plane (cf.
\cite{Mu}, \cite{Gak}, and the references therein). More recently
this topic has been considered in \cite{Dyn1}, \cite{Dyn2},
\cite{FMM}, \cite{GG}, \cite{Ga}, \cite{Kr}, \cite{MOV},
\cite[Chapter~X, \S\,4]{Meyer}, \cite{Tay}, \cite{Wi}.

Consider an Ahlfors regular subset $\Sigma$ of ${\mathbb{R}}^n$
(i.e., a closed, nonempty, set satisfying \eqref{2.0.1}), and equip it with ${\mathcal{H}}^{n-1}$,
the $(n-1)$-dimensional Hausdorff measure in ${\mathbb{R}}^n$, restricted to $\Sigma$.
The latter measure happens to be a positive, locally finite, complete, doubling,
Borel regular (hence Radon) measure on $\Sigma$. In particular, the Lebesgue scale $L^p(\Sigma)$,
$0<p\leq\infty$, is always understood with respect to the aforementioned measure.
A good deal of analysis goes through in this setting, such as the
$L^p$-boundedness of Hardy-Littlewood maximal operator on $\Sigma$, Lebesgue's Differentiation
Theorem for locally integrable functions on $\Sigma$, and the density of H\"older functions with
bounded support in $L^p(\Sigma)$. See, e.g., \cite{AM}, \cite{CoWe71}, \cite{CW}, \cite{Ch},
and the references therein.

Classically, given an Ahlfors regular subset $\Sigma$ of ${\mathbb{R}}^n$, the Riesz transforms
are defined as principal value singular integral operators on $\Sigma$ with kernels
$(x_j-y_j)/(\omega_{n-1}|x-y|^{n})$, for $1\leq j\leq n$.
Specifically, if $\omega_{n-1}$ is the area of the unit sphere in
${\mathbb{R}}^n$, for each $j\in\{1,\dots,n\}$ define the $j$-th principal value Riesz transform
\begin{eqnarray}\label{T-pv.4khg}
\begin{array}{c}
R^{{}^{\rm pv}}_jf(x):=\lim\limits_{\varepsilon\to 0^{+}}R_{j,\varepsilon}f(x)\,\,\text{ where, for each }\,
\varepsilon>0,
\\[10pt]
\displaystyle
R_{j,\varepsilon}f(x):=\frac{1}{\omega_{n-1}}
\int\limits_{\stackrel{y\in\Sigma}{|x-y|>\varepsilon}}
\frac{x_j-y_j}{|x-y|^{n}}f(y)\,d{\mathcal{H}}^{n-1}(y),\quad x\in\Sigma.
\end{array}
\end{eqnarray}
It turns out that if $\Sigma$ is countably rectifiable (of dimension $n-1$)
then for each $f\in L^2(\Sigma)$ the above limit exists at ${\mathcal{H}}^{n-1}$-a.e.
point $x\in\Sigma$. In fact, a result of X. Tolsa (cf. \cite{Tol}) states that if
an arbitrary set $\Sigma\subset{\mathbb{R}}^{n}$ has
${\mathcal{H}}^{n-1}(\Sigma)<+\infty$ then
\begin{equation}\label{TToo-1}
\parbox{7.80cm}{$\Sigma$ is countably rectifiable (of dimension $n-1$) if and only if
for each $j\in\{1,\dots,n\}$ the limit
$$
\hskip -3.00in
\lim_{\varepsilon\to 0^+}\int\limits_{\stackrel{y\in\Sigma}{|y-x|>\varepsilon}}
\frac{x_j-y_j}{|x-y|^{n}}\,d{\mathcal{H}}^{n-1}(y)
$$
exists for ${\mathcal{H}}^{n-1}$-a.e. point $x$ belonging to $\Sigma$.}
\end{equation}

There is yet another related brand of Riesz transforms whose definition places no additional
demands on the underlying Ahlfors regular set $\Sigma$ of ${\mathbb{R}}^n$. The definition
in question is of a distribution theory flavor and proceeds by fixing $\alpha\in(0,1)$ and
considering ${\mathscr{C}}_c^\alpha(\Sigma)$, the space of H\"older functions of order $\alpha$
with compact support in $\Sigma$. This is a Banach space, and we denote by
$\big({\mathscr{C}}_c^\alpha(\Sigma)\big)^\ast$ its dual. Then, for each $j\in\{1,\dots,n\}$,
one defines the $j$-th distributional Riesz transform as the operator
\begin{equation}\label{yrf56f-RRR}
R_j:{\mathscr{C}}_c^\alpha(\Sigma)\longrightarrow\big({\mathscr{C}}_c^\alpha(\Sigma)\big)^\ast
\end{equation}
with the property that for every $f,g\in{\mathscr{C}}_c^\alpha(\Sigma)$ one has
\begin{equation}\label{yrf56f-RRR.1}
\big\langle R_jf,g\big\rangle=\frac{1}{2\,\omega_{n-1}}\int_{\Sigma}\int_{\Sigma}
\frac{x_j-y_j}{|x-y|^{n}}\big[f(y)g(x)-f(x)g(y)\big]\,d{\mathcal{H}}^{n-1}(y)\,d{\mathcal{H}}^{n-1}(x)
\end{equation}
where, in this context, $\langle\cdot,\cdot\rangle$ stands for the natural paring between
$\big({\mathscr{C}}_c^\alpha(\Sigma)\big)^\ast$ and ${\mathscr{C}}_c^\alpha(\Sigma)$.
It may be checked without difficulty that the above integral is absolutely convergent,
ultimately rendering the distributional Riesz transform $R_j$ linear and continuous in
the context of \eqref{yrf56f-RRR}. Moreover, the distributional Riesz transform $R_j$
just introduced is associated with the kernel $(x_j-y_j)/(\omega_{n-1}|x-y|^{n})$ in
the sense that for each $f\in{\mathscr{C}}_c^\alpha(\Sigma)$ the functional
$R_jf\in\big({\mathscr{C}}_c^\alpha(\Sigma)\big)^\ast$ is of function type
on the set $\Sigma\setminus{\rm supp}\,f$ and
\begin{eqnarray}\label{T-juVVae35}
R_jf(x)=\frac{1}{\omega_{n-1}}\int\limits_{\Sigma}
\frac{x_j-y_j}{|x-y|^{n}}f(y)\,d{\mathcal{H}}^{n-1}(y)\,\,\text{ for }\,\,x\in\Sigma\setminus{\rm supp}\,f.
\end{eqnarray}

The above definition of the distributional Riesz transforms is very much in line with the point of view
adopted in the statement of the classical $T(1)$ theorem of G.~David and J.-L.~Journ\'e \cite{DaJo}.
Originally formulated in the entire Euclidean space, the latter result turned out to be remarkably resilient,
in terms of the demands it places on the ambient. Indeed, the $T(1)$ theorem has been subsequently
generalized to spaces of homogeneous type (in the sense of Coifman and Weiss \cite{CoWe71}, \cite{CW}),
a setting where only the existence of a quasi-distance and a doubling measure is postulated
(cf., e.g., \cite[Theorem~12.3]{AH}, \cite[Chapter~IV]{Ch}, \cite[Theorem~5.56, p.\,166]{HaMuYa08}).
This is a framework in which an Ahlfors regular set $\Sigma\subset{\mathbb{R}}^n$, equipped with
the Euclidean distance and the $(n-1)$-dimensional Hausdorff measure, fits in naturally.

As it turns out, much information (both of analytic and geometric flavor) is encapsulated into
the action of the distributional Riesz transforms \eqref{yrf56f-RRR}-\eqref{yrf56f-RRR.1} on the
constant function $1$. Since the function $1$ may not belong to ${\mathscr{C}}_c^\alpha(\Sigma)$
(which happens precisely when $\Sigma$ is unbounded), one should be careful defining $R_j(1)$.
In agreement with the procedures set in place by the $T(1)$ theorem, we consider $R_j(1)$ to be
the linear functional acting on each function $g\in{\mathscr{C}}_c^\alpha(\Sigma)$ satisfying
the cancellation condition $\int_{\Sigma}g\,d{\mathcal{H}}^{n-1}=0$ according to
\begin{align}\label{yrf56f-RRR.2}
\big\langle R_j(1),g\big\rangle
&:=\frac{1}{2\,\omega_{n-1}}\int_{\Sigma}\int_{\Sigma}
\frac{x_j-y_j}{|x-y|^{n}}\big[\phi(y)g(x)-\phi(x)g(y)\big]\,d{\mathcal{H}}^{n-1}(y)\,d{\mathcal{H}}^{n-1}(x)
\nonumber\\[6pt]
&\quad -\frac{1}{\omega_{n-1}}\int_{\Sigma}\int_{\Sigma}\frac{x_j-y_j}{|x-y|^{n}}
(1-\phi(x))g(y)\,d{\mathcal{H}}^{n-1}(y)\,d{\mathcal{H}}^{n-1}(x),
\end{align}
where $\phi\in{\mathscr{C}}_c^\alpha(\Sigma)$ is an auxiliary function chosen to satisfy
$\phi\equiv 1$ near ${\rm supp}\,g$. In this vein, let us remark that, in the case when
$\Sigma$ is compact, we do have ${\mathscr{C}}_c^\alpha(\Sigma)={\mathscr{C}}^\alpha(\Sigma)$
hence, in particular, we now have $1\in{\mathscr{C}}_c^\alpha(\Sigma)$. In such a scenario,
it may be readily verified that $R_j(1)$, defined as in \eqref{yrf56f-RRR.2}, is the restriction
of the functional $R_j1\in\big({\mathscr{C}}_c^\alpha(\Sigma)\big)^\ast$, defined as in
\eqref{yrf56f-RRR.1} with $f=1$, to the space consisting of functions from ${\mathscr{C}}_c^\alpha(\Sigma)$
which integrate to zero. It is therefore reassuring to know that the various points of view on the
nature of the action of the distributional Riesz transform $R_j$ on the constant function $1$ are consistent.

At the analytical level, the $T(1)$ theorem (for operators associated with odd kernels)
gives that, for each fixed $j\in\{1,\dots,n\}$,
\begin{equation}\label{Ma-utrt}
\parbox{11.00cm}{the distributional Riesz transform $R_j$ from \eqref{yrf56f-RRR}-\eqref{yrf56f-RRR.1}
extends to a bounded linear operator on $L^2(\Sigma)$ if and only if $R_j(1)\in{\rm BMO}(\Sigma)$,}
\end{equation}
where ${\rm BMO}(\Sigma)$ is the John-Nirenberg space of functions of bounded
mean oscillations on $\Sigma$ (regarded as a space of homogeneous type).

At this stage, a few comments are in order, about the specific manner in which
the various brands of Riesz transforms introduced earlier relate to one another. Assume that
$\Sigma$ is an Ahlfors regular subset of ${\mathbb{R}}^n$ which is countably rectifiable
(of dimension $n-1$). First, it turns out that if for some $j\in\{1,\dots,n\}$
one (hence both) of the two equivalent conditions in \eqref{Ma-utrt} holds then the extension
of the distributional Riesz transform $R_j$ to a bounded linear operator on $L^2(\Sigma)$
(mentioned in \eqref{Ma-utrt}) is realized precisely by the principal value Riesz transform
$R^{{}^{\rm pv}}_j$ (defined for each $f\in L^2(\Sigma)$ as in \eqref{T-pv.4khg} at
${\mathcal{H}}^{n-1}$-a.e. $x\in\Sigma$). In particular, for each $j\in\{1,\dots,n\}$ there holds:
\begin{equation}\label{TTbb-1aG}
\begin{array}{c}
\parbox{8.00cm}{if $\Sigma\subset{\mathbb{R}}^n$ is a compact Ahlfors regular set which is countably rectifiable
(of dimension $n-1$) and $R_j(1)\in{\rm BMO}(\Sigma)$ then, for ${\mathcal{H}}^{n-1}$-a.e. $x\in\Sigma$,}
\\[16pt]
\displaystyle
R_j(1)(x)=\lim_{\varepsilon\to 0^+}\int\limits_{\stackrel{y\in\Sigma}{|y-x|>\varepsilon}}
\frac{x_j-y_j}{|x-y|^{n}}\,d{\mathcal{H}}^{n-1}(y).
\end{array}
\end{equation}
Second, if for some $j\in\{1,\dots,n\}$ the principal value Riesz transform
$R^{{}^{\rm pv}}_j$, originally acting on ${\mathscr{C}}_c^\alpha(\Sigma)$, is known to extend to a bounded
linear operator on $L^2(\Sigma)$, then $R^{{}^{\rm pv}}_j$ coincides on ${\mathscr{C}}_c^\alpha(\Sigma)$
with the distributional Riesz transform $R_j$ defined as in \eqref{yrf56f-RRR}-\eqref{yrf56f-RRR.1}.
Third, having fixed $j\in\{1,\dots,n\}$, the principal value Riesz transform $R^{{}^{\rm pv}}_j$
extends to a bounded linear operator on $L^2(\Sigma)$ if and only if for each $\varepsilon>0$
the $j$-th truncated Riesz transform $R_{j,\varepsilon}$ defined as in \eqref{T-pv.4khg}
is bounded on $L^2(\Sigma)$ uniformly in $\varepsilon$, if and only if the $j$-th maximal
Riesz transform $R_{j,\ast}$ is bounded on $L^2(\Sigma)$ where, for each $f\in L^2(\Sigma)$,
\begin{eqnarray}\label{T-p-utr}
R_{j,\ast}f(x):=\sup\limits_{\varepsilon>0}\big|(R_{j,\varepsilon}f)(x)\big|,\qquad x\in\Sigma.
\end{eqnarray}
All these results may be established via arguments of Calder\'on-Zygmund theory flavor, such as
Cotlar's inequality, the Calder\'on-Zygmund decomposition, Marcinkiewicz's interpolation theorem,
the boundedness of the Hardy-Littlewood maximal operator, etc.

At the geometrical level, the Nazarov, Tolsa, Volberg recent main result in \cite{NTV} mentioned
earlier may be rephrased, in light of \eqref{Ma-utrt}, as follows: under the background assumption
that $\Sigma$ is an Ahlfors regular subset of ${\mathbb{R}}^n$, one has
\begin{equation}\label{Mabb88}
\Sigma\,\mbox{ uniformly rectifiable set}\,\Longleftrightarrow\,
R_j(1)\in{\rm BMO}(\Sigma)\,\,\,\mbox{ for each }\,\,j\in\{1,\dots,n\}.
\end{equation}
Hence, within the class of Ahlfors regular subsets of ${\mathbb{R}}^n$,
the membership of the $R_j(1)$'s to the John-Nirenberg space ${\rm BMO}$
characterizes uniform rectifiability. As mentioned earlier in the introduction, this result
refines earlier work of G.~David and S.~Semmes \cite{DaSe91} in which these authors have
proved that uniform rectifiability within the class of Ahlfors regular subsets of
${\mathbb{R}}^n$ is equivalent to the $L^2$-boundedness in that ambient of all truncated
singular integral operators, uniform with respect to the truncation (or, equivalently,
the $L^2$-boundedness of all maximal operators), associated with all kernels of the
form $k(x-y)$, where the function $k\in{\mathscr{C}}^\infty({\mathbb{R}}^n\setminus\{0\})$
is odd and satisfies
\begin{equation}\label{Ma-hREDa}
\sup_{x\in{\mathbb{R}}^n\setminus\{0\}}\Big[|x|^{(n-1)+|\gamma|}\big|(\partial^\gamma k)(x)\big|\Big]<+\infty,
\qquad\forall\,\gamma\in{\mathbb{N}}_0^n.
\end{equation}

In relation to the brands of Riesz transforms introduced earlier, the results
of G.~David and S.~Semmes in \cite{DaSe91} imply\footnote{in concert with the Calder\'on-Zygmund
machinery alluded to earlier, and bearing in mind \eqref{URRFVCa}} that, for each $j\in\{1,\dots,n\}$,
\begin{equation}\label{uytggf-RRR}
\parbox{11.00cm}{whenever $\Sigma$ is a uniformly rectifiable set in ${\mathbb{R}}^n$,
the principal value Riesz transform $R^{{}^{\rm pv}}_j$ is a well-defined, linear and bounded
operator on $L^2(\Sigma)$, which agrees on ${\mathscr{C}}^\alpha_c(\Sigma)$ with the distributional
Riesz transform $R_j$.}
\end{equation}

From the perspective of \eqref{Mabb88}, one of the issues addressed by our first main result is
that of extracting more geometric regularity for $\Sigma$ if more analytic regularity
for the $R_j(1)$'s is available. We shall study this issue in the case when $\Sigma:=\partial\Omega$,
the topological boundary of an open subset $\Omega$ of ${\mathbb{R}}^n$.
This fits into the paradigm of describing
geometric characteristics (such as regularity of a certain nature)
of a given set in terms of properties of suitable analytical
entities (such as singular integral operators) associated with this
environment. Specifically, we have the following theorem (for all
relevant definitions the reader is referred to \S\ref{S-2}).

\begin{theorem}\label{Main-T1aa}
Assume $\Omega\subseteq\mathbb{R}^n$ is an Ahlfors regular domain with a compact boundary, satisfying
$\partial\Omega=\partial(\overline{\Omega})$. Set $\sigma:={\mathcal{H}}^{n-1}\lfloor\partial\Omega$
and define $\Omega_{+}:=\Omega$ and $\Omega_{-}:=\mathbb{R}^n\setminus\overline{\Omega}$.

Then for each $\alpha\in(0,1)$ the following claims are equivalent:
\begin{enumerate}
\item[{\rm (a)}] $\Omega$ is a domain of class ${\mathscr{C}}^{1+\alpha}$
{\rm (}or Lyapunov domain of order $\alpha${\rm )};
\item[{\rm (b)}] the distributional Riesz transforms, defined as in
\eqref{yrf56f-RRR}-\eqref{yrf56f-RRR.1} with $\Sigma:=\partial\Omega$,
satisfy
\begin{equation}\label{eq:RIESZ33}
R_j1\in{\mathscr{C}}^{\alpha}(\partial\Omega)
\,\,\,\mbox{ for each }\,\,j\in\{1,\dots,n\};
\end{equation}
\item[{\rm (c)}] $\Omega$ is a {\rm UR} domain and, given any odd homogeneous polynomial
$P$ of degree $l\geq 1$ in ${\mathbb{R}}^{n}$, the singular integral operator
\begin{eqnarray}\label{T-pv.44}
Tf(x):=\lim_{\varepsilon\to 0^{+}}
\int\limits_{\stackrel{y\in\partial\Omega}{|x-y|>\varepsilon}}
\frac{P(x-y)}{|x-y|^{n-1+l}}f(y)\,d\sigma(y),\qquad
x\in\partial\Omega,
\end{eqnarray}
is meaningfully defined for every $f\in{\mathscr{C}}^\alpha(\partial\Omega)$, and
maps ${\mathscr{C}}^\alpha(\partial\Omega)$ boundedly into itself;
\item[{\rm (d)}] $\Omega$ is a {\rm UR} domain and one has
\begin{equation}\label{eq:RIEtt}
{\mathscr{R}}^{\pm}_j1\in{\mathscr{C}}^{\alpha}(\Omega_{\pm})\,\,\,\mbox{
for each }\,\,j\in\{1,\dots,n\}
\end{equation}
where, for $j\in\{1,\dots,n\}$,
\begin{eqnarray}\label{T-pv.4kh445}
{\mathscr{R}}^{\pm}_jf(x):=\frac{1}{\omega_{n-1}}\int\limits_{\partial\Omega}
\frac{x_j-y_j}{|x-y|^{n}}f(y)\,d\sigma(y),\qquad x\in\Omega_{\pm};
\end{eqnarray}
\item[{\rm (e)}] $\Omega$ is a {\rm UR} domain and, for each odd homogeneous
polynomial $P$ of degree $l\geq 1$ in ${\mathbb{R}}^{n}$, the integral
operators
\begin{eqnarray}\label{T-layer.44}
{\mathbb{T}}_{\pm}f(x):=\int\limits_{\partial\Omega}\frac{P(x-y)}{|x-y|^{n-1+l}}f(y)\,d\sigma(y),
\qquad x\in\Omega_\pm,
\end{eqnarray}
map ${\mathscr{C}}^\alpha(\partial\Omega)$ boundedly into
${\mathscr{C}}^\alpha\big(\Omega_{\pm}\big)$.
\end{enumerate}

Moreover, if $\Omega$ is a ${\mathscr{C}}^{1+\alpha}$ domain for
some $\alpha\in(0,1)$, there exists a finite constant $C>0$,
depending only on $n$, $\alpha$, ${\rm diam}(\partial\Omega)$, the
upper Ahlfors regularity constant of $\partial\Omega$, and
$\|\nu\|_{{\mathscr{C}}^\alpha(\partial\Omega)}$ {\rm (}where $\nu$
is the outward unit normal to $\Omega${\rm )} with the property that
for each odd homogeneous polynomial $P$ of degree $l\geq 1$ in ${\mathbb{R}}^{n}$
the integral operators \eqref{T-layer.44}, \eqref{T-pv.44} satisfy
\begin{align}\label{jhygff8533}
\big\|{\mathbb{T}}_{\pm}f\big\|_{{\mathscr{C}}^\alpha\big(\overline{\Omega_{\pm}}\,\big)}
\leq
C^{l}2^{l^2}\|P\|_{L^2(S^{n-1})}\|f\|_{{\mathscr{C}}^\alpha(\partial\Omega)},\qquad
\forall\,f\in{\mathscr{C}}^\alpha(\partial\Omega),
\end{align}
and
\begin{align}\label{jhygff8533.2}
\big\|Tf\big\|_{{\mathscr{C}}^\alpha(\partial\Omega)} \leq
C^{l}2^{l^2}\|P\|_{L^2(S^{n-1})}\|f\|_{{\mathscr{C}}^\alpha(\partial\Omega)},\qquad
\forall\,f\in{\mathscr{C}}^\alpha(\partial\Omega).
\end{align}
\end{theorem}

The operators described in \eqref{T-pv.44} may be thought of as
generalized Riesz transforms since they correspond
to \eqref{T-pv.44} with
\begin{equation}\label{eq:Pjah}
P(x):=x_j/\omega_{n-1}\,\,\mbox{ for
}\,\,x=(x_1,\dots,x_n)\in{\mathbb{R}}^n, \quad 1 \le j \le n.
\end{equation}
For the same choices of the polynomials, the claim in part {\rm (e)}
of Theorem~\ref{Main-T1aa} implies that the harmonic single-layer
operator (cf. \eqref{iu7grEE} for a definition) is well-defined,
linear, and bounded as a mapping
\begin{equation}\label{Rdac-1}
\mathscr{S}:{\mathscr{C}}^\alpha(\partial\Omega)\longrightarrow
{\mathscr{C}}^{1+\alpha}(\Omega_{\pm}).
\end{equation}

In concert with the comments meant to clarify how the distributional Riesz transforms related
to the principal value Riesz transforms, Theorem~\ref{Main-T1aa} readily implies the following corollary.

\begin{corollary}\label{Main-T1aa-CCCa}
Let $\Omega$ be a nonempty, proper, open subset of $\mathbb{R}^n$ with compact boundary,
satisfying $\partial\Omega=\partial\big(\,\overline{\Omega}\,\big)$. Then for every
$\alpha\in(0,1)$ the following statements are equivalent:
\begin{enumerate}
\item[{\rm (i)}] $\Omega$ is a domain of class ${\mathscr{C}}^{1+\alpha}$;
\item[{\rm (ii)}] $\Omega$ is an Ahlfors regular domain and, for each $j\in\{1,\dots,n\}$
the distributional Riesz transform $R_j$ defined as in \eqref{yrf56f-RRR}-\eqref{yrf56f-RRR.1}
with $\Sigma:=\partial\Omega$ induces a linear and bounded operator in the context
\begin{equation}\label{eq:RIE-yttr.1}
R_j:{\mathscr{C}}^{\alpha}(\partial\Omega)\longrightarrow{\mathscr{C}}^{\alpha}(\partial\Omega);
\end{equation}
\item[{\rm (iii)}] $\Omega$ is an Ahlfors regular domain and
\begin{equation}\label{eq:RIE-yttr.1bb}
R_j1\in{\mathscr{C}}^{\alpha}(\partial\Omega)\,\,\text{ for each }\,\,j\in\{1,\dots,n\};
\end{equation}
\item[{\rm (iv)}] $\Omega$ is a {\rm UR} domain and, for each $j\in\{1,\dots,n\}$
the principal value Riesz transform $R^{{}^{\rm pv}}_j$ defined as in \eqref{T-pv.4khg}
with $\Sigma:=\partial\Omega$ induces a linear and bounded operator in the context
\begin{equation}\label{eq:RIE-yttr.2}
R^{{}^{\rm pv}}_j:{\mathscr{C}}^{\alpha}(\partial\Omega)\longrightarrow{\mathscr{C}}^{\alpha}(\partial\Omega).
\end{equation}
\item[{\rm (v)}] $\Omega$ is a {\rm UR} domain and
\begin{equation}\label{eq:RIE-yttr.2bb}
R^{{}^{\rm pv}}_j1\in{\mathscr{C}}^{\alpha}(\partial\Omega)\,\,\text{ for each }\,\,j\in\{1,\dots,n\}.
\end{equation}
\end{enumerate}
\end{corollary}

In dimension two, there is a variant of Theorem~\ref{Main-T1aa}
starting from the demand that the boundary of the domain in question
is an upper Ahlfors regular Jordan curve and, in lieu of the Riesz
transforms, using the following version of the classical Cauchy
integral operator in the principal value sense:
\begin{equation}\label{eq:Cauchy33}
{\mathfrak{C}}^{{}^{\rm pv}}\!f(z):=\lim_{\varepsilon\to 0^{+}}\frac{1}{2\pi i}
\int\limits_{\stackrel{\zeta\in\partial\Omega}{|z-\zeta|>\varepsilon}}
\frac{f(\zeta)}{\zeta-z}\,d{\mathcal{H}}^1(\zeta),\qquad
z\in\partial\Omega.
\end{equation}

\begin{theorem}\label{THM-main222}
Let $\Omega\subseteq{\mathbb{C}}$ be a bounded open set whose
boundary is an upper Ahlfors regular Jordan curve and fix
$\alpha\in(0,1)$. Then $\Omega$ is a domain of class
${\mathscr{C}}^{1+\alpha}$ if and only if the operator
\eqref{eq:Cauchy33} satisfies
${\mathfrak{C}}^{{}^{\rm pv}}1\in{\mathscr{C}}^\alpha(\partial\Omega)$.
\end{theorem}

Under the initial background hypotheses on $\Omega$ made in
Theorem~\ref{Main-T1aa}, $\Omega$ being a ${\mathscr{C}}^1$ domain
is equivalent with
$\nu\in{\mathscr{C}}^0(\partial\Omega)$ (cf.
\cite{HoMiTa07} in this regard). This being said, the limiting case
$\alpha=0$ of the equivalence {\rm (a)}$\Leftrightarrow${\rm (b)} in
Theorem~\ref{Main-T1aa} requires replacing the space of continuous
functions by the (larger) Sarason space ${\rm VMO}$, of functions of
vanishing mean oscillations (on $\partial\Omega$, viewed as a space
of homogeneous type, in the sense of Coifman-Weiss, when equipped
with the measure $\sigma$ and the Euclidean distance). Specifically,
the following result holds.

\begin{theorem}\label{Kbagg.A}
Let $\Omega\subseteq\mathbb{R}^n$ be an Ahlfors regular domain with a compact boundary,
and denote by $\nu$ the geometric measure theoretic outward unit normal to $\Omega$.
Then
\begin{equation}\label{eq:iugf.6r4}
\left.
\begin{array}{r}
\nu\in{\rm VMO}(\partial\Omega)\,\,\mbox{ and}
\\[4pt]
\mbox{$\partial\Omega$ uniformly rectifiable set}
\end{array}
\right\} \Longleftrightarrow R_j1\in{\rm
VMO}(\partial\Omega)\,\,\mbox{ for all }\,\,j\in\{1,\dots,n\}.
\end{equation}
\end{theorem}

Equivalence \eqref{eq:iugf.6r4} should be contrasted with
\eqref{Mabb88}. In the present context, the additional background
assumption ${\mathcal{H}}^{n-1}(\partial\Omega\setminus\partial_\ast\Omega)=0$
(which is part of the definition of an Ahlfors regular domain; cf. Definition~\ref{ADRDOM})
merely ensures that the geometric measure theoretic outward unit
normal $\nu$ to $\Omega$ is well-defined $\sigma$-a.e. on
$\partial\Omega$.

The collection of all geometric conditions entering
Theorem~\ref{Kbagg.A}, i.e., that $\Omega\subseteq\mathbb{R}^n$ is
an Ahlfors regular domain such that $\partial\Omega$ is a uniformly rectifiable set,
amounts to saying that $\Omega$ is a {\rm UR} domain (cf.
Definition~\ref{Def-UB}). Concerning this class of domains, it has
been noted in \cite[Corollary~3.9, p.\,2633]{HoMiTa10} that
\begin{equation}\label{eq:rTG866}
\parbox{9.80cm}
{if $\Omega\subset{\mathbb{R}}^n$ is an open set satisfying a
two-sided corkscrew condition (in the sense of Jerison-Kenig
\cite{JeKe82}) and whose boundary is Ahlfors regular, then $\Omega$
is a {\rm UR} domain.}
\end{equation}

In fact, the same circle of techniques yielding
Theorem~\ref{Kbagg.A} also allows us to characterize the class of
regular {\rm SKT} domains, originally introduced in
\cite[Definition~4.8, p.\,2690]{HoMiTa10} by demanding:
$\delta$-Reifenberg flatness for some sufficiently small $\delta>0$
(cf. Definition~\ref{Def-R4}), Ahlfors regular boundary, and
vanishing mean oscillations for the geometric measure theoretic
outward unit normal. Specifically, combining \eqref{eq:rTG866},
Theorem~\ref{Kbagg.A}, Theorem~\ref{KbAAbb}, and \cite[Theorem~4.21,
p.\,2711]{HoMiTa10} gives the following theorem.

\begin{theorem}\label{yr46bp}
If $\Omega\subseteq\mathbb{R}^n$ is an open set with a compact
Ahlfors regular boundary, satisfying a two-sided John condition as
described in Definition~\ref{Def-John} {\rm (}which, in particular,
implies the two-sided corkscrew condition{\rm )} then
\begin{equation}\label{eq:rDC}
\begin{array}{c}
\mbox{$R_j1\in{\rm VMO}(\partial\Omega)$ for every
$j\in\{1,\dots,n\}$}
\\[4pt]
\mbox{if and only if $\Omega$ is a regular {\rm SKT} domain}.
\end{array}
\end{equation}
\end{theorem}

It turns out that the equivalence {\rm (a)}$\Leftrightarrow${\rm
(b)} in Theorem~\ref{Main-T1aa} essentially self-extends to the
larger scale of Besov spaces $B^{p,p}_s(\partial\Omega)$ with
$p\in[1,\infty]$ and $s\in(0,1)$ satisfying $sp>n-1$, for which the
H\"older spaces occur as a special, limiting case, corresponding to
$p=\infty$. For a precise statement, see Theorem~\ref{Thm-Besov}.

The category of singular integral operators falling under the scope
of Theorem~\ref{Main-T1aa} already includes boundary layer
potentials associated with basic PDE's of mathematical physics, such
as the Laplacian, the Helmholtz operator, the Lam\'e system, the
Stokes system, and even higher-order elliptic systems (cf., e.g.,
\cite{CK}, \cite{HW}, \cite{DM}, \cite{MM}). This being said,
granted the estimates established in the last part of
Theorem~\ref{Main-T1aa}, the method of spherical harmonics then
allows us to prove the following result, dealing with a more general
class of operators.

\begin{theorem}\label{Main-T1aBBB}
Let $\Omega$ be a ${\mathscr{C}}^{1+\alpha}$ domain,
$\alpha\in(0,1)$, with compact boundary, and let
$k\in{\mathscr{C}}^\infty({\mathbb{R}}^n\setminus\{0\})$ be an odd
function satisfying $k(\lambda x)=\lambda^{1-n}k(x)$ for all
$\lambda\in(0,\infty)$ and $x\in{\mathbb{R}}^{n}\setminus\{0\}$. In
addition, assume that there exists a sequence
$\{m_l\}_{l\in{\mathbb{N}}_0}\subseteq{\mathbb{N}}_0$ for which
\begin{equation}\label{eq:Reac}
\sum_{l=0}^\infty 4^{l^2}l^{-2m_l}
\big\|(\Delta_{S^{n-1}})^{m_l}\big(k\big|_{S^{n-1}}\big)\big\|_{L^2(S^{n-1})}
<+\infty,
\end{equation}
where $\Delta_{S^{n-1}}$ is the Laplace-Beltrami operator on the
unit sphere $S^{n-1}$ in ${\mathbb{R}}^n$.

Then the singular integral operators
\begin{align}\label{T-layer.44-BaB}
&
{\mathbb{T}}f(x):=\int_{\partial\Omega}k(x-y)f(y)\,d\sigma(y),\qquad
x\in\Omega,
\\[8pt]
&Tf(x):=\lim_{\varepsilon\to 0^{+}}
\int\limits_{\stackrel{y\in\partial\Omega}{|x-y|>\varepsilon}}
k(x-y)f(y)\,d\sigma(y),\qquad x\in\partial\Omega,
\label{T-pv.44-BaB}
\end{align}
induce linear and bounded mappings
\begin{align}\label{maKP-BaB}
{\mathbb{T}}:{\mathscr{C}}^\alpha(\partial\Omega)\longrightarrow
{\mathscr{C}}^\alpha\big(\overline{\Omega}\,\big),\qquad
T:{\mathscr{C}}^\alpha(\partial\Omega)\longrightarrow{\mathscr{C}}^\alpha(\partial\Omega).
\end{align}
\end{theorem}

We wish to note that Theorem~\ref{Main-T1aBBB} refines the
implication {\rm (a)}$\,\Rightarrow${\rm (e)} in
Theorem~\ref{Main-T1aa} since, as explained in Remark~\ref{yafv},
condition \eqref{eq:Reac} is satisfied whenever the kernel $k$ is of
the form $P(x)/|x|^{n-1+l}$ for some homogeneous polynomial $P$ of
degree $l\in 2{\mathbb{N}}-1$ in ${\mathbb{R}}^n$. In fact,
condition \eqref{eq:Reac} holds for kernels $k$ that are
real-analytic away from $0$ with lacunary Taylor series (involving
sufficiently large gaps between the non-zero coefficients of their
expansions, depending on $n$, $\alpha$, ${\rm
diam}\,(\partial\Omega)$,
$\|\nu\|_{{\mathscr{C}}^\alpha(\partial\Omega)}$, and the upper
Ahlfors regularity constant of $\partial\Omega$). Thus, the
conclusions in Theorem~\ref{Main-T1aBBB} are valid for such kernels
which are also odd and positive homogeneous of degree $1-n$.

Even though the statement does not reflect it, the proof of
Theorem~\ref{Main-T1aa} makes essential use of the Clifford algebra
${\mathcal{C}}\!\ell_{n}$, a highly non-commutative generalization
of the field of complex numbers to $n$-dimensions, which also turns
out to be geometrically sensitive. Indeed, this is a tool which has
occasionally emerged at the core of a variety of problems at the
interface between geometry and analysis. For us, one key aspect of
this algebraic setting is the close relationship between the Riesz
transforms and the principal value\footnote{in the standard sense of
removing balls centered at the singularity and taking the limit as
the radii shrink to zero} Cauchy-Clifford integral operator
${\mathcal{C}}^{{}^{\rm pv}}$ (defined in \eqref{Cau-C2}). For the purpose of this
introduction we single out the remarkable formula
\begin{equation}\label{eq:Csq22.ABC}
\nu=-4{\mathcal{C}}^{{}^{\rm pv}}\!\Big(\sum_{j=1}^n\big(R^{{}^{\rm pv}}_j1\big)e_j\Big)
\,\,\mbox{ at $\sigma$-a.e. point on }\,\,\partial\Omega
\end{equation}
expressing the (geometric measure theoretic) outward unit normal to
$\Omega$ as the Clifford algebra cocktail $\sum_{j=1}^n
\big(R^{{}^{\rm pv}}_j1\big)e_j$ of principal value Riesz transforms acting
on the constant function $1$, coupled with the imaginary units $e_j$ in ${\mathcal{C}}\!\ell_{n}$,
then finally distorted through the action of the Cauchy-Clifford operator ${\mathcal{C}}^{{}^{\rm pv}}\!$.
Identity \eqref{eq:Csq22.ABC} plays a basic role in the proof of
{\rm (b)}\,$\Rightarrow$\,{\rm (a)} in Theorem~\ref{Main-T1aa},
together with a higher-dimensional generalization in a rough setting
of the classical Plemelj-Privalov theorem stating that the principal
value Cauchy integral operator on a piecewise smooth Jordan curve
without cusps in the plane is bounded on H\"older spaces (cf.
\cite{Plem}, \cite{Priv1}, \cite{Priv2}; cf. also \cite{Iftimie} for
a higher dimensional version for Lyapunov domains with compact
boundaries). Specifically, in Theorem~\ref{i65r5ED} we show that
whenever $\Omega\subset{\mathbb{R}}^{n}$ is a Lebesgue measurable
set whose boundary is compact, upper Ahlfors regular, and satisfies
${\mathcal{H}}^{n-1}(\partial\Omega\setminus\partial_\ast\Omega)=0$,
it follows that for each $\alpha\in(0,1)$ the principal value
Cauchy-Clifford operator ${\mathcal{C}}^{{}^{\rm pv}}\!$ induces a well-defined, linear,
and bounded mapping
\begin{equation}\label{Rd-8655}
{\mathcal{C}}^{{}^{\rm pv}}\!:{\mathscr{C}}^{\alpha}(\partial\Omega)\otimes{\mathcal{C}}\!\ell_{n}
\longrightarrow{\mathscr{C}}^{\alpha}(\partial\Omega)\otimes{\mathcal{C}}\!\ell_{n}.
\end{equation}
The strategy employed in the proof of the implication
{\rm (a)}\,$\Rightarrow$\,{\rm (e)} in Theorem~\ref{Main-T1aa}
is somewhat akin to that of establishing a ``$T(1)$-theorem" in the
sense that matters are reduced to checking that ${\mathbb{T}}_{\pm}$
act reasonably on the constant function $1$ (see \eqref{tAAcc.44iii}
in this regard). In turn, this is accomplished via a proof by induction
on $l\in 2{\mathbb{N}}-1$, the degree of the homogeneous polynomial $P$.
The base case $l=1$, corresponding to linear combinations of polynomials
as in \eqref{eq:Pjah}, is dealt with by viewing $(x_j-y_j)/|x-y|^n$
as a dimensional multiple of $\partial_j E_{\!{}_\Delta}\!(x-y)$
where $E_{\!{}_\Delta}$ is the standard fundamental solution for the
Laplacian in ${\mathbb{R}}^n$. As such, the key cancellation
property that eventually allows us to establish the desired H\"older
estimate in this base case may be ultimately traced back to the PDE
satisfied by $(x_j-y_j)/|x-y|^n$. In carrying out the inductive step
we make essential use of elements of Clifford analysis permitting us
to relate ${\mathbb{T}}_\pm 1$ to the action of certain integral
operators constructed as in \eqref{T-layer.44} but relative to lower
degree polynomials acting on components of the outward unit normal
$\nu$ to $\Omega$. In this scenario, what allows the use of the
induction hypothesis is the fact that, since $\Omega$ is a domain of
class ${\mathscr{C}}^{1+\alpha}$, the said components belong to
${\mathscr{C}}^{\alpha}(\partial\Omega)$.

The layout of the paper is as follows. Section~\ref{S-2} contains a
discussion of background material of geometric measure theoretic
nature, along with some auxiliary lemmas which are relevant in our
future endeavors. In Section~\ref{S-3} we first recall a version of
the Calder\'on-Zygmund theory for singular integral operators on
Lebesgue spaces in {\rm UR} domains, and then proceed to establish
several useful preliminary estimates for general singular integral
operators. Next, Section~\ref{S-4} is reserved for a presentation of
those aspects of Clifford analysis which are relevant for the
present work. Section~\ref{S-5} is devoted to a study of
Cauchy-Clifford integral operators (both of boundary-to-domain and
boundary-to-boundary type) in the context of H\"older spaces. In
contrast with the Calder\'on-Zygmund theory for singular integrals
in {\rm UR} domains reviewed in the first part of Section~\ref{S-3},
the novelty here is the consideration of a much larger category of
domains (see Theorem~\ref{i65r5ED} for details). In the last part of
Section~\ref{S-5} we also discuss the harmonic single and double
layer potentials (involved in the initial induction step in the
proof of the implication {\rm (a)}\,$\Rightarrow$\,{\rm (e)} in
Theorem~\ref{Main-T1aa}). Finally, in Section~\ref{S-6}, the proofs
of Theorems~\ref{Main-T1aa}, \ref{THM-main222}, \ref{Main-T1aBBB}
are presented, while Section~\ref{S-7} contains the proofs of
Theorem~\ref{Kbagg.A}, and the Besov space version of the
equivalence {\rm (a)}$\Leftrightarrow${\rm (b)} in
Theorem~\ref{Main-T1aa} (see Theorem~\ref{Thm-Besov}), and
also a more general version of \eqref{eq:rDC} in
Theorem~\ref{KbAAbb}.

\vskip 0.08in
\noindent{\it Acknowledgments:}
The first author has been supported in part by the Simons Foundation
grant $\#\,$200750, the second author has been supported in part by the
Simons Foundation grant $\#\,$281566 and by a University of Missouri Research Leave grant,
while the third author has been supported in part by the grants 2009SGR420 (Generalitat de Catalunya)
and MTM2010-15657 (Ministerio de Educaci\'on y Ciencia). The authors are also grateful
to L. Escauriaza and M. Taylor for some useful correspondence on the subject of the paper.

\section{Geometric Measure Theoretic Preliminaries}
\setcounter{equation}{0} \label{S-2}

Throughout, ${\mathbb{N}}_0:={\mathbb{N}}\cup\{0\}$ and we shall
denote by ${\mbox{\bf 1}}_E$ the characteristic function of a set
$E$. For $\alpha\in(0,1)$ and $U\subseteq\mathbb{R}^n$ arbitrary set
(implicitly assumed to have cardinality $\geq 2$), define the {\tt
homogeneous} {\tt H\"older} {\tt space} {\tt of} {\tt order}
$\alpha$ on $U$ as
\begin{equation}\label{eqn.HolderSpace1}
\dot{\mathscr{C}}^\alpha(U):=\big\{u:U\to\mathbb{C}:\,[u]_{\dot{\mathscr{C}}^\alpha(U)}<+\infty\big\},
\end{equation}
where $[\,\cdot\,]_{\dot{\mathscr{C}}^\alpha(U)}$ stands for the
seminorm
\begin{equation}\label{lBba}
[u]_{\dot{\mathscr{C}}^\alpha(U)}:=\sup_{\substack{x,y\in U \\ x\neq
y}}\frac{|u(x)-u(y)|}{|x-y|^\alpha}.
\end{equation}
The {\tt inhomogeneous} {\tt H\"older} {\tt space} {\tt of} {\tt
order} $\alpha$ on $U$ is then defined as
\begin{equation}\label{eqn.HolderSpace2}
\mathscr{C}^\alpha(U):=\big\{u\in\dot{\mathscr{C}}^\alpha(U):\,u\mbox{
is bounded in $U$}\big\},
\end{equation}
and is equipped with the norm
\begin{equation}\label{eqn.HolderNNN}
\|u\|_{\mathscr{C}^\alpha(U)}:=\sup_U
|u|+[u]_{\dot{\mathscr{C}}^\alpha(U)},\qquad
\forall\,u\in\mathscr{C}^\alpha(U).
\end{equation}
Also, denote by $\mathscr{C}^\alpha_c(U)$ the subspace of ${\mathscr{C}}^\alpha(U)$ consisting
of functions vanishing outside of a relatively compact subset of $U$.
Moreover, if ${\mathcal{O}}$ is an open, nonempty, subset of
${\mathbb{R}}^n$, then for $\alpha\in(0,1)$ given define
\begin{equation}\label{eMnan}
{\mathscr{C}}^{1+\alpha}({\mathcal{O}}):=\big\{u\in{\mathscr{C}}^1({\mathcal{O}}):\,
\|u\|_{{\mathscr{C}}^{1+\alpha}({\mathcal{O}})}<+\infty\big\},
\end{equation}
where
\begin{equation}\label{eMnan.2}
\|u\|_{{\mathscr{C}}^{1+\alpha}({\mathcal{O}})}
:=\sup_{x\in\mathcal{O}}\,|u(x)|+\sup_{x\in\mathcal{O}}\,|(\nabla
u)(x)| +\sup_{\substack{x,y\in{\mathcal{O}} \\ x\neq y}}
\frac{\big|(\nabla u)(x)-(\nabla u)(y)\big|}{|x-y|^\alpha}.
\end{equation}

The following observations will be tacitly used in the sequel. For
each set $U\subseteq\mathbb{R}^n$ and any $\alpha\in(0,1)$, we have
that ${\mathscr{C}}^\alpha(U)$ is an algebra, the spaces
$\dot{\mathscr{C}}^\alpha(U)$ and $\mathscr{C}^\alpha(U)$ are
contained in the space of uniformly continuous functions on $U$, and
$\dot{\mathscr{C}}^\alpha(U)=\dot{\mathscr{C}}^\alpha(\overline{U})$,
$\mathscr{C}^\alpha(U)=\mathscr{C}^\alpha(\overline{U})$. Moreover,
$\dot{\mathscr{C}}^\alpha(U)=\mathscr{C}^\alpha(U)$ if $U$ is
bounded. Finally, we shall make no notational distinction between
a H\"older space of scalar functions and its version involving
vector-valued functions. A similar convention is employed for other
function spaces used in this work.

\begin{definition}\label{lipdom}
A nonempty, open, proper subset $\Omega$ of ${\mathbb{R}}^n$ is
called a {\tt domain} {\tt of} {\tt class}
${\mathscr{C}}^{1+\alpha}$ for some $\alpha\in(0,1)$ {\rm (}or a
{\tt Lyapunov} {\tt domain} {\tt of} {\tt order} $\alpha${\rm )}, if
there exist $r,\,h>0$ with the following significance. For every
point $x_0\in\partial\Omega$ one can find a coordinate system
$(x_1,\dots,x_n)=(x',x_n)$ in ${\mathbb{R}}^n$ which is isometric to
the canonical one and has origin at $x_0$, along with a real-valued
function $\varphi\in{\mathscr{C}}^{1+\alpha}({\mathbb{R}}^{n-1})$
such that
\begin{eqnarray}\label{cylinders}
\Omega\cap {\mathcal{C}}(r,h)=\bigl\{x=(x',x_n)\in{\mathbb{R}}^{n-1}
\times{\mathbb{R}}:\,|x'|<r\,\,\,\mbox{and}\,\,\varphi(x')<x_n<h\bigr\},
\end{eqnarray}
where ${\mathcal{C}}(r,h)$ stands for the cylinder
\begin{equation}\label{eq:123t6}
\bigl\{x=(x',x_n)\in{\mathbb{R}}^{n-1}\times{\mathbb{R}}:\,
|x'|<r\mbox{ and }-h<x_n<h\bigr\}.
\end{equation}
\end{definition}

Strictly speaking, the traditional definition of a
Lyapunov\footnote{Also spelled as Liapunov} domain
$\Omega\subseteq{\mathbb{R}}^n$ of order $\alpha$ requires that
$\partial\Omega$ is locally given by the graph of a differentiable
function $\varphi:{\mathbb{R}}^{n-1}\to{\mathbb{R}}$ whose normal
$\nu$ to its graph $\Sigma$ has the property that the acute angle
$\theta_{x,\,y}$ between $\nu(x)$ and $\nu(y)$ for two arbitrary
points $x,y\in\Sigma$ satisfies $\theta_{x,\,y}\leq C|x-y|^\alpha$;
see, e.g., \cite[D\'efinition~2.1, p.\,301]{Iftimie}. This being
said, it is easy to see that the latter condition implies that $\nu$
is H\"older continuous of order $\alpha$ and, ultimately, that
$\Omega$ is a domain of class ${\mathscr{C}}^{1+\alpha}$ in the
sense of our Definition~\ref{lipdom}.

We shall now present a brief summary of a number of definitions and
results from geometric measure theory which are relevant for the
current work (cf. H.\,Federer \cite{Fed96}, W.\,Ziemer \cite{Zi89},
L.\,Evans and R.\,Gariepy \cite{EvGa92} for more details). Call a
Lebesgue measurable set $\Omega\subset{\mathbb{R}}^n$ of
{\tt locally} {\tt finite} {\tt perimeter} provided
$\nabla{\bf 1}_\Omega$ is a locally finite Borel regular
${\mathbb{R}}^n$-valued
measure. Given a Lebesgue measurable set $\Omega\subset{\mathbb{R}}^n$
of locally finite perimeter we denote by $\sigma$ the total variation
measure of $\nabla{\bf 1}_\Omega$. Then $\sigma$ is a locally finite
positive measure, supported on $\partial\Omega$. Clearly, each
component of $\nabla{\bf 1}_\Omega$ is absolutely continuous with
respect to $\sigma$ so from the Radon-Nikodym theorem it follows
that
\begin{eqnarray}\label{hmttr-3}
\nabla{\mbox{\bf 1}}_\Omega=-\nu\sigma,
\end{eqnarray}
where
\begin{eqnarray}\label{Nu-Dj}
\parbox{9.80cm}{$\nu$ is an ${\mathbb{R}}^n$-valued function with components
in $L^\infty(\partial\Omega,\sigma)$ and which satisfies $|\nu(x)|=1$ at $\sigma$-a.e. point
$x\in\partial\Omega$.}
\end{eqnarray}
Above and elsewhere, by $L^p(\partial\Omega,\sigma)$,
$0<p\leq\infty$, we denote the usual scale of Lebesgue spaces on
$\partial\Omega$ (with respect to the measure $\sigma$). In the
sequel we shall frequently identify $\sigma$ with its restriction to
$\partial\Omega$, with no special mention. We shall refer to $\nu$
and $\sigma$, respectively, as the (geometric measure theoretic)
{\tt outward} {\tt unit} {\tt normal} and the {\tt surface} {\tt
measure} on $\partial\Omega$.

Next, denote by ${\mathscr{L}}^n$ the Lebesgue measure in
${\mathbb{R}}^n$ and recall that the {\tt measure}-{\tt theoretic}
{\tt boundary} $\partial_*\Omega$ of a Lebesgue measurable set
$\Omega\subseteq{\mathbb{R}}^n$ is defined by
\begin{align}\label{2.1.10}
\partial_*\Omega:=\Bigl\{x\in\partial\Omega:\,
\limsup\limits_{r\rightarrow
0^{+}}\,&\frac{{\mathscr{L}}^n(B(x,r)\cap\Omega)}{r^n}>0
\nonumber\\[6pt]
&\mbox{and }\,\,\, \limsup\limits_{r\rightarrow
0^{+}}\,\frac{{\mathscr{L}}^n(B(x,r)\setminus\Omega)}{r^n}>0\Bigr\}.
\end{align}
Also, the {\tt reduced} {\tt boundary} $\partial^*\Omega$ of
$\Omega$ is defined as
\begin{eqnarray}\label{2S-red}
\partial^*\Omega:=\bigl\{x\in\partial\Omega:\,|\nu(x)|=1\bigr\}.
\end{eqnarray}
As is well-known, (cf. \cite[Lemma~5.9.5 on p.\,252]{Zi89} and
\cite[p.\,208]{EvGa92}) one has
\begin{eqnarray}\label{2.1.11}
\partial^*\Omega\subseteq\partial_*\Omega\subseteq\partial\Omega
\quad\mbox{ and }\quad
\mathcal{H}^{n-1}(\partial_*\Omega\setminus\partial^*\Omega)=0,
\end{eqnarray}
where ${\mathcal{H}}^{n-1}$ is the $(n-1)$-dimensional Hausdorff
measure in ${\mathbb{R}}^n$. Also,
\begin{eqnarray}\label{2.1.5}
\sigma=\mathcal{H}^{n-1}\lfloor\partial^*\Omega.
\end{eqnarray}
Hence, if $\Omega$ has locally finite perimeter, it follows from
\eqref{2.1.11} that the outward unit normal is defined $\sigma$-a.e.
on $\partial_*\Omega$. In particular, if
\begin{eqnarray}\label{Tay-1}
{\mathcal{H}}^{n-1}(\partial\Omega\setminus\partial_*\Omega)=0,
\end{eqnarray}
then from \eqref{2S-red}-\eqref{2.1.11} we see that the outward unit
normal $\nu$ is defined $\sigma$-a.e.~on $\partial\Omega$, and
\eqref{2.1.5} becomes
$\sigma={\mathcal{H}}^{n-1}\lfloor\partial\Omega$. Works of Federer
and of De Giorgi also give that
\begin{eqnarray}\label{2.CRpar}
\mbox{$\partial^*\Omega$ is {\tt countably} {\tt rectifiable} (of
dimension $n-1$)},
\end{eqnarray}
in the sense that it is a countable disjoint union
\begin{eqnarray}\label{2.1.6}
\partial^*\Omega=N\cup\Big(\bigcup\limits_{k\in{\mathbb{N}}}M_k\Big),
\end{eqnarray}
where each $M_k$ is a compact subset of an $(n-1)$-dimensional
${\mathscr{C}}^1$ surface in ${\mathbb{R}}^n$ and
$\mathcal{H}^{n-1}(N)=0$. It then happens that $\nu$ is normal to
each such surface, in the usual sense. For further reference let us
remark here that, as is apparent from \eqref{2.CRpar},
\eqref{2.1.11}, and \eqref{2.1.6},
\begin{eqnarray}\label{TaGva134}
\parbox{8.60cm}{if $\Omega\subset{\mathbb{R}}^n$ is a Lebesgue measurable
set which has locally finite perimeter and for which \eqref{Tay-1} holds,
then $\partial\Omega$ is countably rectifiable (of dimension $n-1$).}
\end{eqnarray}

The following characterization of the class of
${\mathscr{C}}^{1+\alpha}$ domains from \cite{HoMiTa07} is going to
play an important role for us here.

\begin{theorem}\label{Th-C1}
Assume that $\Omega$ is a nonempty, open, proper subset of
${\mathbb{R}}^n$, of locally finite perimeter, with compact
boundary, for which
\begin{eqnarray}\label{hmttr-1}
\partial\Omega=\partial\big(\,\overline{\Omega}\,\big),
\end{eqnarray}
and denote by $\nu$ the geometric measure theoretic outward unit
normal to $\partial\Omega$, as defined in
\eqref{hmttr-3}-\eqref{Nu-Dj}. Also, fix $\alpha\in(0,1)$. Then
$\Omega$ is a ${\mathscr{C}}^{1+\alpha}$ domain if and only if,
after altering $\nu$ on a set of $\sigma$-measure zero, one has
$\nu\in{\mathscr{C}}^\alpha(\partial\Omega)$.
\end{theorem}

\noindent Condition \eqref{hmttr-1} expresses the fact that the
domain $\Omega$ sits on just one side of its topological boundary,
and is designed to preclude pathological happenstances such as a
slit disk. By the Jordan-Brower separation theorem (cf.
\cite[Theorem~1, p.\,284]{Alexander}), \eqref{hmttr-1} is
automatically satisfied if $\partial\Omega$ is a compact, connected,
$(n-1)$-dimensional topological manifold without boundary (since in
this scenario ${\mathbb{R}}^n\setminus\partial\Omega$ consists
precisely of two components, each with boundary $\partial\Omega$;
see \cite{ABMMZ} for details).

Changing topics, we remind the reader that a set
$\Sigma\subset{\mathbb{R}}^n$ is called {\tt Ahlfors} {\tt regular}
provided it is closed, nonempty, and there exists $C\in(1,\infty)$
such that
\begin{eqnarray}\label{2.0.1}
C^{-1}\,r^{n-1}\leq{\mathcal{H}}^{n-1}\bigl(B(x,r)\cap\Sigma\bigr)\leq
C\,r^{n-1},
\end{eqnarray}
for each $x\in\Sigma$ and $r\in (0,{\rm diam}\,\Sigma)$. When
considered by itself, the second inequality above will be referred
to as {\tt upper} {\tt Ahlfors} {\tt regularity}. In this vein, we
wish to remark that (cf. \cite[Theorem~1, p.\,222]{EvGa92})
\begin{equation}\label{eq:hdusd}
\parbox{9.00cm}{any Lebesgue measurable subset of ${\mathbb{R}}^n$ with an
upper Ahlfors regular boundary is of locally finite perimeter.}
\end{equation}
It is natural to make the following definition.

\begin{definition}\label{ADRDOM}
Call an open, nonempty, proper subset $\Omega$ of ${\mathbb{R}}^n$ an
{\tt Ahlfors} {\tt regular} {\tt domain} provided $\partial\Omega$ is an Ahlfors regular
set and ${\mathcal{H}}^{n-1}(\partial\Omega\setminus\partial_*\Omega)=0$.
\end{definition}

Let us remark here that \eqref{TaGva134} and \eqref{eq:hdusd}
imply the following result:
\begin{eqnarray}\label{TaGv-kgggg}
\parbox{9.80cm}{if $\Omega\subset{\mathbb{R}}^n$ is a Lebesgue measurable set
with an upper Ahlfors regular boundary satisfying
${\mathcal{H}}^{n-1}(\partial\Omega\setminus\partial_*\Omega)=0$, then $\Omega$
is a set of locally finite perimeter and its topological boundary,
$\partial\Omega$, is countably rectifiable (of dimension $n-1$).}
\end{eqnarray}
For further use, we record the following consequence of \eqref{TaGv-kgggg} and Definition~\ref{ADRDOM}:
\begin{eqnarray}\label{TaGv-kgggg.222}
\parbox{8.70cm}{any Ahlfors regular domain in ${\mathbb{R}}^n$ has
a countably rectifiable topological boundary (of dimension $n-1$).}
\end{eqnarray}
Later on, the following result is going to be of significance to us.

\begin{proposition}\label{Yfe488}
Let $\Sigma\subseteq{\mathbb{R}}^n$ be an Ahlfors regular set which
is countably rectifiable {\rm (}of dimension $n-1${\rm )}. Define
$\sigma:={\mathcal{H}}^{n-1}\lfloor\Sigma$ and consider an arbitrary
function $f\in L^1_{\rm loc}(\Sigma,\sigma)$. Then for each
$j\in\{1,\dots,n\}$,
\begin{align}\label{Tdad-65}
\lim_{\varepsilon\to
0^+}\left\{\sup_{r\in(\varepsilon/2,\,\varepsilon)}
\Bigg|\int\limits_{\stackrel{y\in\Sigma}{\varepsilon/4<|y-x|\leq r}}
\frac{x_j-y_j}{|x-y|^{n}}f(y)\,d\sigma(y)\Bigg|\right\}=0\,\,\mbox{
for $\sigma$-a.e. $x\in\Sigma$}.
\end{align}
\end{proposition}

\begin{proof}
Fix $j\in\{1,\dots,n\}$ and pick some large $R>0$. For each
$\varepsilon\in(0,1)$, $r\in(\varepsilon/2,\varepsilon)$, and
$x\in\Sigma\cap B(0,R)$ split
\begin{align}\label{Tdad-65.A}
\int\limits_{\stackrel{y\in\Sigma}{\varepsilon/4<|y-x|\leq r}}
\frac{x_j-y_j}{|x-y|^{n}} &
f(y)\,d\sigma(y)=I_{\varepsilon,\,r}+II_{\varepsilon,\,r}
\end{align}
where
\begin{align}\label{Tdad-65.B}
I_{\varepsilon,\,r}
&:=\int\limits_{\stackrel{y\in\Sigma}{\varepsilon/4<|y-x|\leq r}}
\frac{x_j-y_j}{|x-y|^{n}}[f(y)-f(x)]\,d\sigma(y),
\\[4pt]
II_{\varepsilon,\,r}
&:=f(x)\Bigg\{\int\limits_{\stackrel{y\in\Sigma\cap
B(0,R+1)}{\varepsilon/4<|y-x|<1}}
\frac{x_j-y_j}{|x-y|^{n}}\,d\sigma(y)
-\int\limits_{\stackrel{y\in\Sigma\cap B(0,R+1)}{r<|y-x|<1}}
\frac{x_j-y_j}{|x-y|^{n}}\,d\sigma(y)\Bigg\}. \label{Tdad-65.C}
\end{align}
The left-to-right implication in \eqref{TToo-1}, used for the set
$\Sigma\cap B(0,R+1)$, gives that $\sigma$-a.e. point
$x\in\Sigma\cap B(0,R)$ has the property that for each $\delta>0$
there exists $\theta_\delta\in(0,1)$ such that for each
$\theta_1,\theta_2\in(0,\theta_\delta)$ we have
\begin{align}\label{Tdad-65.mabf}
\Bigg|\,\,\int\limits_{\stackrel{y\in\Sigma\cap
B(0,R+1)}{\theta_1<|y-x|<1}} \frac{x_j-y_j}{|x-y|^{n}}\,d\sigma(y)
-\int\limits_{\stackrel{y\in\Sigma\cap B(0,R+1)}{\theta_2<|y-x|<1}}
\frac{x_j-y_j}{|x-y|^{n}}\,d\sigma(y)\Bigg|<\delta.
\end{align}
In turn, this readily yields
\begin{align}\label{Tdad-65.G}
\lim_{\varepsilon\to
0^+}\left\{\sup_{r\in(\varepsilon/2,\varepsilon)}
\big|II_{\varepsilon,\,r}\big|\right\}=0\,\,\mbox{ for $\sigma$-a.e.
$x\in\Sigma\cap B(0,R)$}.
\end{align}
Next, thanks to the upper Ahlfors regularity condition satisfied by
$\Sigma$, we may estimate
\begin{align}\label{Tdad-65.D}
\big|I_{\varepsilon,\,r}\big|
&\leq\Big(\frac{4}{\varepsilon}\Big)^{n-1} \int_{\Sigma\cap
B(x,\varepsilon)}|f(y)-f(x)|\,d\sigma(y) \leq c\meanint_{\Sigma\cap
B(x,\varepsilon)}|f(y)-f(x)|\,d\sigma(y),
\end{align}
where the barred integral indicates mean average. Hence, on the one
hand,
\begin{align}\label{Tdad-65.E}
\lim_{\varepsilon\to
0^+}\left\{\sup_{r\in(\varepsilon/2,\varepsilon)}
\big|I_{\varepsilon,\,r}\big|\right\}=0\,\,\mbox{ if $x$ is a
Lebesgue point for $f$}.
\end{align}
On the other hand, the triplet $(\Sigma,|\cdot-\cdot|,\sigma)$ is a
space of homogeneous type and the underlying measure is Borel
regular. As such, Lebesgue's Differentiation Theorem gives that
$\sigma$-a.e. point in $\Sigma$ is a Lebesgue point for $f$. Bearing
this in mind, the desired conclusion now follows from
\eqref{Tdad-65.A}, \eqref{Tdad-65.G}, and \eqref{Tdad-65.E}.
\end{proof}

In the treatment of the principal value Cauchy-Clifford integral
operator in \S\ref{S-5}, the following lemma plays a significant
role.

\begin{lemma}\label{rtggb}
Let $\Omega\subset{\mathbb{R}}^n$ be a Lebesgue measurable set of locally
finite perimeter such that \eqref{Tay-1} holds. Then, for each
$x\in\partial^\ast\Omega$, there exists a Lebesgue measurable set
${\mathcal{O}}_x\subset(0,1)$ of density $1$ at $0$, i.e.,
satisfying
\begin{equation}\label{eq:pytr}
\lim_{\varepsilon\to 0^{+}}
\frac{{\mathscr{L}}^1\big({\mathcal{O}}_x\cap(0,\varepsilon)\big)}{\varepsilon}=1,
\end{equation}
with the property that
\begin{equation}\label{eq:UUt4ED}
\lim_{{\mathcal{O}}_x\ni r\to 0^{+}}
\frac{{\mathcal{H}}^{n-1}\big(\Omega\cap\partial B(x,r)\big)}
{\omega_{n-1}r^{n-1}}=\frac{1}{2}.
\end{equation}
\end{lemma}

\begin{proof}
We largely follow \cite{Ta-PC}.
Given $x\in\partial^\ast\Omega$, there exists an approximate tangent plane $\pi$ to $\Omega$ at $x$
(cf. the discussion on \cite[p.\,2627]{HoMiTa10}) and we denote by $\pi^{\pm}$ the two half-spaces
into which $\pi$ divides ${\mathbb{R}}^n$ (with the convention that the outward unit normal to $\pi^{-}$ is $\nu(x)$).
For each $r>0$, set $\partial^{\pm}B(x,r):=\partial B(x,r)\cap\pi^{\pm}$ and introduce
\begin{equation}\label{iagf.E1}
W(x,r):=\partial^{-}B(x,r)\triangle\big[\Omega\cap\partial B(x,r)\big]
\end{equation}
where, generally speaking, $U\triangle V$ denotes the symmetric difference $(U\setminus V)\cup(V\setminus U)$.
With this notation, in the proof of Proposition~3.3 on p.\,2628 of \cite{HoMiTa10} it has been shown that
\begin{equation}\label{iagf.E2}
\int_0^R{\mathcal{H}}^{n-1}\big(W(x,r)\big)\,dr=o(R^n)\,\,\mbox{ as }\,\,R\to 0^{+}.
\end{equation}
Thus, if we consider the function
\begin{equation}\label{iagf.E3}
\phi:(0,1)\to[0,\infty)\,\,\mbox{ given by }\,\,
\phi(r):=r^{1-n}{\mathcal{H}}^{n-1}\big(W(x,r)\big)\,\,\mbox{ for each }\,\,r\in(0,1),
\end{equation}
it follows from \eqref{iagf.E2} that
\begin{equation}\label{iagf.E4}
\int_{R/2}^R\phi(r)\,dr\leq\Big(\frac{R}{2}\Big)^{1-n}\int_0^R{\mathcal{H}}^{n-1}\big(W(x,r)\big)\,dr
=o(R)\,\,\mbox{ as }\,\,R\to 0^{+}.
\end{equation}
Bring in the dyadic intervals $I_k:=\big[2^{-(k+1)},2^{-k}\big]$ for $k\in{\mathbb{N}}_0$ and
note that \eqref{iagf.E4} entails
\begin{equation}\label{iagf.E5}
\delta_k:=\meanint_{I_k}\phi(r)\,dr\longrightarrow 0^{+}\,\,\mbox{ as }\,\,k\to\infty.
\end{equation}
For each $k\in{\mathbb{N}}_0$ split
\begin{equation}\label{iagf.E6}
I_k=A_k\cup B_k,\,\,\mbox{ with }\,\,
B_k:=\big\{r\in I_k:\,\phi(r)>\sqrt{\delta_k}\,\big\}\,\,\mbox{ and }\,\,
A_k:=I_k\setminus B_k.
\end{equation}
Then Chebychev's inequality permits us to estimate
\begin{equation}\label{iagf.E7}
\frac{{\mathscr{L}}^1(B_k)}{{\mathscr{L}}^1(I_k)}
\leq\frac{1}{\sqrt{\delta_k}}\meanint_{I_k}\phi(r)\,dr=\sqrt{\delta_k}\,,
\qquad\forall\,k\in{\mathbb{N}}_0.
\end{equation}
In light of \eqref{iagf.E5}, this implies that if we now define
\begin{equation}\label{iagf.E8}
{\mathcal{O}}_x:=\bigcup_{k\in{\mathbb{N}}_0}A_k\subset(0,1),
\end{equation}
then
\begin{equation}\label{iagf.E9}
\lim_{{\mathcal{O}}_x\ni r\to 0^{+}}\phi(r)=0.
\end{equation}
We claim that \eqref{eq:pytr} also holds for this choice of ${\mathcal{O}}_x$.
To see that this is the case, assume that some arbitrary $\theta>0$ has been fixed.
For each $\varepsilon\in(0,1)$, let $N_\varepsilon\in{\mathbb{N}}_0$ be such that
$2^{-N_\varepsilon-1}<\varepsilon\leq 2^{-N_\varepsilon}$. Since $N_\varepsilon\to\infty$
as $\varepsilon\to 0^{+}$, it follows from \eqref{iagf.E5} that there exists
$\varepsilon_\theta>0$ with the property that
\begin{equation}\label{iagf.E9aam}
\mbox{$\delta_k\leq\theta^2$ whenever $0<\varepsilon<\varepsilon_\theta$ and
$k\geq N_\varepsilon$}.
\end{equation}
Assuming that $0<\varepsilon<\varepsilon_\theta$ we may then estimate
\begin{align}\label{iagf.E9aanbV}
0 & \leq\frac{\varepsilon-{\mathscr{L}}^1\big({\mathcal{O}}_x\cap(0,\varepsilon)\big)}
{\varepsilon}
=\varepsilon^{-1}{\mathscr{L}}^1\big((0,\varepsilon)\setminus{\mathcal{O}}_x\big)
\nonumber\\[4pt]
&\leq\varepsilon^{-1}{\mathscr{L}}^1\big((0,2^{-N_\varepsilon})\setminus{\mathcal{O}}_x\big)
=\varepsilon^{-1}\sum_{k=N_\varepsilon}^\infty{\mathscr{L}}^1(B_k)
\nonumber\\[4pt]
&\leq\varepsilon^{-1}\sum_{k=N_\varepsilon}^\infty{\mathscr{L}}^1(I_k)\sqrt{\delta_k}
\leq\varepsilon^{-1}\theta 2^{-N_\varepsilon}\leq\theta/2.
\end{align}
This finishes the proof of \eqref{eq:pytr}. At this stage there remains to
observe that since, generally speaking,
$|{\mathscr{H}}^{n-1}(U)-{\mathscr{H}}^{n-1}(V)|\leq{\mathscr{H}}^{n-1}(U\Delta V)$,
from \eqref{iagf.E1} we have
\begin{equation}\label{iagf.E10}
\left|\frac{{\mathcal{H}}^{n-1}\big(\Omega\cap\partial B(x,r)\big)}
{\omega_{n-1}r^{n-1}}-\frac{1}{2}\right|\leq\frac{{\mathcal{H}}^{n-1}\big(W(x,r)\big)}
{\omega_{n-1}r^{n-1}}=\frac{1}{\omega_{n-1}}\phi(r),
\end{equation}
for each $r\in(0,1)$. Then formula \eqref{eq:UUt4ED} is a consequence of this and \eqref{iagf.E9}.
\end{proof}

Following G.~David and S.~Semmes \cite{DaSe91} we now make the
following definition.

\begin{definition}\label{Def-unif.rect}
Call a subset $\Sigma$ of ${\mathbb{R}}^{n}$ a {\tt uniformly} {\tt
rectifiable} {\tt set} provided it is Ahlfors regular and the
following holds. There exist $\varepsilon$, $M\in(0,\infty)$ such
that for each $x\in\Sigma$ and $R\in(0,{\rm diam}\,\Sigma)$, there
is a Lipschitz map $\varphi:B^{n-1}_R\rightarrow {\mathbb{R}}^{n}$
{\rm (}where $B^{n-1}_R$ is a ball of radius $R$ in
${\mathbb{R}}^{n-1}${\rm )} with Lipschitz constant $\leq M$, such
that
\begin{eqnarray}\label{3.1.9a}
{\mathcal{H}}^{n-1}\bigl(\Sigma\cap
B(x,R)\cap\varphi(B^{n-1}_R)\bigr)\geq \varepsilon R^{n-1}.
\end{eqnarray}
\end{definition}

\noindent Informally speaking, uniform rectifiability is about the
ability of identifying big pieces of Lipschitz images inside the
given set (in a uniform, scale invariant, fashion) and can be
thought of as a quantitative version of countable rectifiability.
From \cite[p.\,2629]{HoMiTa10} we know that
\begin{equation}\label{URRFVCa}
\parbox{7.00cm}{any uniformly rectifiable set $\Sigma\subset{\mathbb{R}}^n$ is countably rectifiable
(of dimension $n-1$).}
\end{equation}
Following \cite{HoMiTa10}, we shall also make the following
definition.

\begin{definition}\label{Def-UB}
Call a nonempty open proper subset $\Omega$ of ${\mathbb{R}}^n$ a {\tt UR}
{\rm (}{\tt uniformly} {\tt rectifiable}{\rm )} {\tt domain} provided $\Omega$
is an Ahlfors regular domain whose topological boundary, $\partial\Omega$, is a
uniformly rectifiable set.
\end{definition}

For further use, it is useful to point out that, as is apparent from
definitions,
\begin{eqnarray}\label{URUR}
\parbox{8.00cm}{if $\Omega\subset{\mathbb{R}}^n$ is a {\rm UR}
domain with $\partial\Omega=\partial(\overline{\Omega})$ then
${\mathbb{R}}^n\setminus\overline{\Omega}$ is a {\rm UR}
domain, with the same boundary.}
\end{eqnarray}

We now turn to the notion of nontangential boundary trace of
functions defined in a nonempty, proper, open set
$\Omega\subset{\mathbb{R}}^{n}$. Fix $\kappa>0$ and for each
boundary point $x\in\partial\Omega$ introduce the nontangential
approach region
\begin{eqnarray}\label{NT-1}
\Gamma_\kappa(x):=\{y\in\Omega:\,|x-y|<(1+\kappa)\,{\rm
dist}(y,\partial\Omega)\}.
\end{eqnarray}
It should be noted that, under the current hypotheses, it could
happen that $\Gamma_\kappa(x)=\varnothing$ for points
$x\in\partial\Omega$ (as is the case if, e.g., $\Omega$ has a
suitable cusp with vertex at $x$). Next, given a Lebesgue measurable function
$u:\Omega\to{\mathbb{R}}$, we wish to consider its limit at boundary points $x\in\partial\Omega$
taken from within nontangential approach regions with vertex at $x$.
For such a limit to be meaningfully defined at $\sigma$-a.e. point on $\partial\Omega$
(where, as usual, $\sigma:={\mathcal{H}}^{n-1}\lfloor\partial\Omega$), it is necessary
that
\begin{eqnarray}\label{YA-1}
x\in\overline{\Gamma_{\kappa}(x)}
\,\,\,\mbox{ for $\sigma$-a.e. }\,\,x\in\partial\Omega.
\end{eqnarray}
We shall call an open set $\Omega\subseteq{\mathbb{R}}^n$ satisfying
\eqref{YA-1} above {\tt weakly} {\tt accessible}. Assuming that this
is the case, we say that $u$ has a nontangential boundary trace almost everywhere on $\partial\Omega$
if for $\sigma$-a.e. point $x\in\partial\Omega$ there exists
some $N(x)\subset\Gamma_{\kappa}(x)$ of measure zero such that the limit
\begin{eqnarray}\label{nkc-EE-2}
\big(u\big|^{{}^{\rm n.t.}}_{\partial\Omega}\big)(x):=\lim_{\Gamma_{\kappa}(x)\setminus N(x)\ni y\to x}u(y)
\,\,\text{ exists}.
\end{eqnarray}
When $u$ is a continuous function in $\Omega$, we may take $N(x)=\varnothing$.
For future use, let us also define the nontangential maximal
operator of $u$ as
\begin{equation}\label{eq:NNNa}
\big({\mathcal{N}}u\big)(x)
:=\|u\|_{L^\infty(\Gamma_\kappa(x))}\in [0,\infty],\qquad\forall\,x\in\partial\Omega,
\end{equation}
where the essential supremum (taken to be $0$ if $\Gamma_\kappa(x)=\varnothing$)
in the right-hand side is taken with respect to the Lebesgue measure in ${\mathbb{R}}^n$.

The following result has been proved in \cite[Proposition~2.9, p.\,2588]{HoMiTa10}.

\begin{proposition}\label{Ctay-5}
Any Ahlfors regular domain is weakly accessible.
As a corollary, any {\rm UR} domain is weakly accessible.
\end{proposition}

We continue by recording the definition of the class of uniform
domains introduced by O.~Martio and J.~Sarvas in \cite{MaSa78}.

\begin{definition}\label{Def-UD}
Call a nonempty, proper, open set $\Omega\subseteq{\mathbb{R}}^n$ a
{\tt uniform} {\tt domain} if there exists a constant $c\in(0,\infty)$
with the property that:
\begin{eqnarray}\label{jon-1}
\parbox{11.00cm}{for each $x,y\in\Omega$ there exists $\gamma:[0,1]\to\Omega$ rectifiable curve
joining $x$ and $y$ such that ${\rm length}(\gamma)\leq c|x-y|$ and which has the property
$$
\hskip -1.50in
\min\,\bigl\{{\rm length}(\gamma_{x,z})\,,\,{\rm length}(\gamma_{z,y})\bigr\}\leq
c\,{\rm dist}(z,\partial\Omega)
$$
for all $z\in\gamma([0,1])$, where $\gamma_{x,z}$ and $\gamma_{z,y}$ are the two components of
the path $\gamma([0,1])$ joining $x$ with $z$, and $z$ with $y$, respectively.}
\end{eqnarray}
\end{definition}

Condition \eqref{jon-1} asserts that the length of $\gamma([0,1])$
is comparable to the distance between its endpoints and that, away
from its endpoints, the curve $\gamma$ stays correspondingly far
from $\partial\Omega$. Hence, heuristically, condition \eqref{jon-1}
implies that points in $\Omega$ can be joined in $\Omega$ by a
curvilinear (or twisted) double cone which is neither too crocked
nor too thin. Here we wish to note that, given an open nonempty
subset $\Omega$ of ${\mathbb{R}}^n$ with compact boundary along with
some $\alpha\in(0,1)$, the following implication holds:
\begin{equation}\label{eq:tr55}
\mbox{$\Omega$ is a ${\mathscr{C}}^{1+\alpha}$
domain}\,\Longrightarrow\, \mbox{$\Omega$ is a uniform domain}.
\end{equation}

Throughout, we make the convention that, given a nonempty, proper
subset $\Omega$ of ${\mathbb{R}}^{n}$, we abbreviate
\begin{equation}\label{eq:DFFV}
\rho(z):={\rm dist}(z,\partial\Omega)\,\,\,\mbox{ for every
}\,\,z\in\Omega.
\end{equation}

\begin{lemma}\label{Lem-J1}
Let $\Omega\subset{\mathbb{R}}^{n}$ be a uniform domain. Then for
each $\alpha\in(0,1)$ there exists a finite constant $C>0$,
depending only on $\alpha$ and $\Omega$, such that the estimate
\begin{eqnarray}\label{jon-1bis}
[u]_{\dot{\mathscr{C}}^\alpha(\Omega)} \leq
C\sup_{x\in\Omega}\,\Bigl\{\rho(x)^{1-\alpha}|\nabla u(x)|\Bigr\}
\end{eqnarray}
holds for every function $u\in {\mathscr{C}}^1(\Omega)$.
\end{lemma}

\begin{proof}
Consider $c>0$ such that condition \eqref{jon-1} is satisfied. Let
then $x,y\in\Omega$ be two arbitrary points and assume that $\gamma$
is as in Definition~\ref{Def-UD}. Denote by $L$ and $s$,
respectively, the length of the curve $\gamma^*:=\gamma([0,1])$ and
the arc-length parameter on $\gamma^*$, $s\in[0,L]$. Also, let
$[0,\,L]\ni s\mapsto\gamma(s)\in\gamma^*$ be the canonical
arc-length parametrization of $\gamma^*$. In particular,
$s\mapsto\gamma(s)$ is absolutely continuous,
$\left|\frac{d\gamma}{ds}\right|=1$ for almost every $s$, and for
every continuous function $f$ in $\Omega$
\begin{eqnarray}\label{integral}
\int_{\gamma^*}f:=\int_0^{L}f(\gamma(s))\,ds.
\end{eqnarray}
Thus, from \eqref{jon-1} and \eqref{integral}, for each
$\alpha\in(0,1)$ we have
\begin{eqnarray}\label{property}
\int_{\gamma^*}\rho^{\alpha-1} & = &
\int_0^L\rho(\gamma(s))^{\alpha-1}\,ds \leq c^{1-\alpha}
\int_0^L\Bigl[\min\,\{s,L-s\}\Bigr]^{\alpha-1}\,ds
\nonumber\\[6pt]
& \leq & 2c^{1-\alpha}\int_0^{L/2}s^{\alpha-1}\,ds
=C(c,\alpha)L^{\alpha}\leq C(c,\alpha)|x-y|^{\alpha}.
\end{eqnarray}
Then, since $\left|\frac{d\gamma}{ds}\right|=1$ for almost every
$s$, for every $u\in {\mathscr{C}}^1(\Omega)$ we may write
\begin{align}\label{jon-4}
|u(x)-u(y)|
&=\Bigl|\int_0^{L}\frac{d}{ds}\big[u(\gamma(s))\big]\,ds\Bigr|
\nonumber\\[4pt]
& \leq\int_0^L\big|(\nabla
u)(\gamma(s))\big|\,ds=\int_{\gamma^*}|\nabla u|
\nonumber\\[4pt]
& \leq\sup_{\gamma^*}\Bigl\{|\nabla u|\,\rho^{1-\alpha}\Bigr\}
\int_{\gamma^*}\rho^{\alpha-1}
\nonumber\\[4pt]
& \leq C|x-y|^{\alpha}\Bigl\||\nabla
u|\,\rho^{1-\alpha}\Bigr\|_{L^\infty(\Omega)},
\end{align}
finishing the proof of \eqref{jon-1bis}.
\end{proof}

Recall that for each $k\in{\mathbb{N}}$ we let ${\mathscr{L}}^k$
stand for the $k$-dimensional Lebesgue measure in ${\mathbb{R}}^k$.
Also, we shall let $\langle\cdot,\cdot\rangle$ denote the standard
inner product of vectors in $\mathbb{R}^n$.

\begin{lemma}\label{uteeLL}
Assume that $D\subseteq{\mathbb{R}}^n$ is a set of locally finite
perimeter. Denote by $\nu$ its geometric measure theoretic outward
unit normal and define
$\sigma:={\mathcal{H}}^{n-1}\lfloor\partial_\ast D$. Also, suppose
that $\vec{F}\in{\mathscr{C}}^1_0({\mathbb{R}}^n,{\mathbb{R}}^n)$.
Then for each $x\in{\mathbb{R}}^n$,
\begin{equation}\label{eq:TgavFF.iii}
\int_{D\cap B(x,\,r)}{\rm div}\vec{F}\,d{\mathscr{L}}^n
=\int_{\partial_\ast D\cap
B(x,\,r)}\langle\vec{F},\nu\rangle\,d\sigma +\int_{D\cap\partial
B(x,\,r)}\langle\vec{F},\nu\rangle\,d{\mathcal{H}}^{n-1}
\end{equation}
and
\begin{equation}\label{eq:TgavFF}
\int_{D\setminus B(x,\,r)}{\rm div}\vec{F}\,d{\mathscr{L}}^n
=\int_{\partial_\ast D\setminus
B(x,\,r)}\langle\vec{F},\nu\rangle\,d\sigma -\int_{D\cap\partial
B(x,\,r)}\langle\vec{F},\nu\rangle\,d{\mathcal{H}}^{n-1}
\end{equation}
for ${\mathscr{L}}^1$-a.e. $r\in(0,\infty)$, where $\nu$ in each of
the last integrals in the above right hand-sides is the outward unit
normal to $B(x,r)$.
\end{lemma}

\begin{proof}
Identity \eqref{eq:TgavFF.iii} is simply \cite[Lemma~1,
p.\,195]{EvGa92}. Then \eqref{eq:TgavFF} follows by combining this
with the Gauss-Green formula from \cite[Theorem~1, p.\,209]{EvGa92}.
\end{proof}

We conclude this section by recording the following two-dimensional
result which is going to be relevant when dealing with the proof of
Theorem~\ref{THM-main222}.

\begin{proposition}\label{treePUb}
Let $\Omega\subseteq{\mathbb{C}}$ be a bounded open set whose
boundary is an upper Ahlfors regular Jordan curve. Then $\Omega$ is
a simply connected {\rm UR} domain satisfying
$\partial\Omega=\partial(\overline{\Omega})$. Hence, in particular,
${\mathcal{H}}^1(\partial\Omega\setminus\partial_\ast\Omega)=0$ and
${\mathbb{C}}\setminus\overline{\Omega}$ is also a {\rm UR} domain
with the same boundary as $\Omega$.

Moreover, the curve $\partial\Omega$ is rectifiable and if $L$
denotes its length and $[0,L]\ni s\mapsto z(s)\in\Sigma$ is its
arc-length parametrization, then
\begin{equation}\label{jgvc-A.1iY}
{\mathscr{H}}^1(E)={\mathscr{L}}^1(z^{-1}(E)),\qquad\forall\,E\subseteq\partial\Omega
\,\,\mbox{ measurable set}
\end{equation}
where ${\mathscr{L}}^1$ is the one-dimensional Lebesgue measure, and
if $\nu$ denotes the geometric measure theoretic outward unit normal
to $\Omega$ then
\begin{eqnarray}\label{Ic-1aGf.Fva.3}
\nu(z(s))=-i\,z'(s)\,\,\,\mbox{ for ${\mathscr{L}}^1$-a.e.
$s\in[0,L]$}.
\end{eqnarray}
\end{proposition}

\noindent A proof of Proposition~\ref{treePUb} may be found in
\cite{IMiMiTa.1}.

\section{Background and Preparatory Estimates for Singular Integrals}
\setcounter{equation}{0} \label{S-3}

The proofs of the main results require a number of prerequisites,
and this section collects several useful estimates for singular
integral operators. The first theorem in this regard essentially amounts to
a version of the Calder\'on-Zygmund theory for singular integrals on
uniformly rectifiable sets.

\begin{theorem}\label{Main-T2-BBBB}
There exists a positive integer $N=N(n)$ with the following
significance. Suppose $\Sigma\subseteq{\mathbb{R}}^{n}$ is a uniformly rectifiable set and define
$\sigma:={\mathcal{H}}^{n-1}\lfloor\Sigma$. Also, consider a function
\begin{eqnarray}\label{ker}
\begin{array}{l}
k\in{\mathscr{C}}^N({\mathbb{R}}^{n}\setminus\{0\})\,\,
\mbox{ satisfying }\,\, k(-x)=-k(x)\,\,\mbox{ for each }
\,\,x\in{\mathbb{R}}^n
\\[4pt]
\mbox{and so that }\,\,k(\lambda\,x)=\lambda^{-(n-1)}k(x)
\,\,\mbox{ for all }\,\, \lambda>0,\,\,\,
x\in{\mathbb{R}}^{n}\setminus\{0\}.
\end{array}
\end{eqnarray}
For each $\varepsilon>0$, consider the truncated singular
integral operator
\begin{eqnarray}\label{pv-layer2}
T_{\varepsilon}f(x):=\int\limits_{\stackrel{y\in\Sigma}{|x-y|>\varepsilon}}
k(x-y)f(y)\,d\sigma(y),\qquad x\in\Sigma,
\end{eqnarray}
and define the maximal operator $T_*$ by setting
\begin{eqnarray}\label{T-pv}
T_*f(x):=\sup_{\varepsilon>0}|T_{\varepsilon}f(x)|,\qquad x\in\Sigma.
\end{eqnarray}

Then for each $p\in(1,\infty)$ there exists a constant
$C\in(0,\infty)$ depending only on $p$ and $\Sigma$, such that
\begin{eqnarray}\label{Tmax-bdd}
\|T_*f\|_{L^p(\Sigma,\sigma)}\leq
C\big\|k\big|_{S^{n-1}}\big\|_{{\mathscr{C}}^N(S^{n-1})}\|f\|_{L^p(\Sigma,\sigma)}
\end{eqnarray}
for every $f\in L^p(\Sigma,\sigma)$. Furthermore, given any
$p\in[1,\infty)$, for each function $f\in L^p(\Sigma,\sigma)$ the limit
\begin{eqnarray}\label{main-lim}
Tf(x):=\lim_{\varepsilon\to 0^+}T_{\varepsilon}f(x)
\end{eqnarray}
exists for $\sigma$-a.e. $x\in\Sigma$, and the induced operators
\begin{align}\label{maKP56}
& T:L^p(\Sigma,\sigma)\longrightarrow
L^p(\Sigma,\sigma),\quad p\in(1,\infty),
\\[4pt]
& T:L^1(\Sigma,\sigma)\longrightarrow
L^{1,\infty}(\Sigma,\sigma),
\label{maKPbis56}
\end{align}
are well-defined, linear and bounded. In addition, for each
$p\in(1,\infty)$, the adjoint of the operator $T$ acting on
$L^p(\Sigma,\sigma)$ is $-T$ acting on $L^{p'}(\Sigma,\sigma)$
with $1/p+1/p'=1$. Finally, corresponding to the end-point $p=\infty$,
the operator $T$ also induces a linear and bounded mapping
\begin{align}\label{maKP56-BBN}
T:L^\infty(\Sigma,\sigma)\longrightarrow
{\rm BMO}\,(\Sigma).
\end{align}
\end{theorem}

Once the existence of the principal value singular integral operator $T$ defined
by the limit in \eqref{main-lim} has been established, all other claims follow
from \cite{DaSe91} and standard harmonic analysis. As far as the issue of
well-definiteness of $T$ is concerned, it is not difficult to reduce matters
to case when $\Sigma$ is a $(n-1)$-dimensional Lipschitz graph (\cite{Ta-PC}).
In the latter scenario, the desired result is known. For example, the desired
conclusion is contained in \cite[Theorem~3.33, p.\,2669]{HoMiTa10}, where a more
general result, (applicable to variable coefficient operators on boundaries of
{\rm UR} domains) can be found. A direct proof for Lipschitz graphs may be found in
\cite[Proposition~B.2, p.\,163]{HMT15}. In this vein, see also \cite[pp.\,63-64]{Dav}
for a sketch of a proof.

Our next theorem deals with nontangential maximal function estimates and jump-relations
for integral operators on {\rm UR} domains. For a proof, the reader is once again
referred to \cite[Theorem~3.33, p.\,2669]{HoMiTa10}.

\begin{theorem}\label{Main-T2}
Assume $\Omega\subset{\mathbb{R}}^n$ is a {\rm UR} domain and let
$\sigma:={\mathcal{H}}^{n-1}\lfloor\partial\Omega$ and $\nu$ denote
respectively, the surface measure on $\partial\Omega$ and the outward
unit normal to $\Omega$. Select a function $k$ as in \eqref{ker} with $N=N(n)$
sufficiently large, and define
\begin{eqnarray}\label{T-layer}
{\mathcal{T}}f(x):=\int_{\partial\Omega}k(x-y)f(y)\,d\sigma(y),
\qquad x\in\Omega.
\end{eqnarray}

Then for each $p\in(1,\infty)$ there exists a finite
constant $C=C(\Omega,k,p)>0$ such that
\begin{eqnarray}\label{T-Har}
\|{\mathcal{N}}({\mathcal{T}}f)\|_{L^p(\partial\Omega,\sigma)} \leq
C\|f\|_{L^p(\partial\Omega,\sigma)},\qquad\forall\,f\in L^p(\partial\Omega,\sigma),
\end{eqnarray}
and, corresponding to $p=1$,
\begin{eqnarray}\label{T-ntBIS}
\|{\mathcal{N}}({\mathcal{T}}f)\|_{L^{1,\infty}(\partial\Omega,\sigma)}
\leq C\|f\|_{L^1(\partial\Omega,\sigma)},\qquad\forall\,f\in L^1(\partial\Omega,\sigma).
\end{eqnarray}
Also, if `hat' denotes the Fourier transform in ${\mathbb{R}}^n$ and
$i:=\sqrt{-1}\in{\mathbb{C}}$, then for every $f\in L^p(\partial\Omega,\sigma)$ with $p\in[1,\infty)$
the jump-formula
\begin{eqnarray}\label{main-jump}
\Big({\mathcal{T}}f\big|^{{}^{\rm n.t.}}_{\partial\Omega}\Big)(x)
=\lim\limits_{\stackrel{\Omega\ni z\to
x}{z\in\Gamma_{\kappa}(x)}}{\mathcal{T}}f(z)
=\frac{1}{2i}\,\widehat{k}(\nu(x))f(x)+Tf(x)
\end{eqnarray}
is valid at $\sigma$-a.e. point $x\in\partial\Omega$, where $T$ is the principal value
singular integral operator associated with the kernel $k$ as in \eqref{main-lim}.
\end{theorem}

The Fourier transform in ${\mathbb{R}}^n$ employed in
\eqref{main-jump} is
\begin{equation}\label{eq:iytTT}
\widehat{\phi}(\xi):=\int_{{\mathbb{R}}^n}e^{-i\langle
x,\,\xi\rangle}\phi(x)\,dx, \qquad\xi\in{\mathbb{R}}^n.
\end{equation}
Let us also remark here that the hypotheses \eqref{ker} imposed on
the kernel $k$ imply that
$|k(x)|\leq\|k\|_{L^\infty(S^{n-1})}|x|^{1-n}$ for each
$x\in{\mathbb{R}}^n\setminus\{0\}$. Hence, $k$ is a tempered
distribution in ${\mathbb{R}}^n$ and $\widehat{k}$, originally
considered in the class of tempered distributions in
${\mathbb{R}}^n$, satisfies
\begin{equation}\label{eq:hat-K}
\begin{array}{c}
\widehat{k}\in{\mathscr{C}}^{\,m}({\mathbb{R}}^n\setminus\{0\})
\,\,\mbox{ if $N\in{\mathbb{N}}$ is even}
\\[4pt]
\mbox{and $m\in{\mathbb{N}}_0$ is such that $m<N-1$}
\end{array}
\end{equation}
(cf. \cite[Exercise~4.60, p.\,133]{DM}). In particular,
\eqref{eq:hat-K} ensures that $\widehat{k}(\nu(x))$ is meaningfully
defined in \eqref{main-jump} for $\sigma$-a.e. $x\in\partial\Omega$
whenever $N\geq 2$.

\begin{lemma}\label{Lemma.1}
Suppose $\Omega$ is a nonempty proper open subset of
${\mathbb{R}}^n$ with a compact boundary, satisfying an upper
Ahlfors regularity condition with constant $c\in(0,\infty)$. In this
setting, define $\sigma:={\mathcal{H}}^{n-1}\lfloor\partial\Omega$
and consider an integral operator
\begin{equation}\label{eq:TtaY.22}
\mathscr{T}\!f(x):=\int_{\partial\Omega}k(x,y)f(y)\,d\sigma(y),\qquad
x\in\Omega,
\end{equation}
whose kernel $k:\Omega\times\partial\Omega\to{\mathbb{R}}$ has the
property that there exists some finite positive constant $C_0$ such
that
\begin{align}\label{yre}
|k(x,y)|\leq\frac{C_0}{|x-y|^{n-1}}
\end{align}
for each $x\in\Omega$ and $\sigma$-a.e. $y\in\partial\Omega$. Also,
suppose that
\begin{align}\label{tAAcc.22}
\sup_{x\in\Omega}\,\big|\mathscr{T}1(x)\big|<+\infty.
\end{align}

Then for every $\alpha\in(0,1)$ one has
\begin{align}\label{jht6}
\sup_{x\in\Omega}\,|\mathscr{T}\!f(x)| & \leq
c\,C_0\frac{2^{2n-2+\alpha}}{2^\alpha-1}\big(1+[{\rm
diam}(\partial\Omega)]^\alpha\big)
[f]_{\dot{\mathscr{C}}^\alpha(\partial\Omega)}
\nonumber\\[4pt]
&\quad +\Big(\|\mathscr{T}1\|_{L^\infty(\Omega)} +c\,C_0\,\big[{\rm
diam}\,(\partial\Omega)\big]^{n-1}\Big)\|f\|_{L^\infty(\partial\Omega)},
\end{align}
for every $f\in{\mathscr{C}}^\alpha(\partial\Omega)$.
\end{lemma}

\begin{proof}
Pick an arbitrary $f\in{\mathscr{C}}^\alpha(\partial\Omega)$ and fix
any $x\in\Omega$. Consider first the case when ${\rm
dist}(x,\partial\Omega)\geq 1$, in which scenario we may directly
estimate
\begin{align}\label{jhtttt}
|\mathscr{T}\!f(x)|\leq
C_0\,\sigma(\partial\Omega)\|f\|_{L^\infty(\partial\Omega)} \leq
c\,C_0\,\big[{\rm
diam}\,(\partial\Omega)\big]^{n-1}\|f\|_{L^\infty(\partial\Omega)}.
\end{align}
In the case when ${\rm dist}(x,\partial\Omega)<1$, select a point
$x_\ast\in\partial\Omega$ such that
\begin{equation}\label{eq:rt65g}
|x-x_\ast|={\rm dist}(x,\partial\Omega)=:r\in(0,1)
\end{equation}
and split $\mathscr{T}\!f(x)=I+II+III$, where
\begin{align}\label{yraIh7}
I &:=\int_{\partial\Omega\cap
B(x_\ast,2r)}k(x,y)\big[f(y)-f(x_\ast)\big]\,d\sigma(y),
\\[4pt]
II &:=\int_{\partial\Omega\setminus
B(x_\ast,2r)}k(x,y)\big[f(y)-f(x_\ast)\big]\,d\sigma(y),
\label{yraIh8}
\end{align}
and
\begin{equation}\label{yraIh9}
III:=(\mathscr{T}1)(x)f(x_\ast).
\end{equation}
Note that
\begin{align}\label{yrrre}
|I| & \leq\int_{\partial\Omega\cap
B(x_\ast,2r)}|k(x,y)|\,|f(y)-f(x_\ast)|\,d\sigma(y)
\nonumber\\[4pt]
& \leq C_0[f]_{\dot{\mathscr{C}}^\alpha(\partial\Omega)}
\int_{\partial\Omega\cap B(x_\ast,2r)}\frac{|y-x_\ast|^\alpha}
{\lvert x-y\rvert^{n-1}}\,d\sigma(y)
\nonumber\\[4pt]
& \leq C_0[f]_{\dot{\mathscr{C}}^\alpha(\partial\Omega)}\,
\frac{(2r)^\alpha}{r^{n-1}}\,\sigma\big(\partial\Omega\cap
B(x_\ast,2r)\big),
\end{align}
where the third inequality comes from the fact that $\lvert
y-x_\ast\rvert^\alpha\leq(2r)^\alpha$ on the domain of integration,
and the fact that $1/\lvert x-y\rvert\leq 1/\lvert
x-x_\ast\rvert=1/r$, for all $y\in\partial\Omega$. Hence,
\begin{align}
\lvert I\rvert\leq 2^{n-1+\alpha}c\,C_0
[f]_{\dot{\mathscr{C}}^\alpha(\partial\Omega)},
\end{align}
bearing in mind \eqref{eq:rt65g} and the upper Ahlfors regularity of
$\partial\Omega$. Also,
\begin{align}\label{yrrr075}
\big|II\big|\leq C_0[f]_{\dot{\mathscr{C}}^\alpha(\partial\Omega)}
\int_{\partial\Omega\setminus B(x_\ast,2r)} \frac{\lvert
y-x_\ast\rvert^\alpha}{|x-y|^{n-1}}\,d\sigma(y).
\end{align}
Note that if $y\in\partial\Omega\setminus B(x_\ast,2r)$ then
\begin{align}\label{yrrr076}
\lvert y-x_\ast\rvert\leq\lvert y-x\rvert+\lvert
x-x_\ast\rvert\quad\textrm{and}\quad r\leq\frac{\lvert
y-x_\ast\rvert}{2}\implies\lvert y-x_\ast\rvert\leq 2\lvert
y-x\rvert.
\end{align}
Hence, $1/|x-y|^{n-1}\leq 2^{n-1}/|y-x_\ast|^{n-1}$ on the domain of
integration $\partial\Omega\setminus B(x_\ast,2r)$. Also, if we
introduce
\begin{equation}\label{eq:NNa}
N:=\Big[\log_2\big(\tfrac{{\rm
diam}(\partial\Omega)}{r}\big)\Big]\in{\mathbb{N}},
\end{equation}
then $\partial\Omega\setminus B(x_\ast,2^kr)=\varnothing$ for each
integer $k>N$. Together, these observations and \eqref{yrrr075}
allow us to estimate
\begin{align}
|II| & \leq 2^{n-1}C_0[f]_{\dot{\mathscr{C}}^\alpha(\partial\Omega)}
\int_{\partial\Omega\setminus
B(x_\ast,2r)}\frac{|y-x_\ast|^\alpha}{|y-x_\ast|^{n-1}}\,d\sigma(y)
\\[4pt]
& \leq 2^{n-1}C_0[f]_{\dot{\mathscr{C}}^\alpha(\partial\Omega)}
\sum_{k=1}^N\,\int\limits_{\partial\Omega\cap[B(x_\ast,2^{k+1}r)\setminus
B(x_\ast,2^kr)]} \frac{1}{|y-x_\ast|^{n-1-\alpha}}\,d\sigma(y)
\nonumber
\\[4pt]
&\leq 2^{n-1}C_0[f]_{\dot{\mathscr{C}}^\alpha(\partial\Omega)}
\sum_{k=1}^N (2^kr)^{-(n-1-\alpha)}\sigma\big(\partial\Omega\cap
B(x_\ast,2^{k+1}r)\big). \nonumber
\end{align}
Thus, by the upper Ahlfors regularity condition,
\begin{align}
|II| & \leq 2^{n-1}C_0[f]_{\dot{\mathscr{C}}^\alpha(\partial\Omega)}
\sum_{k=1}^N (2^kr)^{-(n-1-\alpha)}c(2^{k+1}r)^{n-1} \nonumber
\\[4pt]
& =2^{2n-2}c\,C_0
r^\alpha[f]_{\dot{\mathscr{C}}^\alpha(\partial\Omega)}\sum_{k=1}^N
(2^\alpha)^k
\nonumber\\[4pt]
& \leq 2^{2n-2+\alpha}c\,C_0
r^\alpha[f]_{\dot{\mathscr{C}}^\alpha(\partial\Omega)}
\frac{(2^N)^\alpha}{2^\alpha-1}
\nonumber\\[4pt]
& \leq \frac{2^{2n-2+\alpha}}{2^\alpha-1}
c\,C_0[f]_{\dot{\mathscr{C}}^\alpha(\partial\Omega)} \big[{\rm
diam}\,(\partial\Omega)\big]^\alpha.
\end{align}
Since, clearly,
$|III|\leq\|\mathscr{T}1\|_{L^\infty(\Omega)}\|f\|_{L^\infty(\partial\Omega)}$,
the desired conclusion follows.
\end{proof}

\begin{lemma}\label{Lemma.2}
Retain the same assumptions on $\Omega$ as in Lemma~\ref{Lemma.1}
and consider an integral operator
\begin{equation}\label{eq:TtaY.33}
\mathscr{Q}\!f(x):=\int_{\partial\Omega}q(x,y)f(y)\,d\sigma(y),\qquad
x\in\Omega,
\end{equation}
whose kernel $q:\Omega\times\partial\Omega\to{\mathbb{R}}$ is
assumed to satisfy
\begin{align}\label{yre.33}
|q(x,y)|\leq\frac{C_1}{|x-y|^{n}},\qquad\forall\,x\in\Omega,\quad\forall\,y\in\partial\Omega,
\end{align}
for some finite positive constant $C_1$. Also, with $\rho$ as in
\eqref{eq:DFFV}, suppose there exists $\alpha\in(0,1)$ with the property that
\begin{align}\label{tAAcc.33}
C_2:=\sup_{x\in\Omega}\,\Big\{\rho(x)^{1-\alpha}|(\mathscr{Q}1)(x)|\Big\}<+\infty.
\end{align}

Then one has
\begin{align}\label{jht6.33}
\sup_{x\in\Omega}\,\Big\{\,\rho(x)^{1-\alpha}|\mathscr{Q}\!f(x)|\,\Big\}
\leq\frac{2^{2n-1+\alpha}}{1-2^{\alpha-1}}\,c\,C_1[f]_{\dot{\mathscr{C}}^\alpha(\partial\Omega)}
+C_2\|f\|_{L^\infty(\partial\Omega)},
\end{align}
for every $f\in{\mathscr{C}}^\alpha(\partial\Omega)$.
\end{lemma}

\begin{proof}
Select an arbitrary $f\in{\mathscr{C}}^\alpha(\partial\Omega)$. Pick
some $x\in\Omega$ and choose $x_\ast\in\partial\Omega$ such that
$|x-x_\ast|=\rho(x)=:r$. Split $\mathscr{Q}\!f(x)=I+II+III$, where
\begin{align}\label{yraIh7.33}
I &:=\int_{\partial\Omega\cap
B(x_\ast,2r)}q(x,y)\big[f(y)-f(x_\ast)\big]\,d\sigma(y),
\\[4pt]
II &:=\int_{\partial\Omega\setminus
B(x_\ast,2r)}q(x,y)\big[f(y)-f(x_\ast)\big]\,d\sigma(y),
\label{yraIh8.33}
\end{align}
and
\begin{equation}\label{yraIh9.33}
III:=(\mathscr{Q}1)(x)f(x_\ast).
\end{equation}
Then
\begin{align}\label{yrrre.33}
|I| & \leq\int_{\partial\Omega\cap B(x_\ast,2r)}
|q(x,y)|\,|f(y)-f(x_\ast)|\,d\sigma(y)
\nonumber\\[4pt]
& \leq C_1[f]_{\dot{\mathscr{C}}^\alpha(\partial\Omega)}
\int_{\partial\Omega\cap
B(x_\ast,2r)}\frac{|y-x_\ast|^\alpha}{|x-y|^{n}}\,d\sigma(y)
\nonumber\\[4pt]
& \leq C_1[f]_{\dot{\mathscr{C}}^\alpha(\partial\Omega)}\,
\frac{(2r)^\alpha}{r^{n}}\,\sigma\big(\partial\Omega\cap
B(x_\ast,2r)\big)
\nonumber\\[4pt]
& \leq
2^{n-1+\alpha}c\,C_1\rho(x)^{\alpha-1}[f]_{\dot{\mathscr{C}}^\alpha(\partial\Omega)}.
\end{align}
Next, keeping in mind that $1/|x-y|^{n}\leq 2^{n}/|y-x_\ast|^{n}$ on
$\partial\Omega\setminus B(x_\ast,2r)$ (cf. \eqref{yrrr076}), we may
estimate
\begin{align}\label{yrrr075.33}
|II| &\leq C_1[f]_{\dot{\mathscr{C}}^\alpha(\partial\Omega)}
\int_{\partial\Omega\setminus B(x_\ast,2r)}
\frac{|y-x_\ast|^\alpha}{|x-y|^{n}}\,d\sigma(y).
\nonumber\\[4pt]
& \leq 2^{n}C_1[f]_{\dot{\mathscr{C}}^\alpha(\partial\Omega)}
\int_{\partial\Omega\setminus
B(x_\ast,2r)}\frac{|y-x_\ast|^\alpha}{|y-x_\ast|^{n}}\,d\sigma(y)
\nonumber\\[4pt]
& \leq 2^{n}C_1[f]_{\dot{\mathscr{C}}^\alpha(\partial\Omega)}
\sum_{k=1}^\infty\,\int\limits_{\partial\Omega\cap[B(x_\ast,2^{k+1}r)\setminus
B(x_\ast,2^kr)]} \frac{1}{|y-x_\ast|^{n-\alpha}}\,d\sigma(y)
\nonumber
\\[4pt]
&\leq 2^{n}C_1[f]_{\dot{\mathscr{C}}^\alpha(\partial\Omega)}
\sum_{k=1}^\infty (2^kr)^{-(n-\alpha)}\sigma\big(\partial\Omega\cap
B(x_\ast,2^{k+1}r)\big)
\nonumber\\[4pt]
& \leq 2^{n}C_1[f]_{\dot{\mathscr{C}}^\alpha(\partial\Omega)}
\sum_{k=1}^\infty (2^kr)^{-(n-\alpha)}c(2^{k+1}r)^{n-1}
\nonumber\\[4pt]
& =2^{2n-1}c\,C_1 r^{\alpha-1}
[f]_{\dot{\mathscr{C}}^\alpha(\partial\Omega)}\sum_{k=1}^\infty
(2^{\alpha-1})^k
\nonumber\\[4pt]
& =\frac{2^{2n-2+\alpha}}{1-2^{\alpha-1}}c\,C_1\rho(x)^{\alpha-1}
[f]_{\dot{\mathscr{C}}^\alpha(\partial\Omega)}.
\end{align}
Given that $\rho(x)^{1-\alpha}|III|\leq
C_2\|f\|_{L^\infty(\partial\Omega)}$, estimate \eqref{jht6.33} is
established.
\end{proof}

\begin{lemma}\label{Lemma.3}
Let $\Omega$ be a nonempty open proper subset of ${\mathbb{R}}^n$
whose boundary is compact and satisfies an upper Ahlfors regularity
condition with constant $c\in(0,\infty)$. In this setting, define
$\sigma:={\mathcal{H}}^{n-1}\lfloor\partial\Omega$ and consider an
integral operator
\begin{equation}\label{eq:TtaY.44}
\mathcal{T}\!f(x):=\int_{\partial\Omega}K(x,y)f(y)\,d\sigma(y),\qquad
x\in\Omega,
\end{equation}
whose kernel $K:\Omega\times\partial\Omega\to{\mathbb{R}}$ has the
property that there exists a finite constant $B>0$ such that
\begin{align}\label{yre.44a}
|K(x,y)|+|x-y||\nabla_x K(x,y)|\leq\frac{B}{|x-y|^{n-1}}
\end{align}
for each $x\in\Omega$ and $\sigma$-a.e. $y\in\partial\Omega$. Fix
some $\alpha\in(0,1)$ and suppose that
\begin{align}\label{tAAcc.44iii}
A:=\sup_{x\in\Omega}\,\big|(\mathcal{T}1)(x)\big|
+\sup_{x\in\Omega}\,\Big\{\rho(x)^{1-\alpha}\big|\nabla(\mathcal{T}1)(x)\big|\Big\}<+\infty.
\end{align}

Then for every $f\in{\mathscr{C}}^\alpha(\partial\Omega)$ one has
\begin{align}\label{jht6.44}
\sup_{x\in\Omega}\,|\mathcal{T}\!f(x)| +\sup_{x\in\Omega}\, &
\Big\{\rho(x)^{1-\alpha}\big|\nabla(\mathcal{T}f)(x)\big|\Big\}
\nonumber\\[4pt]
&\leq c\,B\,C_{n,\alpha} \big(2+[{\rm
diam}(\partial\Omega)]^\alpha\big)[f]_{\dot{\mathscr{C}}^\alpha(\partial\Omega)}
\nonumber\\[4pt]
&\quad +\big(2A+c\,B[{\rm
diam}(\partial\Omega)]^{n-1}\big)\|f\|_{L^\infty(\partial\Omega)},
\end{align}
where
\begin{equation}\label{eq:Bcc77}
C_{n,\alpha}:=2^{2n-2-\alpha}\max\big\{(2^\alpha-1)^{-1}\,,\,2(1-2^{\alpha-1})^{-1}\big\}.
\end{equation}

As a consequence, there exists a finite constant
$C_{n,\alpha,\Omega}>0$ with the property that for every
$f\in{\mathscr{C}}^\alpha(\partial\Omega)$ one has
\begin{align}\label{jht6.44nB}
\sup_{x\in\Omega}\,|\mathcal{T}\!f(x)|
+\sup_{x\in\Omega}\,\Big\{\rho(x)^{1-\alpha}\big|\nabla(\mathcal{T}f)(x)\big|\Big\}
\leq
C_{n,\alpha,\Omega}(A+B)\|f\|_{{\mathscr{C}}^\alpha(\partial\Omega)}.
\end{align}
\end{lemma}

\begin{proof}
This is an immediate consequence of Lemma~\ref{Lemma.1} and
Lemma~\ref{Lemma.2}.
\end{proof}

\section{Clifford Analysis}
\setcounter{equation}{0} \label{S-4}

A key tool for us is Clifford analysis, and here we elaborate on
those aspects used in the proof of Theorem~\ref{Main-T1aa}. To
begin, the {\tt Clifford} {\tt algebra} with $n$ imaginary units is
the minimal enlargement of ${\mathbb{R}}^{n}$ to a unitary real
algebra $({\mathcal{C}}\!\ell_{n},+,\odot)$, which is not generated
(as an algebra) by any proper subspace of ${\mathbb{R}}^{n}$, and
such that
\begin{equation}\label{X-sqr}
x\odot x=-|x|^2\quad\mbox{for any
}\,x\in{\mathbb{R}}^{n}\hookrightarrow{\mathcal{C}}\!\ell_{n}.
\end{equation}
This identity readily implies that, if $\{e_j\}_{1\leq j\leq n}$ is
the standard orthonormal basis in ${\mathbb{R}}^{n}$, then
\begin{equation}\label{im-e}
e_j\odot e_j=-1\quad\mbox{and}\quad e_j\odot e_k=-e_k\odot e_j
\,\,\mbox{ whenever }\,1\leq j\neq k\leq n.
\end{equation}
In particular, identifying the canonical basis $\{e_j\}_{1\leq j\leq
n}$ from ${\mathbb{R}}^{n}$ with the $n$ imaginary units generating
${\mathcal{C}}\!\ell_{n}$, yields the embedding\footnote{As the
alerted reader might have noted, for $n=2$ the identification in
\eqref{embed} amounts to embedding ${\mathbb{R}}^2$ into
quaternions, i.e., ${\mathbb{R}}^2\hookrightarrow{\mathbb{H}}
:=\{x_0+x_1{\mathbf{i}}+x_2{\mathbf{j}}+x_3{\mathbf{k}}:\,
x_0,x_1,x_2,x_3\in{\mathbb{R}}\}$ via
${\mathbb{R}}^2\ni(x_1,x_2)\equiv
x_1{\mathbf{i}}+x_2{\mathbf{j}}\in{\mathbb{H}}$. The reader is
reassured that this is simply a matter of convenience, and we might
as well have arranged so that the embedding \eqref{embed} comes
down, when $n=2$, to perhaps the more familiar identification
${\mathbb{R}}^2\equiv{\mathbb{C}}$, by taking ${\mathbb{R}}^n\ni
x=(x_0,x_1,\dots,x_{n-1})\equiv x_0+x_1 e_1+\dots
x_{n-1}e_{n-1}\in{\mathcal{C}}\!\ell_{n-1}$. The latter choice leads
to a parallel theory to the one presented here, entailing only minor
natural alterations.}
\begin{equation}\label{embed}
{\mathbb{R}}^{n}\hookrightarrow{\mathcal{C}}\!\ell_{n},\qquad
{\mathbb{R}}^{n}\ni x=(x_1,\dots,x_{n})\equiv
\sum_{j=1}^{n}x_je_j\in{\mathcal{C}}\!\ell_{n}.
\end{equation}
Also, any element $u\in{\mathcal{C}}\!\ell_{n}$ can be uniquely
represented in the form
\begin{equation}\label{eq4.1}
u=\sum_{l=0}^{n}{\sum_{|I|=l}}'u_I\,e_I,\quad u_I\in{\mathbb{R}}.
\end{equation}
Here $e_I$ stands for the product $e_{i_1}\odot
e_{i_2}\odot\cdots\odot e_{i_l}$ if $I=(i_1,i_2,\dots,i_l)$ and
$e_0:=e_{\varnothing}:=1$ is the multiplicative unit. Also, $\sum'$
indicates that the sum is performed only over strictly increasing
multi-indices, i.e., $I=(i_1,i_2,\dots,i_l)$ with $1\leq
i_1<i_2<\dots<i_l\leq n$. We endow ${\mathcal{C}}\!\ell_{n}$ with
the natural Euclidean metric
\begin{equation}\label{eq:ju6t}
|u|:=\Bigl\{\sum_I|u_I|^2\Bigr\}^{1/2}\,\,\mbox{ for each
}\,\,u=\sum_Iu_Ie_I\in{\mathcal{C}}\!\ell_{n}.
\end{equation}
The Clifford conjugation on ${\mathcal{C}}\!\ell_{n}$, denoted by
`bar', is defined as the unique real-linear involution on
${\mathcal{C}}\!\ell_{n}$ for which
$\overline{e_I}e_I=e_I\overline{e_I}=1$ for any multi-index $I$.
More specifically, given $u=\sum_Iu_Ie_I\in{\mathcal{C}}\!\ell_{n}$
we set $\overline{u}:=\sum_Iu_I\overline{e_I}$ where, for each
$I=(i_1,i_2,\dots,i_l)$ with $1\leq i_1<i_2<\dots<i_l\leq n$,
\begin{equation}\label{eq:i654}
\overline{e_I}=(-1)^l e_{i_l}\odot e_{i_{l-1}}\odot\cdots\odot
e_{i_1}.
\end{equation}
Let us also define the scalar part of
$u=\sum_Iu_Ie_I\in{\mathcal{C}}\!\ell_{n}$ as $u_0:=u_{\varnothing}$,
and endow ${\mathcal{C}}\!\ell_{n}$ with the natural Hilbert space
structure
\begin{eqnarray}\label{Hil-1}
\langle u,v\rangle:=\sum_I u_Iv_I,\qquad\mbox{if }\,\,
u=\sum_Iu_Ie_I,\,\,v=\sum_Iv_Ie_I\in{\mathcal{C}}\!\ell_{n}.
\end{eqnarray}
It follows directly from definitions that
\begin{eqnarray}\label{Hil-aabV}
\overline{x}=-x\,\,\mbox{ for each
}\,\,x\in{\mathbb{R}}^{n}\hookrightarrow{\mathcal{C}}\!\ell_{n},
\end{eqnarray}
and other properties are collected in the lemma below.

\begin{lemma}\label{i6gf65}
For any $u,v\in{\mathcal{C}}\!\ell_{n}$ one has
\begin{align}\label{Hil-2}
&|u|^2=(u\odot\overline{u})_0=(\overline{u}\odot u)_0,
\\[4pt]
&\langle u,v\rangle=(u\odot\overline{v})_0=(\overline{u}\odot v)_0,
\label{Hil-3bb}
\\[4pt]
&\overline{u\odot v}=\overline{v}\odot\overline{u},
\label{Hil-5}
\\[4pt]
&|\overline{u}|=|u|,
\label{Hil-5bb}
\\[4pt]
&|u\odot v|\leq 2^{n/2}|u||v|,
\label{eq:rtta}
\end{align}
and
\begin{equation}\label{eq:rtt}
|u\odot v|=|u||v|\,\,\mbox{ if either $u$ or $v$ belongs to
${\mathbb{R}}^n \hookrightarrow{\mathcal{C}}\!\ell_{n}$}.
\end{equation}
\end{lemma}

\begin{proof}
Properties \eqref{Hil-2}-\eqref{Hil-5bb} are straightforward
consequences of definitions. To justify \eqref{eq:rtta}, assume
$u=\sum_Iu_Ie_I\in{\mathcal{C}}\!\ell_{n}$ and
$v=\sum_Jv_Je_J\in{\mathcal{C}}\!\ell_{n}$ have been given. Then
\begin{align}\label{eq:utf4}
|u\odot v| &=\Big|\sum_I\Big(\sum_Ju_Iv_J e_I\odot e_J\Big)\Big|
\leq\sum_I\Big|\sum_Ju_Iv_J e_I\odot e_J\Big|
\nonumber\\[4pt]
&=\sum_I\Big(\sum_J|u_Iv_J|^2\Big)^{1/2}
=|v|\sum_I|u_I|\leq|v|\Big(\sum_I|u_I|^2\Big)^{1/2}\Big(\sum_I
1\Big)^{1/2}
\nonumber\\[4pt]
&=2^{n/2}|u||v|.
\end{align}
Above, the triangle inequality has been employed in the second step.
The third step relies on \eqref{eq:ju6t} and the observation that,
for each $I$ fixed, the family of Clifford algebra elements
$\{e_I\odot e_J\}_J$ coincides modulo signs with the orthonormal
basis $\{e_K\}_K$. The penultimate step is the discrete
Cauchy-Schwarz inequality.

As regards \eqref{eq:rtt}, assume that
$v\in{\mathbb{R}}^n\hookrightarrow{\mathcal{C}}\!\ell_{n}$ and write
\begin{align}\label{eq:utf4DV}
|u\odot v|^2 & =\big((u\odot v)\odot\overline{u\odot v}\big)_0
=\big(u\odot (v\odot\overline{v})\odot\overline{u}\big)_0
\nonumber\\[4pt]
&=|v|^2(u\odot\overline{u})_0=|u|^2|v|^2,
\end{align}
by \eqref{Hil-2}, \eqref{Hil-5}, \eqref{Hil-aabV}, and
\eqref{X-sqr}. Finally, the case when
$u\in{\mathbb{R}}^n\hookrightarrow{\mathcal{C}}\!\ell_{n}$ follows
from what we have just proved with the help of \eqref{Hil-5} and
\eqref{Hil-5bb}.
\end{proof}

Next, recall the {\tt Dirac} {\tt operator}
\begin{equation}\label{Dirac}
D:=\sum_{j=1}^{n}e_j\partial_j.
\end{equation}
In the sequel, we shall use $D_L$ and $D_R$ to denote the action of
$D$ on a ${\mathscr{C}}^1$ function
$u:\Omega\to{\mathcal{C}}\!\ell_{n}$ (where $\Omega$ is an open
subset of ${\mathbb{R}}^{n}$) from the left and from the right,
respectively. For a sufficiently nice domain $\Omega$ with outward
unit normal $\nu=(\nu_1,\dots,\nu_{n})$ (identified with the
${\mathcal{C}}\!\ell_{n}$-valued function
$\nu=\sum_{j=1}^{n}\nu_je_j$) and surface measure $\sigma$, and for
any two reasonable ${\mathcal{C}}\!\ell_{n}$-valued functions $u,v$
in $\Omega$, the following integration by parts formula holds:
\begin{align}\label{D-IBP}
&\int_{\partial\Omega}u(x)\odot\nu(x)\odot v(x)\,d\sigma(x)
\nonumber\\[4pt]
&\qquad\quad =\int_{\Omega}\Bigl\{(D_R u)(x)\odot v(x)+u(x)\odot(D_L
v)(x)\Bigr\}\,dx.
\end{align}
More detailed accounts of these and related matters can be found in
\cite{BDS} and \cite{Mi}. In general, if
$\big({\mathscr{X}},\|\cdot\|_{\mathscr{X}}\big)$ is a Banach space
then by ${\mathscr{X}}\otimes{\mathcal{C}}\!\ell_{n}$ we shall
denote the Banach space consisting of elements of the form
\begin{equation}\label{eq4.1agg}
u=\sum_{l=0}^{n}{\sum_{|I|=l}}'u_I\,e_I,\quad u_I\in{\mathscr{X}},
\end{equation}
equipped with the natural norm
\begin{equation}\label{eq4.1agg.2}
\|u\|_{{\mathscr{X}}\otimes{\mathcal{C}}\!\ell_{n}}
:=\sum_{l=0}^{n}{\sum_{|I|=l}}'\|u_I\|_{\mathscr{X}}.
\end{equation}
A simple but useful observation in this context is that
\begin{equation}\label{M-nu}
\begin{array}{l}
\mbox{if $\Omega\subset{\mathbb{R}}^n$ is a domain of class
${\mathscr{C}}^{1+\alpha}$ for some $\alpha\in(0,1)$ then}
\\[4pt]
\nu\odot:{\mathscr{C}}^\alpha(\partial\Omega)\otimes{\mathcal{C}}\!\ell_{n}\longrightarrow
{\mathscr{C}}^\alpha(\partial\Omega)\otimes{\mathcal{C}}\!\ell_{n}\quad\mbox{is
an isomorphism}
\\[4pt]
\mbox{whose norm and the norm of its inverse are $\leq
2\|\nu\|_{{\mathscr{C}}^\alpha(\partial\Omega)}$}.
\end{array}
\end{equation}
Indeed, by \eqref{X-sqr}, its inverse is $-\nu\odot$ and the
aforementioned norm estimates are simple consequences of
\eqref{eq:rtt}, bearing in mind that
$\|\nu\|_{{\mathscr{C}}^\alpha(\partial\Omega)}\geq 1$.

For each $s\in\{1,\dots,n\}$ we let $[\,\cdot\,]_s$ denote the
projection onto the $s$-th Euclidean coordinate, i.e., $[x]_s:=x_s$
if $x=(x_1,\dots,x_{n})\in{\mathbb{R}}^{n}$. The following lemma, in
the spirit of work of Semmes in \cite{Se}, will
play an important role for us.

\begin{lemma}\label{L-Semmes}
For any odd, harmonic, homogeneous polynomial $P(x)$,
$x\in{\mathbb{R}}^{n}$ {\rm (}with $n\geq 2${\rm )}, of degree
$l\geq 3$, there exist a family $P_{rs}(x)$, $1\leq r,s\leq n$, of
harmonic, homogeneous polynomials of degree $l-2$, as well as a
family of odd, ${\mathscr{C}}^\infty$ functions
\begin{equation}\label{eq:klj}
k_{rs}:{\mathbb{R}}^{n}\setminus\{0\}\longrightarrow{\mathbb{R}}^{n}
\hookrightarrow{\mathcal{C}}\!\ell_{n},\qquad 1\leq r,s\leq n,
\end{equation}
which are homogeneous of degree $-(n-1)$, and for each
$x\in{\mathbb{R}}^{n}\setminus\{0\}$ satisfy
\begin{align}\label{pro-1}
& \frac{P(x)}{|x|^{n-1+l}}=\sum_{r,s=1}^{n}[k_{rs}(x)]_s\,\,\,\mbox{
and}
\\[4pt]
& (D_Rk_{rs})(x)=\frac{l-1}{n+l-3}\frac{\partial}{\partial x_r}
\left(\frac{P_{rs}(x)}{|x|^{n+l-3}}\right),\quad 1\leq r,s\leq n.
\label{pro-2}
\end{align}
Moreover, there exists a finite dimensional constant $c_n>0$ such
that
\begin{equation}\label{eq:Esfag}
\max_{1\leq r,s\leq n}\|k_{rs}\|_{L^\infty(S^{n-1})} +\max_{1\leq
r,s\leq n}\|\nabla k_{rs}\|_{L^\infty(S^{n-1})} \leq
c_n\,2^{l}\|P\|_{L^1(S^{n-1})}.
\end{equation}
\end{lemma}

\begin{proof}
Given now an odd, harmonic, homogeneous polynomial $P(x)$ of degree
$l\geq 3$ in ${\mathbb{R}}^{n}$, for $r,s\in\{1,\dots,n\}$ introduce
\begin{equation}\label{eq:yre36}
P_{rs}(x):=\frac{1}{l(l-1)}(\partial_r\partial_sP)(x),\qquad\forall\,x\in{\mathbb{R}}^n.
\end{equation}
Then each $P_{rs}$ is an odd, harmonic, homogeneous polynomial of
degree $l-2$ in ${\mathbb{R}}^{n}$, and Euler's formula for
homogeneous functions gives
\begin{equation}\label{eq:yre37}
P(x)=\sum_{r,s=1}^n x_r x_s
P_{rs}(x),\qquad\forall\,x\in{\mathbb{R}}^n,
\end{equation}
and, for each $r,s\in\{1,\dots,n\}$,
\begin{equation}\label{eq:yre37BG}
\langle(\nabla
P_{rs})(x),x\rangle=(l-2)P_{rs}(x),\qquad\forall\,x\in{\mathbb{R}}^n.
\end{equation}

To proceed, assume first that $n\geq 3$ and, for each
$r,s\in\{1,\dots,n\}$, define the function
$k_{rs}:{\mathbb{R}}^{n}\setminus\{0\}\longrightarrow{\mathbb{R}}^{n}
\hookrightarrow{\mathcal{C}}\!\ell_{n}$ by setting
\begin{equation}\label{eq:yre38}
k_{rs}(x):=\frac{1}{(n+l-3)(n+l-5)}\sum_{j=1}^n
\partial_r\partial_j\Big(\frac{P_{rs}(x)}{|x|^{n+l-5}}\Big)e_j,
\qquad\forall\,x\in{\mathbb{R}}^{n}\setminus\{0\}.
\end{equation}
The fact that $n,l\geq 3$ ensure that both $n+l-3\not=0$ and
$n+l-5\not=0$ so each $k_{rs}$ is well-defined, odd,
${\mathscr{C}}^\infty$ and homogeneous of degree $-(n-1)$ in
${\mathbb{R}}^{n}\setminus\{0\}$. In addition,
\begin{equation}\label{eq:yre38-JJ}
k_{rs}(x)=\frac{1}{(n+l-3)(n+l-5)}
D_R\Big[\partial_r\Big(\frac{P_{rs}(x)}{|x|^{n+l-5}}\Big)\Big],
\qquad\forall\,x\in{\mathbb{R}}^{n}\setminus\{0\},
\end{equation}
hence for all $x\in{\mathbb{R}}^{n}\setminus\{0\}$ we may write
\begin{align}\label{eq:yre39}
(D_Rk_{rs})(x) &=\frac{1}{(n+l-3)(n+l-5)}
D^2_R\Big[\partial_r\Big(\frac{P_{rs}(x)}{|x|^{n+l-5}}\Big)\Big]
\nonumber\\[4pt]
&=\frac{-1}{(n+l-3)(n+l-5)}
\Delta\Big[\partial_r\Big(\frac{P_{rs}(x)}{|x|^{n+l-5}}\Big)\Big]
\nonumber\\[4pt]
&=:I+II+III,
\end{align}
where
\begin{align}\label{eq:yre39-C}
I &:=\frac{-1}{(n+l-3)(n+l-5)}\partial_r\Big[\frac{(\Delta
P_{rs})(x)}{|x|^{n+l-5}}\Big]=0,
\nonumber\\[4pt]
II &:=\frac{-1}{(n+l-3)(n+l-5)}\partial_r \Big[2\big\langle(\nabla
P_{rs})(x),\nabla\big[|x|^{-(n+l-5)}\big]\big\rangle\Big]
\nonumber\\[4pt]
&\,\,=\frac{2}{n+l-3}\partial_r \Big[\frac{\big\langle(\nabla
P_{rs})(x),x\big\rangle}{|x|^{n+l-3}}\Big]
\nonumber\\[4pt]
&\,\,=\frac{2(l-2)}{n+l-3}\partial_r\Big[\frac{P_{rs}(x)}{|x|^{n+l-3}}\Big],
\nonumber\\[4pt]
III &:=\frac{-1}{(n+l-3)(n+l-5)}\partial_r
\Big[P_{rs}(x)\Delta\big[|x|^{-(n+l-5)}\big]\Big]
\nonumber\\[4pt]
&\,\,=\frac{-l+3}{n+l-3}\partial_r\Big[\frac{P_{rs}(x)}{|x|^{n+l-3}}\Big],
\end{align}
by the harmonicity of $P$, \eqref{eq:yre37BG}, and straightforward
algebra. This proves that \eqref{pro-1} holds when $n\geq 3$. Going
further, from \eqref{eq:yre38} and the fact that
\begin{equation}\label{eq:feds}
\sum_{r=1}^{n}(\partial_r P_{rs})(x)=\sum_{s=1}^{n}(\partial_s
P_{rs})(x)=0 \,\,\mbox{ and }\,\,\sum_{r=1}^{n}P_{rr}(x)=0
\end{equation}
(as seen from \eqref{eq:yre36} and the harmonicity of $P$), we
deduce that for each $x\in{\mathbb{R}}^{n}\setminus\{0\}$
\begin{align}\label{pro-1bVV}
\sum_{r,s=1}^{n}[k_{rs}(x)]_s &=\frac{1}{(n+l-3)(n+l-5)}
\sum_{r,s=1}^{n}\partial_r\partial_s\Big(\frac{P_{rs}(x)}{|x|^{n+l-5}}\Big)
\nonumber\\[4pt]
&=\frac{1}{(n+l-3)(n+l-5)}
\sum_{r,s=1}^{n}P_{rs}(x)\partial_r\partial_s\big[|x|^{-(n+l-5)}\big]
\nonumber\\[4pt]
&=\frac{-1}{n+l-3}\sum_{r,s=1}^{n}P_{rs}(x)\Big\{
\frac{\delta_{rs}}{|x|^{n+l-3}}-(n+l-3)\frac{x_r
x_s}{|x|^{n+l-1}}\Big\}
\nonumber\\[4pt]
&=\frac{P(x)}{|x|^{n-1+l}}.
\end{align}
This establishes \eqref{pro-2} for $n\geq 3$. Moving on, for each
$\gamma\in{\mathbb{N}}_0^n$, interior estimates for the harmonic
function $P$ give
\begin{align}\label{eq:rFG}
\|\partial^\gamma P\|_{L^\infty(S^{n-1})} & \leq
c_{n,\gamma}\int_{B(0,2)}|P(x)|\,dx
=c_{n,\gamma}\int_{S^{n-1}}|P(\omega)|\Big(\int_0^2
r^{n-1+l}\,dr\Big)\,d\omega
\nonumber\\[4pt]
&=c_{n,\gamma}\frac{2^l}{n+l}\|P\|_{L^1(S^{n-1})},
\end{align}
where we have also used the fact that $P$ is homogeneous of degree
$l$. The estimates in \eqref{eq:Esfag} now readily follow on account
of \eqref{eq:yre38}, \eqref{eq:yre36}, and \eqref{eq:rFG}.

To treat the two-dimensional case, first we observe that if $Q_m(x)$
is an arbitrary homogeneous polynomial of degree
$m\in{\mathbb{N}}_0$ in ${\mathbb{R}}^n$ with $n\geq 2$ and
$\lambda>0$ then
\begin{equation}\label{eq:yre34-PP}
\frac{Q_m(x)}{|x|^{n+m-\lambda}}\,\,\mbox{ is a tempered
distribution in ${\mathbb{R}}^n$}.
\end{equation}
If, in addition, $Q_m(x)$ is harmonic and $\lambda<n$ then (cf.
\cite[p.\,73]{St70}) also
\begin{equation}\label{eq:yre34}
{\mathcal{F}}_{x\to\xi}\Big(\frac{Q_m(x)}{|x|^{n+m-\lambda}}\Big)
=\gamma_{n,m,\lambda}\,\frac{Q_m(\xi)}{|\xi|^{m+\lambda}} \,\,\mbox{
as tempered distributions in ${\mathbb{R}}^n$},
\end{equation}
where ${\mathcal{F}}_{x\to\xi}$ is an alternative notation for the
Fourier transform in ${\mathbb{R}}^n$ from \eqref{eq:iytTT} and
\begin{equation}\label{eq:yre35}
\gamma_{n,m,\lambda}:=(-1)^{3m/2}\pi^{n/2}2^{\lambda}
\frac{\Gamma(m/2+\lambda/2)}{\Gamma(m/2+n/2-\lambda/2)}.
\end{equation}
Pick now an odd, harmonic, homogeneous polynomial $P(x)$ of degree
$l\geq 3$ in ${\mathbb{R}}^{2}$ and define $P_{rs}$ for
$r,s\in\{1,\dots,n\}$ as in \eqref{eq:yre36}. Hence, once again each
$P_{rs}$ is an odd, harmonic, homogeneous polynomial of degree $l-2$
in ${\mathbb{R}}^{2}$, and \eqref{eq:yre37} holds. Moreover,
\eqref{eq:yre34} used for $n=2$, $m=l-2$, $\lambda=1$, and
$Q_m=P_{rs}$ yields
\begin{equation}\label{eq:yre34b}
\frac{P_{rs}(x)}{|x|^{l-1}}=-(-1)^{3l/2}2\pi\,
{\mathcal{F}}^{-1}_{\xi\to
x}\Big(\frac{P_{rs}(\xi)}{|\xi|^{l-1}}\Big).
\end{equation}

Now, for each $r,s\in\{1,2\}$ define the function
$k_{rs}:{\mathbb{R}}^{2}\setminus\{0\}\longrightarrow{\mathbb{R}}^{2}
\hookrightarrow{\mathcal{C}}\!\ell_{2}$ by setting
\begin{equation}\label{eq:yre38-REP}
k_{rs}(x):=(-1)^{3l/2}2\pi\sum_{j=1}^2 {\mathcal{F}}^{-1}_{\xi\to
x}\Big(\xi_r\xi_j\frac{P_{rs}(\xi)}{|\xi|^{l+1}}\Big)e_j,
\qquad\forall\,x\in{\mathbb{R}}^{2}\setminus\{0\}.
\end{equation}
By \eqref{eq:yre34-PP} used with $n=2$, $m=l$, $\lambda=1$, and
$Q_m(\xi)=\xi_r\xi_jP_{rs}(\xi)$, it follows that
$\xi_r\xi_j\frac{P_{rs}(\xi)}{|\xi|^{l+1}}$ is a tempered
distribution in ${\mathbb{R}}^2$. Consequently, $k_{rs}$ in
\eqref{eq:yre38-REP} is meaningfully defined and, from
\cite[Proposition~4.58, p.\,132]{DM}, we deduce that
$k_{rs}\in{\mathscr{C}}^\infty({\mathbb{R}}^{2}\setminus\{0\})$.
Also, based on standard properties of the Fourier transform (cf.,
e.g., \cite[Chapter~4]{DM}) it follows that $k_{rs}$ is odd and
homogeneous of degree $-1$ in ${\mathbb{R}}^{2}\setminus\{0\}$. In
addition,
\begin{align}\label{eq:yre39BG}
(D_Rk_{rs})(x) &=(-1)^{3l/2}2\pi\sum_{\ell,j=1}^2
\partial_{x_\ell}{\mathcal{F}}^{-1}_{\xi\to x}\Big(\xi_r\xi_j\frac{P_{rs}(\xi)}{|\xi|^{l+1}}\Big)
e_j\odot e_\ell
\nonumber\\[4pt]
&=\sqrt{-1}(-1)^{3l/2}2\pi\sum_{\ell,j=1}^2
{\mathcal{F}}^{-1}_{\xi\to
x}\Big(\xi_r\xi_j\xi_\ell\frac{P_{rs}(\xi)}{|\xi|^{l+1}}\Big)
e_j\odot e_\ell=:I+II,
\end{align}
where $I$ and $II$ are the pieces produced by summing up over
$j=\ell$ and $j\not=\ell$, respectively. Since in the latter
scenario $\xi_\ell\,\xi_j=\xi_j\xi_\ell$ while $e_j\odot
e_\ell=-e_\ell\odot e_j$ it follows that $II=0$. Given that
$e_j\odot e_j=-1$ for each $j\in\{1,2\}$, we conclude that
\begin{align}\label{eq:yre40}
(D_Rk_{rs})(x) &=-\sqrt{-1}(-1)^{3l/2}2\pi\sum_{j=1}^2
{\mathcal{F}}^{-1}_{\xi\to
x}\Big(\xi_r\xi_j^2\frac{P_{rs}(\xi)}{|\xi|^{l+1}}\Big)
\nonumber\\[4pt]
&=-\sqrt{-1}(-1)^{3l/2}2\pi\, {\mathcal{F}}^{-1}_{\xi\to
x}\Big(\xi_r\frac{P_{rs}(\xi)}{|\xi|^{l-1}}\Big)
\nonumber\\[4pt]
&=-(-1)^{3l/2}2\pi\,\partial_{x_r}\Big[ {\mathcal{F}}^{-1}_{\xi\to
x}\Big(\frac{P_{rs}(\xi)}{|\xi|^{l-1}}\Big)\Big]
=\partial_{x_r}\Big[\frac{P_{rs}(x)}{|x|^{l-1}}\Big],
\end{align}
where the last step uses \eqref{eq:yre34b}. Hence, \eqref{pro-1}
holds when $n=2$. Finally, from \eqref{eq:yre38}, \eqref{eq:yre37},
and \eqref{eq:yre34} (used for $P$) we deduce that for each
$x\in{\mathbb{R}}^{2}\setminus\{0\}$ we have
\begin{align}\label{pro-1bVVbg}
\sum_{r,s=1}^{2}[k_{rs}(x)]_s &=(-1)^{3l/2}2\pi\sum_{r,s=1}^{2}
{\mathcal{F}}^{-1}_{\xi\to
x}\Big(\xi_r\xi_s\frac{P_{rs}(\xi)}{|\xi|^{l+1}}\Big)
\nonumber\\[4pt]
&=(-1)^{3l/2}2\pi\,{\mathcal{F}}^{-1}_{\xi\to
x}\Big(\frac{P(\xi)}{|\xi|^{l+1}}\Big) =\frac{P(x)}{|x|^{l+1}}.
\end{align}
This establishes \eqref{pro-2} when $n=2$.

At this stage, there remains to justify \eqref{eq:Esfag} in the case
$n=2$. To this end, pick
$\psi\in{\mathscr{C}}^\infty_0({\mathbb{R}}^2)$ with $0\leq\psi\leq
1$, $\psi=1$ on $B(0,1)$ and $\psi=0$ on
${\mathbb{R}}^2\setminus\overline{B(0,2)}$. Fix $r,s,j\in\{1,2\}$
and abbreviate $u(\xi):=\xi_r\xi_j P_{rs}(\xi)/|\xi|^{l+1}$ for
$\xi\in{\mathbb{R}}^2\setminus\{0\}$. Then $u$ is locally integrable
and defines a tempered distribution in ${\mathbb{R}}^2$. Hence, for
each $\alpha\in{\mathbb{N}}_0^2$ with $|\alpha|=2$ and
$\xi\in\overline{B(0,1)}$ we may write
\begin{align}\label{eq:11yr43e}
\big|{\mathcal{F}}_{x\to\xi}\big(\psi(x)\partial^\alpha
u(x)\big)\big| &=\big|\big\langle\psi\partial^\alpha
u,e^{-i\langle\xi,\cdot\rangle}\big\rangle\big| =\big|\big\langle
u,\partial^\alpha\big(\psi e^{-i\langle\xi,\cdot\rangle}\big)
\big\rangle\big|
\\[4pt]
&\leq C\int_{B(0,2)}|u(x)|\,dx\leq
C\int_{S^1}|P_{rs}(\omega)|\,d\omega \leq C2^l\|P\|_{L^1(S^1)},
\nonumber
\end{align}
and
\begin{align}\label{eq:11yr43e.2}
\big|{\mathcal{F}}_{x\to\xi}\big((1-\psi(x))\partial^\alpha
u(x)\big)\big| &\leq \|(1-\psi)\partial^\alpha
u\|_{L^1({\mathbb{R}}^2)} \leq\int_{{\mathbb{R}}^2\setminus
B(0,1)}|\partial^\alpha u(x)|\,dx
\nonumber\\[4pt]
&\leq C\int_{S^1}|\partial^\alpha u(\omega)|\,d\omega \leq
C2^l\|P\|_{L^1(S^1)}.
\end{align}
Collectively, \eqref{eq:11yr43e} and \eqref{eq:11yr43e.2} give that,
for each $\alpha\in{\mathbb{N}}_0^2$ with $|\alpha|=2$ and
$\xi\in\overline{B(0,1)}$,
\begin{align}\label{eq:11yr43e.3}
\big|{\mathcal{F}}_{x\to\xi}\big(\partial^\alpha u(x)\big)\big|
&\leq\big|{\mathcal{F}}_{x\to\xi}\big(\psi(x)\partial^\alpha
u(x)\big)\big|
+\big|{\mathcal{F}}_{x\to\xi}\big((1-\psi(x))\partial^\alpha
u(x)\big)\big|
\nonumber\\[4pt]
&\leq C2^l\|P\|_{L^1(S^1)},
\end{align}
hence for each $\xi\in\overline{B(0,1)}$ we have
\begin{align}\label{eq:11yr43e.4}
|\xi|^2\big|\widehat{u}(\xi)\big|=\sum_{\ell=1}^2\big|\xi_\ell^2\widehat{u}(\xi)\big|
=\sum_{\ell=1}^2\big|{\mathcal{F}}_{x\to\xi}\big(\partial_\ell^2
u(x)\big)\big| \leq C2^l\|P\|_{L^1(S^1)}.
\end{align}
In particular, $\|k_{rs}\|_{L^\infty(S^1)}\leq
C\sup_{|\xi|=1}\big|\widehat{u}(\xi)\big| \leq
C2^l\|P\|_{L^1(S^1)}$. A similar circle of ideas also yields
$\|\nabla k_{rs}\|_{L^\infty(S^1)}\leq C2^l\|P\|_{L^1(S^1)}$. This
proves \eqref{eq:Esfag} in the case $n=2$ and completes the proof of
the lemma.
\end{proof}

\section{Cauchy-Clifford Operators on H\"older Spaces}
\setcounter{equation}{0} \label{S-5}

Let $\Omega\subset{\mathbb{R}}^{n}$ be a set of locally finite
perimeter satisfying \eqref{Tay-1}. As before, we shall denote by
$\nu=(\nu_1,\dots,\nu_n)$ the outward unit normal to $\Omega$ and by
$\sigma:={\mathcal{H}}^{n-1}\lfloor\,\partial\Omega$ the surface
measure on $\partial\Omega$. Then the (boundary-to-domain) {\tt
Cauchy-Clifford} {\tt operator} and its principal value (or,
boundary-to-boundary) version associated with $\Omega$ are,
respectively, given by
\begin{eqnarray}\label{Cau-C1}
{\mathcal{C}}f(x):=\frac{1}{\omega_{n-1}}\int_{\partial\Omega}
\frac{x-y}{|x-y|^n}\odot\nu(y)\odot f(y)\,d\sigma(y),\qquad
x\in\Omega,
\end{eqnarray}
and
\begin{eqnarray}\label{Cau-C2}
{\mathcal{C}}^{{}^{\rm pv}}\!f(x):=\lim_{\varepsilon\to 0^+}\frac{1}{\omega_{n-1}}
\int\limits_{\stackrel{y\in\partial\Omega}{|x-y|>\varepsilon}}
\frac{x-y}{|x-y|^n}\odot\nu(y)\odot f(y)\,d\sigma(y),\qquad
x\in\partial\Omega,
\end{eqnarray}
where $f$ is a ${\mathcal{C}}\!\ell_{n}$-valued function defined on
$\partial\Omega$. At the present time, these definitions are
informal as more conditions need to be imposed on the function $f$
and the underlying domain $\Omega$ in order to ensure that these
operators are well-defined and enjoy desirable properties in various
settings of interest. We start by recording the following result, in
the context of uniformly rectifiable domains.

\begin{proposition}\label{iTRfc}
Let $\Omega\subset{\mathbb{R}}^n$ be a {\rm UR} domain. Then for
every $f\in
L^p(\partial\Omega,\sigma)\otimes{\mathcal{C}}\!\ell_{n}$ with
$p\in[1,\infty)$, the function ${\mathcal{C}}^{{}^{\rm pv}}\!f$ is meaningfully defined
$\sigma$-a.e. on $\partial\Omega$, and the actions of the two
Cauchy-Clifford operators on $f$ are related via the boundary
behavior
\begin{equation}\label{2.3.21}
\Big(\mathcal{C}f\Big|^{{}^{\rm n.t.}}_{\partial\Omega}\Big)(x)
:=\lim\limits_{\Gamma_\kappa(x)\ni z\rightarrow x}\,\mathcal{C}f(z)
=\big(\tfrac{1}{2}I+{\mathcal{C}}^{{}^{\rm pv}}\!\big)f(x),\quad\sigma\text{-a.e.
}\,x\in\partial\Omega,
\end{equation}
where $I$ is the identity operator. Moreover, for each
$p\in(1,\infty)$, there exists a finite constant $M=M(n,p,\Omega)>0$
such that
\begin{eqnarray}\label{T-Har.BCca}
\|{\mathcal{N}}({\mathcal{C}}f)\|_{L^p(\partial\Omega,\sigma)}
\leq M\,\|f\|_{L^p(\partial\Omega,\sigma)\otimes{\mathcal{C}}\!\ell_{n}},
\end{eqnarray}
the operator ${\mathcal{C}}^{{}^{\rm pv}}\!$ is well-defined and bounded on
$L^p(\partial\Omega,\sigma)\otimes{\mathcal{C}}\!\ell_{n}$, and the formula
\begin{equation}\label{eq:CsquaTT}
\big({\mathcal{C}}^{{}^{\rm pv}}\big)^2=\tfrac{1}{4}I\,\,\,\mbox{ on }\,\,\,
L^p(\partial\Omega,\sigma)\otimes{\mathcal{C}}\!\ell_{n}
\end{equation}
holds.
\end{proposition}

\begin{proof}
With the exception of \eqref{eq:CsquaTT} (which has been proved in
\cite{HoMiTa10}; cf. also \cite{IMiMiTa} for very general results of
this type), all claims follow from Theorems~\ref{Main-T2-BBBB}-\ref{Main-T2}.
\end{proof}

The goal in this section is to prove similar results when the
Lebesgue scale is replaced by H\"older spaces, in a class of domains
considerably more general than the category of uniformly rectifiable
domains. We begin by proving the following result.

\begin{lemma}\label{urdyyy7}
Let $\Omega\subseteq{\mathbb{R}}^n$ be a Lebesgue measurable set
whose boundary is compact and upper Ahlfors regular {\rm (}hence, in
particular, $\Omega$ is of locally finite perimeter by
\eqref{eq:hdusd}{\rm )}. Denote by $\nu$ the geometric measure
theoretic outward unit normal to $\Omega$ and define
$\sigma:={\mathcal{H}}^{n-1}\lfloor\partial_\ast\Omega$. Then there
exists a number $N=N(n,c)\in(0,\infty)$, depending only on the
dimension $n$ and the upper Ahlfors regularity constant $c$ of
$\partial\Omega$, with the property that
\begin{align}\label{Cau-CJiigv}
\Big|\int_{\partial_\ast\Omega\setminus
B(x,\,r)}\frac{x-y}{|x-y|^n}\odot\nu(y)\,d\sigma(y)\Big| \leq
N,\qquad\forall\,x\in{\mathbb{R}}^n,\quad\forall\,r\in(0,\infty).
\end{align}
\end{lemma}

\begin{proof}
We shall first show that, whenever $\Omega\subseteq{\mathbb{R}}^n$
is a bounded set of locally finite perimeter, having fixed an
arbitrary $x\in{\mathbb{R}}^n$, for ${\mathscr{L}}^1$-a.e.
$\varepsilon>0$ we have
\begin{align}\label{Cau-C2bff}
\int_{\partial_\ast\Omega\setminus
B(x,\,\varepsilon)}\frac{x-y}{|x-y|^n}\odot\nu(y)\,d\sigma(y)
&=\int_{\Omega\cap\partial
B(x,\,\varepsilon)}\frac{x-y}{|x-y|^n}\odot\nu(y)\,d{\mathcal{H}}^{n-1}(y)
\nonumber\\[4pt]
&=\frac{{\mathcal{H}}^{n-1}\big(\Omega\cap\partial
B(x,\,\varepsilon)\big)}{\varepsilon^{n-1}}.
\end{align}
To justify this claim, we start by noting that the second equality
(which holds for any measurable set $\Omega\subset{\mathbb{R}}^n$)
is an immediate consequence of the fact that
\begin{equation}\label{eigrrTY}
y\in\partial B(x,\varepsilon)\,\,\mbox{ implies }\,\,
(x-y)\odot\nu(y)=(x-y)\odot(y-x)/\varepsilon=\varepsilon.
\end{equation}
As regards the first equality in \eqref{Cau-C2bff}, for each
$j,k\in\{1,\dots,n\}$ consider the vector field
\begin{equation}\label{eq:YYTf}
\vec{F}_{jk}(y):=\Big(0,\dots,0,\frac{x_j-y_j}{|x-y|^n},0,\dots,0\Big),\qquad
\forall\,y\in{\mathbb{R}}^n\setminus\{x\},
\end{equation}
with the non-zero component on the $k$-th slot. Thus, we have
$\vec{F}_{jk}\in{\mathscr{C}}^1({\mathbb{R}}^n\setminus\{x\},{\mathbb{R}}^n)$
and, if $E_{\!{}_\Delta}$ stands for the standard fundamental
solution for the Laplacian $\Delta=\partial_1^2+\cdots\partial_n^2$
in ${\mathbb{R}}^n$, given by
\begin{equation}\label{RacvTT}
E_{\!{}_\Delta}\!(x):=
\begin{cases}
\displaystyle \frac{1}{\omega_{n-1}(2-n)}\frac{1}{|x|^{n-2}} &
\textrm{if $n\geq 3$},
\\[15pt]
\displaystyle \frac{1}{2\pi}\ln|x| & \textrm{if $n=2$,}
\end{cases}
\qquad\forall\,x\in\mathbb{R}^n\setminus\{0\},
\end{equation}
then
\begin{equation}\label{eq:u54955}
({\rm div}\,\vec{F}_{jk})(y)=-\omega_{n-1}(\partial_j\partial_k E_{\!{}_\Delta})(x-y),
\qquad\forall\,y\in{\mathbb{R}}^n\setminus\{x\}.
\end{equation}
As a consequence, in ${\mathbb{R}}^n\setminus\{x\}$ we have
\begin{align}\label{eq:itf5f6}
\sum_{j,k=1}^n({\rm div}\vec{F}_{jk})e_j\odot e_k &=\sum_{1\leq
j\not=k\leq n}({\rm div}\,\vec{F}_{jk})e_j\odot e_k-\sum_{j=1}^n{\rm
div}\,\vec{F}_{jj}
\nonumber\\[4pt]
&=-\omega_{n-1}\sum_{1\leq j\not=k\leq n}
(\partial_j\partial_k E_{\!{}_\Delta})(x-\cdot)e_j\odot e_k
+\omega_{n-1}(\Delta E_{\!{}_\Delta})(x-\cdot)
\nonumber\\[4pt]
&=0,
\end{align}
using the fact that $e_j\odot e_k=-e_k\odot e_j$ for $j\not=k$ and
the harmonicity of $E_{\!{}_\Delta}\!(x-\cdot)$ in
${\mathbb{R}}^n\setminus\{x\}$.

At this stage, fix an arbitrary $\varepsilon_o\in(0,\infty)$ and
alter each $\vec{F}_{jk}$ both inside $B(x,\varepsilon_o)$ and
outside an open neighborhood of $\overline{\Omega}$ to a vector
field
$\vec{G}_{jk}\in{\mathscr{C}}^1_0({\mathbb{R}}^n,{\mathbb{R}}^n)$
(this is possible given the working assumption that $\Omega$ is
bounded). Then for ${\mathscr{L}}^1$-a.e.
$\varepsilon\in(\varepsilon_o,\infty)$ based on the formula
\eqref{eq:TgavFF} used for $\vec{F}:=\vec{G}_{jk}$, $D:=\Omega$, and
$r:=\varepsilon$ we may write
\begin{align}\label{eq:Tgahytttr}
0 &=\sum_{j,k=1}^n\Big(\int_{\Omega\setminus B(x,\,\varepsilon)}
{\rm div}\vec{F}_{jk}\,d{\mathscr{L}}^n\Big)e_j\odot e_k
=\sum_{j,k=1}^n\Big(\int_{\Omega\setminus B(x,\,\varepsilon)} {\rm
div}\,\vec{G}_{jk}\,d{\mathscr{L}}^n\Big)e_j\odot e_k
\nonumber\\[4pt]
&=\sum_{j,k=1}^n\Big(\int_{\partial_\ast\Omega\setminus
B(x,\,\varepsilon)}
\langle\vec{G}_{jk},\nu\rangle\,d\sigma\Big)e_j\odot e_k
-\sum_{j,k=1}^n\Big(\int_{\Omega\cap\partial B(x,\,\varepsilon)}
\langle\vec{G}_{jk},\nu\rangle\,d{\mathcal{H}}^{n-1}\Big)e_j\odot
e_k
\nonumber\\[4pt]
&=\sum_{j,k=1}^n\Big(\int_{\partial_\ast\Omega\setminus
B(x,\,\varepsilon)}
\langle\vec{F}_{jk},\nu\rangle\,d\sigma\Big)e_j\odot e_k
-\sum_{j,k=1}^n\Big(\int_{\Omega\cap\partial B(x,\,\varepsilon)}
\langle\vec{F}_{jk},\nu\rangle\,d{\mathcal{H}}^{n-1}\Big)e_j\odot
e_k
\nonumber\\[4pt]
&=\sum_{j,k=1}^n\Big(\int_{\partial_\ast\Omega\setminus
B(x,\,\varepsilon)}
\frac{(x_j-y_j)\nu_k(y)}{|x-y|^n}\,d\sigma(y)\Big)e_j\odot e_k
\nonumber\\[4pt]
&\quad -\sum_{j,k=1}^n\Big(\int_{\Omega\cap\partial
B(x,\,\varepsilon)}
\frac{(x_j-y_j)\nu_k(y)}{|x-y|^n}\,d{\mathcal{H}}^{n-1}(y)\Big)e_j\odot
e_k
\nonumber\\[4pt]
&=\int_{\partial_\ast\Omega\setminus
B(x,\,\varepsilon)}\frac{x-y}{|x-y|^n}\odot\nu(y)\,d\sigma(y)
-\int_{\Omega\cap\partial
B(x,\,\varepsilon)}\frac{x-y}{|x-y|^n}\odot\nu(y)\,d{\mathcal{H}}^{n-1}(y).
\end{align}
With this in hand, the first equality in \eqref{Cau-C2bff} readily
follows. Thus, \eqref{Cau-C2bff} is fully proved.

To proceed, assume that $\Omega\subseteq{\mathbb{R}}^n$ is a bounded
Lebesgue measurable set whose boundary is upper Ahlfors regular.
Then \eqref{Cau-C2bff} implies that for each $x\in{\mathbb{R}}^n$
\begin{align}\label{Cau-C2btg590}
\Big|\int_{\partial_\ast\Omega\setminus
B(x,\,\varepsilon)}\frac{x-y}{|x-y|^n}\odot\nu(y)\,d\sigma(y)\Big|
\leq\frac{{\mathcal{H}}^{n-1}\big(\partial
B(x,\,\varepsilon)\big)}{\varepsilon^{n-1}}=\omega_{n-1},
\end{align}
for ${\mathscr{L}}^1$-a.e. $\varepsilon>0$. Fix now
$x\in{\mathbb{R}}^n$ and pick an arbitrary $r\in(0,\infty)$. Based
on \eqref{Cau-C2btg590} we conclude that there exists
$\varepsilon\in(r/2,r)$ such that
\begin{align}\label{Cau-C2btg590.FF.1}
\Big|\int_{\partial_\ast\Omega\setminus
B(x,\,\varepsilon)}\frac{x-y}{|x-y|^n}\odot\nu(y)\,d\sigma(y)\Big|
\leq\omega_{n-1}.
\end{align}
For this choice of $\varepsilon$ we may then estimate
\begin{align}\label{Cau-C2btg590.FF.2}
\Big|\int_{\partial_\ast\Omega\setminus
B(x,\,r)}\frac{x-y}{|x-y|^n}\odot\nu(y) &\,d\sigma(y)\Big|
\nonumber\\[4pt]
\leq &\,\Big|\int_{\partial_\ast\Omega\setminus
B(x,\,\varepsilon)}\frac{x-y}{|x-y|^n}\odot\nu(y)\,d\sigma(y)\Big|
\nonumber\\[4pt]
&+\Big|\int_{[B(x,\,r)\setminus
B(x,\,\varepsilon)]\cap\partial_\ast\Omega}\frac{x-y}{|x-y|^n}\odot\nu(y)\,d\sigma(y)\Big|
\nonumber\\[4pt]
\leq &\,\omega_{n-1} +\int_{[B(x,\,r)\setminus
B(x,\,\varepsilon)]\cap\partial\Omega}\frac{d{\mathcal{H}}^{n-1}(y)}{|x-y|^{n-1}}
\nonumber\\[4pt]
\leq &\,\omega_{n-1} +\int_{[B(x,\,2\varepsilon)\setminus
B(x,\,\varepsilon)]\cap\partial\Omega}\frac{d{\mathcal{H}}^{n-1}(y)}{|x-y|^{n-1}}
\nonumber\\[4pt]
\leq
&\,\omega_{n-1}+\varepsilon^{-(n-1)}{\mathcal{H}}^{n-1}\big(B(x,\,2\varepsilon)\cap\partial\Omega\big).
\end{align}
If ${\rm dist}\,(x,\partial\Omega)\leq 2\varepsilon$, pick a point
$x_0\in\partial\Omega$ such that ${\rm
dist}(x,\partial\Omega)=|x-x_0|$. In particular, $|x-x_0|\leq
2\varepsilon$ which forces $B(x,2\varepsilon)\subseteq
B(x_0,4\varepsilon)$. As such,
\begin{equation}\label{eq:ryty}
{\mathcal{H}}^{n-1}\big(B(x,2\varepsilon)\cap\partial\Omega\big)
\leq{\mathcal{H}}^{n-1}\big(B(x_0,4\varepsilon)\cap\partial\Omega\big)
\leq c(4\varepsilon)^{n-1},
\end{equation}
with $c\in(0,\infty)$ standing for the upper Ahlfors regularity
constant of $\partial\Omega$. On the other hand, if ${\rm
dist}\,(x,\partial\Omega)>2\varepsilon$ then
${\mathcal{H}}^{n-1}\big(B(x,2\varepsilon)\cap\partial\Omega\big)=0$.
Thus, taking $N:=\omega_{n-1}+c\,4^{n-1}$ the desired conclusion
follows from \eqref{Cau-C2btg590.FF.2} and \eqref{eq:ryty}, in the
case when $\Omega$ is as in the statement of the lemma and also
bounded.

Finally, when $\Omega$ is as in the statement of the lemma but
unbounded, consider $\Omega^c:={\mathbb{R}}^n\setminus\Omega$. Then
$\Omega^c\subseteq{\mathbb{R}}^n$ is a bounded, Lebesgue measurable
set, with the property that $\partial(\Omega^c)=\partial\Omega$ and
$\partial_\ast(\Omega^c)=\partial_\ast\Omega$. Moreover, the
geometric measure theoretic outward unit normal to $\Omega^c$ is
$-\nu$. Then \eqref{Cau-CJiigv} follows from what we have proved so
far applied to $\Omega^c$.
\end{proof}

It is clear from \eqref{Cau-C1} that, the boundary-to-domain
Cauchy-Clifford operator is well-defined on
$L^1(\partial\Omega,\sigma)$. To state our next lemma, recall that
$\rho(\cdot)$ has been introduced in \eqref{eq:DFFV}.

\begin{lemma}\label{yreLLa}
Let $\Omega$ be a nonempty, proper, open subset of ${\mathbb{R}}^{n}$
whose boundary is compact, upper Ahlfors regular, and satisfies
\eqref{Tay-1}. Then the Cauchy-Clifford operator \eqref{Cau-C1} has
the property that
\begin{equation}\label{igttYYH}
\mathcal{C}1=\left\{
\begin{array}{ll}
1\,\,\mbox{ in }\,\,\Omega & \mbox{if $\Omega$ bounded},
\\[4pt]
0\,\,\mbox{ in }\,\,\Omega & \mbox{if $\Omega$ unbounded},
\end{array}
\right.
\end{equation}
and for each $\alpha\in(0,1)$ there exists a finite $M>0$, depending
only on $n$, $\alpha$, ${\rm diam}(\partial\Omega)$, and the upper
Ahlfors regularity constant of $\partial\Omega$, such that for every
$f\in{\mathscr{C}}^\alpha(\partial\Omega)\otimes{\mathcal{C}}\!\ell_{n}$
one has
\begin{align}\label{UgagfgPPn}
\sup_{x\in\Omega}\,\big|(\mathcal{C}f)(x)\big|
+\sup_{x\in\Omega}\,\Big\{\rho(x)^{1-\alpha}\big|\nabla(\mathcal{C}f)(x)\big|\Big\}
\leq M\,\|f\|_{{\mathscr{C}}^\alpha(\partial\Omega)\otimes{\mathcal{C}}\!\ell_{n}}.
\end{align}
\end{lemma}

\begin{proof}
The fact that $\mathcal{C}1=1$ in $\Omega$ when $\Omega$ is bounded
follows from \eqref{Cau-C2bff}, written for $x\in\Omega$ and
suitably small $\varepsilon>0$. That $(\mathcal{C}1)(x)=0$ for each
$x\in\Omega$ when $\Omega$ is unbounded also follows from formula
\eqref{Cau-C2bff}, this time considered for the bounded set
$\Omega^c:={\mathbb{R}}^n\setminus\Omega$ (since in this case
$\Omega^c\cap\partial B(x,\varepsilon)=\varnothing$ if $\varepsilon>0$
is sufficiently small). Having proved \eqref{igttYYH}, then
\eqref{UgagfgPPn} follows with the help of Lemma~\ref{Lemma.3}.
\end{proof}

In contrast to Lemma~\ref{yreLLa} (cf. also Lemma~\ref{yreLLa.Tghba}
below), we note that there exists a bounded open set
$\Omega\subset{\mathbb{R}}^2\equiv{\mathbb{C}}$ whose boundary is a
rectifiable Jordan curve, and there exists a complex-valued function
$f\in{\mathscr{C}}^{1/2}(\partial\Omega)$ with the property that the
boundary-to-domain Cauchy operator naturally associated with
$\Omega$ acting on $f$ is actually an unbounded function in
$\Omega$. See the discussion in \cite{Dyn1}, \cite{Dyn2}.

\begin{lemma}\label{yreLLa.Tghba}
Let $\Omega\subset{\mathbb{R}}^{n}$ be a uniform domain whose
boundary is compact, upper Ahlfors regular, and satisfies
\eqref{Tay-1}. Then the boundary-to-domain Cauchy-Clifford operator
has the property that for each $\alpha\in(0,1)$ is well-defined,
linear, and bounded in the context
\begin{equation}\label{Rdac-2BBa.34}
\mathcal{C}:{\mathscr{C}}^{\alpha}(\partial\Omega)\otimes{\mathcal{C}}\!\ell_{n}\longrightarrow
{\mathscr{C}}^{\alpha}\big(\overline{\Omega}\,\big)\otimes{\mathcal{C}}\!\ell_{n},
\end{equation}
with operator norm controlled in terms of $n$, $\alpha$, ${\rm
diam}(\partial\Omega)$, and the upper Ahlfors regularity constant of
$\partial\Omega$.
\end{lemma}

\begin{proof}
This is a direct consequence of Lemma~\ref{yreLLa} and
Lemma~\ref{Lem-J1}.
\end{proof}

In the class of {\rm UR} domains with compact boundaries that are
also uniform domains, it follows from Lemma~\ref{yreLLa.Tghba} and
the jump-formula \eqref{2.3.21} that the principal value
Cauchy-Clifford operator ${\mathcal{C}}^{{}^{\rm pv}}\!$ defines a bounded mapping from
${\mathscr{C}}^{\alpha}(\partial\Omega)\otimes{\mathcal{C}}\!\ell_{n}$
into itself for each $\alpha\in(0,1)$. The goal is to prove that
this boundedness result actually holds under much more relaxed
background assumptions on the underlying domain. In this regard, a
key aspect has to do with the action of ${\mathcal{C}}^{{}^{\rm pv}}\!$ on constants. Note
that when $\Omega\subset{\mathbb{R}}^n$ is a {\rm UR} domain with
compact boundary, it follows from \eqref{igttYYH} and \eqref{2.3.21}
that the principal value Cauchy-Clifford operator satisfies
\begin{equation}\label{eq:CsqASd.2tTT}
{\mathcal{C}}^{{}^{\rm pv}}\!1=\left\{
\begin{array}{ll}
+\tfrac{1}{2}\,\,\mbox{ on }\,\,\partial\Omega & \mbox{if $\Omega$
is bounded},
\\[6pt]
-\tfrac{1}{2}\,\,\mbox{ on }\,\,\partial\Omega & \mbox{if $\Omega$
is unbounded}.
\end{array}
\right.
\end{equation}
The lemma below establishes a formula similar in spirit to \eqref{eq:CsqASd.2tTT} but for
a much larger class of sets $\Omega\subset{\mathbb{R}}^{n}$ than the category of
{\rm UR} domains with compact boundaries.

\begin{lemma}\label{yreLLa.2}
Let $\Omega\subset{\mathbb{R}}^{n}$ be a Lebesgue measurable set
whose boundary is compact, Ahlfors regular, and such that
\eqref{Tay-1} is satisfied {\rm (}hence, in particular, $\Omega$ has
locally finite perimeter{\rm )}. As in the past, consider
$\sigma:={\mathcal{H}}^{n-1}\lfloor\partial\Omega$ and let $\nu$
denote the outward unit normal to $\Omega$. Then for $\sigma$-a.e.
$x\in\partial\Omega$ there holds
\begin{eqnarray}\label{Cau-C2.6a7}
\lim_{\varepsilon\to 0^+}\frac{1}{\omega_{n-1}}
\int\limits_{\stackrel{y\in\partial\Omega}{|x-y|>\varepsilon}}
\frac{x-y}{|x-y|^n}\odot\nu(y)\,d\sigma(y) = \left\{
\begin{array}{ll}
+\tfrac{1}{2} & \mbox{if $\Omega$ is bounded},
\\[8pt]
-\tfrac{1}{2} & \mbox{if $\Omega$ is unbounded}.
\end{array}
\right.
\end{eqnarray}
\end{lemma}

\begin{proof}
Consider first the case when $\Omega$ is bounded. Fix
$x\in\partial^\ast\Omega$ and pick an arbitrary $\delta>0$. From
Lemma~\ref{rtggb} we know that there exist
${\mathcal{O}}_x\subset(0,1)$ of density $1$ at $0$ (i.e.,
satisfying \eqref{eq:pytr}) and some $r_\delta>0$ with the property
that
\begin{equation}\label{Igag-u642}
\Bigg|\frac{{\mathcal{H}}^{n-1}\big(\Omega\cap\partial
B(x,r)\big)}{\omega_{n-1}r^{n-1}}
-\frac{1}{2}\Bigg|<\delta,\qquad\forall\,r\in{\mathcal{O}}_x\cap(0,r_\delta).
\end{equation}
Since \eqref{eq:pytr} entails
\begin{align}\label{eq:pytr.22}
\lim_{\varepsilon\to 0^{+}}
&\frac{{\mathscr{L}}^1\big({\mathcal{O}}_x
\cap(\varepsilon/2,\varepsilon)\big)}{\varepsilon}
\nonumber\\[4pt]
&\qquad=\lim_{\varepsilon\to
0^{+}}\frac{{\mathscr{L}}^1\big({\mathcal{O}}_x
\cap(0,\varepsilon)\big)}{\varepsilon} -\lim_{\varepsilon\to
0^{+}}\frac{{\mathscr{L}}^1\big({\mathcal{O}}_x
\cap(0,\varepsilon/2)\big)}{\varepsilon}
\nonumber\\[4pt]
&\qquad=1-\frac{1}{2}=\frac{1}{2},
\end{align}
it follows that there exists $\varepsilon_\delta\in(0,r_\delta)$
with the property that
\begin{equation}\label{eq:pytr.33}
{\mathscr{L}}^1\big({\mathcal{O}}_x\cap(\varepsilon/2,\varepsilon)\big)>0,\qquad
\forall\,\varepsilon\in(0,\varepsilon_\delta).
\end{equation}
From our assumptions on $\Omega$ and \eqref{Cau-C2bff} we also know
that
\begin{align}\label{Cau-C2bff.5FGV}
\begin{array}{c}
\mbox{$\exists\,N_x\subset(0,\infty)$ with ${\mathscr{L}}^1(N_x)=0$
such that $\forall\,r\in(0,\infty)\setminus N_x$ we have}
\\[12pt]
\displaystyle
\frac{1}{\omega_{n-1}}\int\limits_{\stackrel{y\in\partial\Omega}
{|x-y|>r}}\frac{x-y}{|x-y|^n}\odot\nu(y)\,d\sigma(y)
=\frac{{\mathcal{H}}^{n-1}\big(\Omega\cap\partial
B(x,r)\big)}{\omega_{n-1}r^{n-1}}.
\end{array}
\end{align}
Consider next $\varepsilon\in(0,\varepsilon_\delta)$ and note that
$\big[{\mathcal{O}}_x\cap(\varepsilon/2,\varepsilon)\big]\setminus
N_x\not=\varnothing$, thanks to \eqref{eq:pytr.33}. As such, it is
possible to select
$r\in\big[{\mathcal{O}}_x\cap(\varepsilon/2,\varepsilon)\big]\setminus
N_x$ for which we then write
\begin{align}\label{Ibahb-u5f44}
\int\limits_{\stackrel{y\in\partial\Omega}{|x-y|>\varepsilon/4}}
\frac{x-y}{|x-y|^n}\odot\nu(y)\,d\sigma(y)
&=\int\limits_{\stackrel{y\in\partial\Omega}{r\geq|x-y|>\varepsilon/4}}
\frac{x-y}{|x-y|^n}\odot\nu(y)\,d\sigma(y)
\nonumber\\[4pt]
&\quad+\int\limits_{\stackrel{y\in\partial\Omega}{|x-y|>r}}
\frac{x-y}{|x-y|^n}\odot\nu(y)\,d\sigma(y).
\end{align}
In turn, \eqref{Ibahb-u5f44}, \eqref{Cau-C2bff.5FGV}, and
\eqref{Igag-u642} permit us to estimate
\begin{align}\label{Ibahb-u5f45}
\Bigg|\frac{1}{\omega_{n-1}}\int\limits_{\stackrel{y\in\partial\Omega}{|x-y|>\varepsilon/4}}
&\frac{x-y}{|x-y|^n}\odot\nu(y)\,d\sigma(y)-\frac{1}{2}\Bigg|
\nonumber\\[4pt]
&\leq\Bigg|\frac{1}{\omega_{n-1}}\int\limits_{\stackrel{y\in\partial\Omega}{r\geq|x-y|>\varepsilon/4}}
\frac{x-y}{|x-y|^n}\odot\nu(y)\,d\sigma(y)\Bigg|
+\Bigg|\frac{{\mathcal{H}}^{n-1}\big(\Omega\cap\partial
B(x,r)\big)}{\omega_{n-1}r^{n-1}} -\frac{1}{2}\Bigg|
\nonumber\\[4pt]
&\leq\sup_{r\in(\varepsilon/2,\,\varepsilon)}
\Bigg|\frac{1}{\omega_{n-1}}\int\limits_{\stackrel{y\in\partial\Omega}{r\geq|x-y|>\varepsilon/4}}
\frac{x-y}{|x-y|^n}\odot\nu(y)\,d\sigma(y)\Bigg|+\delta
\end{align}
which, in light of Proposition~\ref{Yfe488} (whose applicability in
the current setting is ensured by \eqref{TaGva134}), then yields
(bearing in mind \eqref{2.1.11})
\begin{align}\label{Ibahb-u5f46}
\limsup_{\varepsilon\to 0^{+}}
\Bigg|\frac{1}{\omega_{n-1}}\int\limits_{\stackrel{y\in\partial\Omega}{|x-y|>\varepsilon/4}}
\frac{x-y}{|x-y|^n}\odot\nu(y)\,d\sigma(y)-\frac{1}{2}\Bigg|\leq\delta
\,\,\mbox{ for $\sigma$-a.e. $x\in\partial\Omega$}.
\end{align}
Given that $\delta>0$ has been arbitrarily chosen, the version of
\eqref{Cau-C2.6a7} for $\Omega$ bounded readily follows from this.
Finally, the version of \eqref{Cau-C2.6a7} corresponding to $\Omega$
unbounded is a consequence of what we have proved so far, applied to
the bounded set $\Omega^c:={\mathbb{R}}^n\setminus\Omega$ (whose
geometric measure theoretic outward unit normal is $-\nu$).
\end{proof}

The stage has been set to show that, under much less restrictive
conditions on the underlying set $\Omega$ (than the class of {\rm
UR} domains with compact boundaries that are also uniform domains),
the principal value Cauchy-Clifford operator ${\mathcal{C}}^{{}^{\rm pv}}\!$ continues to
be a bounded mapping from
${\mathscr{C}}^{\alpha}(\partial\Omega)\otimes{\mathcal{C}}\!\ell_{n}$
into itself for each $\alpha\in(0,1)$. In this regard, our result
can be thought of as the higher-dimensional generalization of the
classical Plemelj-Privalov theorem according to which the Cauchy
integral operator on a piecewise smooth Jordan curve without cusps
in the plane is bounded on H\"older spaces (cf.  \cite{Plem},
\cite{Priv1}, \cite{Priv2}, as well as the discussion in \cite[\S19,
pp.\,45-49]{Mu}). In addition, we also establish a natural jump
formula and prove that $2{\mathcal{C}}^{{}^{\rm pv}}\!$ is idempotent on
${\mathscr{C}}^{\alpha}(\partial\Omega)\otimes{\mathcal{C}}\!\ell_{n}$
with $\alpha\in(0,1)$. We wish to stress that, even in the more
general geometric measure theoretic setting considered below, we
retain \eqref{Cau-C2} as the definition of the Cauchy-Clifford
operator ${\mathcal{C}}^{{}^{\rm pv}}\!$.

\begin{theorem}\label{i65r5ED}
Let $\Omega\subset{\mathbb{R}}^{n}$ be a Lebesgue measurable set
whose boundary is compact, upper Ahlfors regular, and satisfies
\eqref{Tay-1}. As in the past, define
$\sigma:={\mathcal{H}}^{n-1}\lfloor\partial\Omega$,
and fix an arbitrary $\alpha\in(0,1)$.

Then for each $f\in{\mathscr{C}}^{\alpha}(\partial\Omega)\otimes{\mathcal{C}}\!\ell_{n}$
the limit defining ${\mathcal{C}}^{{}^{\rm pv}}\!f(x)$ as in \eqref{Cau-C2} exists for
$\sigma$-a.e. $x\in\partial\Omega$, and the operator ${\mathcal{C}}^{{}^{\rm pv}}\!$
induces a well-defined, linear, and bounded mapping
\begin{equation}\label{Rdac-2BBa.34BBB}
{\mathcal{C}}^{{}^{\rm pv}}:{\mathscr{C}}^{\alpha}(\partial\Omega)
\otimes{\mathcal{C}}\!\ell_{n}\longrightarrow
{\mathscr{C}}^{\alpha}(\partial\Omega)\otimes{\mathcal{C}}\!\ell_{n}.
\end{equation}

Furthermore, under the additional assumption that the set $\Omega$ is
open, the jump formula
\begin{equation}\label{2.3.21BBBB}
\begin{array}{c}
\big(\mathcal{C}f\big)\Big|^{{}^{\rm n.t.}}_{\partial\Omega}
=\big(\tfrac{1}{2}I+{\mathcal{C}}^{{}^{\rm pv}}\!\big)f
\,\text{ at $\sigma$-a.e. point in }\,\partial\Omega
\\[10pt]
\text{is valid for every function
$f\in{\mathscr{C}}^{\alpha}(\partial\Omega)\otimes{\mathcal{C}}\!\ell_{n}$},
\end{array}
\end{equation}
and one also has
\begin{equation}\label{tgVV}
\big({\mathcal{C}}^{{}^{\rm pv}}\big)^2=\tfrac{1}{4}I\,\,\,\mbox{ on }\,\,\,
{\mathscr{C}}^{\alpha}(\partial\Omega)\otimes{\mathcal{C}}\!\ell_{n}.
\end{equation}
\end{theorem}

Incidentally, given an open set $\Omega$ in the plane, the fact that
its boundary is a piecewise smooth Jordan curve implies that
$\partial\Omega$ is compact and upper Ahlfors regular, while the
additional property that $\partial\Omega$ lacks cusps implies that
\eqref{Tay-1} holds. Hence, our demands on the underlying domain
$\Omega$ are weaker versions of the hypotheses in the formulation of
the classical Plemelj-Privalov theorem mentioned earlier.

\vskip 0.08in
\begin{proof}[Proof of Theorem~\ref{i65r5ED}]
Fix $\alpha\in(0,1)$ and pick an arbitrary function
$f\in{\mathscr{C}}^{\alpha}(\partial\Omega)\otimes{\mathcal{C}}\!\ell_{n}$.
Then for $\sigma$-a.e. $x\in\partial\Omega$, Lemma~\ref{yreLLa.2}
allows us to write
\begin{align}\label{ndksegb}
\lim_{\varepsilon\to 0^+} &\frac{1}{\omega_{n-1}}
\int\limits_{\stackrel{y\in\partial\Omega}{|x-y|>\varepsilon}}
\frac{x-y}{|x-y|^n}\odot\nu(y)\odot f(y)\,d\sigma(y)
\nonumber\\[4pt]
&=\lim_{\varepsilon\to 0^+}\frac{1}{\omega_{n-1}}
\int\limits_{\stackrel{y\in\partial\Omega}{|x-y|>\varepsilon}}
\frac{x-y}{|x-y|^n}\odot\nu(y)\odot\big(f(y)-f(x)\big)\,d\sigma(y)\pm\frac{1}{2}f(x)
\nonumber\\[4pt]
&=\frac{1}{\omega_{n-1}}\int\limits_{\partial\Omega}
\frac{x-y}{|x-y|^n}\odot\nu(y)\odot\big(f(y)-f(x)\big)\,d\sigma(y)\pm\frac{1}{2}f(x),
\end{align}
where the sign of $\frac{1}{2}f(x)$ is plus if $\Omega$ is bounded
and minus if $\Omega$ is unbounded. For the last equality, we have
used Lebesgue's Dominated Convergence Theorem. Indeed, given that
$f(y)-f(x)=O(|x-y|^\alpha)$, an estimate based on the upper Ahlfors
regularity of $\partial\Omega$ in the spirit of \eqref{yrrr075.33}
shows that the last integrand above is absolutely integrable for
each fixed $x\in\partial\Omega$. In turn, \eqref{ndksegb} allows us
to conclude that the limit defining ${\mathcal{C}}^{{}^{\rm pv}}\!f(x)$ as in
\eqref{Cau-C2} exists for $\sigma$-a.e. $x\in\partial\Omega$.
Furthermore, by redefining ${\mathcal{C}}^{{}^{\rm pv}}\!f$ on a set of zero
$\sigma$-measure, there is no loss of generality in assuming that,
for each
$f\in{\mathscr{C}}^{\alpha}(\partial\Omega)\otimes{\mathcal{C}}\!\ell_{n}$
with $\alpha\in(0,1)$,
\begin{eqnarray}\label{Cau-C2kjf}
{\mathcal{C}}^{{}^{\rm pv}}\!f(x)=\pm\tfrac{1}{2}f(x)+\frac{1}{\omega_{n-1}}\int\limits_{\partial\Omega}
\frac{x-y}{|x-y|^n}\odot\nu(y)\odot\big(f(y)-f(x)\big)\,d\sigma(y),\quad\,\forall\,x\in\partial\Omega,\quad
\end{eqnarray}
with the sign dictated by whether $\Omega$ is bounded (plus), or
$\Omega$ is unbounded (minus).

We now proceed to showing that, in the context of
\eqref{Rdac-2BBa.34BBB}, the operator \eqref{Cau-C2kjf} is
well-defined and bounded. To this end, fix distinct points
$x_1,x_2\in\partial\Omega$ and starting from \eqref{Cau-C2kjf} write
\begin{eqnarray}\label{Cau-C2kjf.LP}
{\mathcal{C}}^{{}^{\rm pv}}\!f(x_1)-{\mathcal{C}}^{{}^{\rm pv}}\!f(x_2)=I+II
\end{eqnarray}
where
\begin{eqnarray}\label{Cau-C2kjf.LP.1}
I:=\pm\tfrac{1}{2}\big(f(x_1)-f(x_2)\big)
\end{eqnarray}
and
\begin{align}\label{Cau-C2kjf.LP.2}
II:=\frac{1}{\omega_{n-1}}\int_{\partial\Omega}\Bigg\{
\frac{x_1-y}{|x_1-y|^n} & \odot\nu(y)\odot\big(f(y)-f(x_1)\big)
\nonumber\\[4pt]
&-\frac{x_2-y}{|x_2-y|^n}\odot\nu(y)\odot\big(f(y)-f(x_2)\big)\Bigg\}\,d\sigma(y).
\end{align}
Next, introduce $r:=|x_1-x_2|>0$ and estimate
\begin{align}\label{Cau-C2kjf.LP.3}
\big|II\big|\leq II_1+II_2+II_3,
\end{align}
where
\begin{align}\label{Cau-C2kjf.LP.4}
II_1:=\frac{1}{\omega_{n-1}}\Bigg|\int\limits_{\stackrel{y\in\partial\Omega}{|x_1-y|>2r}}
\frac{x_1-y}{|x_1-y|^n} & \odot\nu(y)\odot\big(f(y)-f(x_1)\big)
\nonumber\\
&-\frac{x_2-y}{|x_2-y|^n}\odot\nu(y)\odot\big(f(y)-f(x_2)\big)\,d\sigma(y)\Bigg|,
\end{align}
while
\begin{align}\label{Cau-C2kjf.LP.4aiii}
II_2
&:=\frac{1}{\omega_{n-1}}\int\limits_{\stackrel{y\in\partial\Omega}{|x_1-y|\leq
2r}}
\Bigg|\frac{x_1-y}{|x_1-y|^n}\odot\nu(y)\odot\big(f(y)-f(x_1)\big)\Bigg|\,d\sigma(y),
\\[6pt]
II_3
&:=\frac{1}{\omega_{n-1}}\int\limits_{\stackrel{y\in\partial\Omega}{|x_1-y|\leq
2r}}
\Bigg|\frac{x_2-y}{|x_2-y|^n}\odot\nu(y)\odot\big(f(y)-f(x_2)\big)\Bigg|\,d\sigma(y).
\label{Cau-C2kjf.LP.4biii}
\end{align}
Note that
\begin{align}\label{Cau-C2kjf.LP.4a}
II_2 \leq
c_n[f]_{\dot{\mathscr{C}}^\alpha(\partial\Omega)\otimes{\mathcal{C}}\!\ell_{n}}
\int\limits_{\stackrel{y\in\partial\Omega}{|x_1-y|\leq 2r}}
\frac{d\sigma(y)}{|x_1-y|^{n-1-\alpha}},
\end{align}
and, given that $|x_1-y|\leq 2r$ forces $|x_2-y|\leq
|x_1-x_2|+|x_1-y|\leq 3r$,
\begin{align}\label{Cau-C2kjf.LP.4b}
II_3
&\leq\frac{1}{\omega_{n-1}}\int\limits_{\stackrel{y\in\partial\Omega}{|x_2-y|\leq
3r}}
\Bigg|\frac{x_2-y}{|x_2-y|^n}\odot\nu(y)\odot\big(f(y)-f(x_2)\big)\Bigg|\,d\sigma(y)
\nonumber\\[4pt]
&\,\,\leq
c_n[f]_{\dot{\mathscr{C}}^\alpha(\partial\Omega)\otimes{\mathcal{C}}\!\ell_{n}}
\int\limits_{\stackrel{y\in\partial\Omega}{|x_2-y|\leq 3r}}
\frac{d\sigma(y)}{|x_2-y|^{n-1-\alpha}}.
\end{align}
On the other hand, with $c\in(0,\infty)$ denoting the upper Ahlfors
regularity constant of $\partial\Omega$, for every
$z\in\partial\Omega$ and $R\in(0,\infty)$ we may estimate
\begin{align}\label{Cau-C2kjf.LP.4c}
\int\limits_{\stackrel{y\in\partial\Omega}{|z-y|<R}}\frac{d\sigma(y)}{|z-y|^{n-1-\alpha}}
&=\sum_{j=1}^\infty\int_{[B(z,2^{1-j}R)\setminus
B(z,2^{-j}R)]\cap\partial\Omega}
\frac{d\sigma(y)}{|z-y|^{n-1-\alpha}}
\nonumber\\[4pt]
&\leq\sum_{j=1}^\infty
(2^{-j}R)^{-(n-1-\alpha)}\sigma\big(B(z,2^{1-j}R)\cap\partial\Omega\big)
\nonumber\\[4pt]
&\leq c\,2^{n-1}\sum_{j=1}^\infty (2^{-j}R)^{\alpha}=MR^{\alpha},
\end{align}
for some constant $M=M(n,\alpha,c)\in(0,\infty)$. In light of this,
we obtain from \eqref{Cau-C2kjf.LP.4a} and \eqref{Cau-C2kjf.LP.4b}
(keeping in mind the significance of the number $r$) that there
exists some constant $M=M(n,\alpha,c)\in(0,\infty)$ with the
property that
\begin{align}\label{Cau-C2kjf.LP.4d}
II_2+II_3 \leq
M[f]_{\dot{\mathscr{C}}^\alpha(\partial\Omega)}|x_1-x_2|^{\alpha}.
\end{align}
Going further, bound
\begin{align}\label{Cau-C2kjf.LP.5}
II_1\leq II_1^{a}+II_1^b,
\end{align}
with
\begin{align}\label{Cau-C2kjf.LP.6}
II_1^{a}
&:=\frac{1}{\omega_{n-1}}\Bigg|\int\limits_{\stackrel{y\in\partial\Omega}{|x_1-y|>2r}}
\frac{x_1-y}{|x_1-y|^n}\odot\nu(y)\odot\big(f(x_2)-f(x_1)\big)\,d\sigma(y)\Bigg|
\nonumber\\[4pt]
&\,\,=\frac{1}{\omega_{n-1}}\Bigg|\Bigg(\int\limits_{\stackrel{y\in\partial\Omega}{|x_1-y|>2r}}
\frac{x_1-y}{|x_1-y|^n}\odot\nu(y)\,d\sigma(y)\Bigg)\odot\big(f(x_2)-f(x_1)\big)\Bigg|
\nonumber\\[4pt]
&\,\,\leq\frac{2^{n/2}}{\omega_{n-1}}\Bigg|\int\limits_{\stackrel{y\in\partial\Omega}{|x_1-y|>2r}}
\frac{x_1-y}{|x_1-y|^n}\odot\nu(y)\,d\sigma(y)\Bigg|\big|f(x_2)-f(x_1)\big|
\nonumber\\[4pt]
&\,\,\leq
M(n,c)\,r^\alpha\,[f]_{\dot{\mathscr{C}}^\alpha(\partial\Omega)\otimes{\mathcal{C}}\!\ell_{n}},
\end{align}
where the penultimate inequality uses \eqref{eq:rtta} while the last
inequality is based on \eqref{Cau-CJiigv}, and
\begin{align}\label{Cau-C2kjf.LP.7}
II_1^{b}
&:=\frac{1}{\omega_{n-1}}\Bigg|\int\limits_{\stackrel{y\in\partial\Omega}{|x_1-y|>2r}}
\Big(\frac{x_1-y}{|x_1-y|^n}-\frac{x_2-y}{|x_2-y|^n}\Big)\odot\nu(y)\odot\big(f(y)-f(x_2)\big)\,d\sigma(y)\Bigg|
\nonumber\\[4pt]
&\,\,\leq\frac{2^{n/2}}{\omega_{n-1}}\int\limits_{\stackrel{y\in\partial\Omega}{|x_1-y|>2r}}
\Big|\frac{x_1-y}{|x_1-y|^n}-\frac{x_2-y}{|x_2-y|^n}\Big||f(y)-f(x_2)|\,d\sigma(y)
\nonumber\\[4pt]
&\,\,\leq
c_n\,r\,[f]_{\dot{\mathscr{C}}^\alpha(\partial\Omega)\otimes{\mathcal{C}}\!\ell_{n}}
\int\limits_{\stackrel{y\in\partial\Omega}{|x_1-y|>2r}}\frac{d\sigma(y)}{|x_1-y|^{n-\alpha}},
\end{align}
using the Mean Value Theorem and the fact that $f$ is H\"older of
order $\alpha$. Here it helps to note that if $y\in\partial\Omega$
and $|x_1-y|>2r$ then $|\xi-y|\approx|x_1-y|$ for all
$\xi\in[x_1,x_2]$, and also $|y-x_2|<|y-x_1|/2$. To continue, with
$c\in(0,\infty)$ denoting the upper Ahlfors regularity constant of
$\partial\Omega$ we observe that
\begin{align}\label{Cau-C2kjf.LP.8}
\int\limits_{\stackrel{y\in\partial\Omega}{|x_1-y|>2r}}\frac{d\sigma(y)}{|x_1-y|^{n-\alpha}}
&=\sum_{j=1}^\infty\int_{[B(x_1,2^{j+1}r)\setminus
B(x_1,2^{j}r)]\cap\partial\Omega}
\frac{d\sigma(y)}{|x_1-y|^{n-\alpha}}
\nonumber\\[4pt]
&\leq\sum_{j=1}^\infty
(2^{j}r)^{-(n-\alpha)}\sigma\big(B(x_1,2^{j+1}r)\cap\partial\Omega\big)
\nonumber\\[4pt]
&\leq c\,2^{n-1}\sum_{j=1}^\infty
(2^{j}r)^{-1+\alpha}=Mr^{-1+\alpha},
\end{align}
for some constant $M=M(n,\alpha,c)\in(0,\infty)$. Combining
\eqref{Cau-C2kjf.LP.5}, \eqref{Cau-C2kjf.LP.6},
\eqref{Cau-C2kjf.LP.7}, and \eqref{Cau-C2kjf.LP.8} we conclude that
there exists a constant $M=M(n,\alpha,c)\in(0,\infty)$ with the
property that
\begin{align}\label{Cau-C2kjf.LP.4WW}
II_1\leq
M[f]_{\dot{\mathscr{C}}^\alpha(\partial\Omega)\otimes{\mathcal{C}}\!\ell_{n}}|x_1-x_2|^{\alpha}.
\end{align}
From \eqref{Cau-C2kjf.LP}-\eqref{Cau-C2kjf.LP.1},
\eqref{Cau-C2kjf.LP.3}, \eqref{Cau-C2kjf.LP.4d}, and
\eqref{Cau-C2kjf.LP.4WW} we may then conclude that
\begin{eqnarray}\label{Cau-C2kjf.LP.768}
\big|{\mathcal{C}}^{{}^{\rm pv}}\!f(x_1)-{\mathcal{C}}^{{}^{\rm pv}}\!f(x_2)\big|\leq
M[f]_{\dot{\mathscr{C}}^\alpha(\partial\Omega)
\otimes{\mathcal{C}}\!\ell_{n}}|x_1-x_2|^{\alpha},
\qquad\forall\,x_1,x_2\in\partial\Omega,
\end{eqnarray}
for some constant $M=M(n,\alpha,c)\in(0,\infty)$. The argument so
far gives that the Cauchy-Clifford singular integral operator
${\mathcal{C}}^{{}^{\rm pv}}$ maps
$\dot{\mathscr{C}}^{\alpha}(\partial\Omega)\otimes{\mathcal{C}}\!\ell_{n}$
boundedly into itself. Having established this, Lemma~\ref{Lemma.1}
may be invoked (bearing in mind that \eqref{Cau-C2kjf} forces
${\mathcal{C}}^{{}^{\rm pv}}1=\pm\tfrac{1}{2}$) in order to finish
the proof of the theorem.

Turning our attention to the last part of the statement of the theorem,
make the additional assumption that the set $\Omega$ is open. As far as
the jump-formula \eqref{2.3.21BBBB} is concerned, it has been
already noted that the action of the boundary-to-domain
Cauchy-Clifford operator \eqref{Cau-C1} is meaningful on H\"older
functions. Also, Proposition~\ref{Ctay-5} ensures that
it is meaningful to consider the nontangential boundary trace in the
left-hand side of \eqref{2.3.21BBBB} given that
$\Omega\subseteq{\mathbb{R}}^n$ is an open set with an Ahlfors
regular boundary satisfying \eqref{Tay-1} (hence, $\Omega$ is an Ahlfors
regular domain; cf. Definition~\ref{ADRDOM}). Assume now that some
$f\in{\mathscr{C}}^{\alpha}(\partial\Omega)\otimes{\mathcal{C}}\!\ell_{n}$
with $\alpha\in(0,1)$ has been given and observe that ${\mathcal{C}}f$ is
continuous in $\Omega$. Fix $x\in\partial^\ast\Omega$ and let ${\mathcal{O}}_x$
be the set given by Lemma~\ref{rtggb} applied with $\Omega$ replaced by the
Lebesgue measurable set ${\mathbb{R}}^n\setminus\Omega$. In particular,
\begin{equation}\label{eq:UUt4ED.iii}
\lim_{{\mathcal{O}}_x\ni\,\varepsilon\to 0^{+}}
\frac{{\mathcal{H}}^{n-1}\big(\partial
B(x,\varepsilon)\setminus\Omega\big)}
{\omega_{n-1}\varepsilon^{n-1}}=\frac{1}{2}.
\end{equation}
For some $\kappa>0$ fixed, write
\begin{align}\label{TD14}
\lim\limits_{\stackrel{z\in\Gamma_\kappa(x)}{z\to
x}}{\mathcal{C}}f(z) &=\lim_{{\mathcal{O}}_x\ni\,\varepsilon\to
0^+}\lim\limits_{\stackrel{z\in\Gamma_\kappa(x)}{z\to x}}
\frac{1}{\omega_{n-1}}\int\limits_{\stackrel{|x-y|>\varepsilon}{y\in\partial\Omega}}
\frac{z-y}{|z-y|^{n}}\odot\nu(y)\odot f(y)\,d\sigma(y)
\nonumber\\[4pt]
&\quad+\lim_{{\mathcal{O}}_x\ni\,\varepsilon\to
0^+}\lim\limits_{\stackrel{z\in\Gamma_\kappa(x)}{z\to x}}
\frac{1}{\omega_{n-1}}\int\limits_{\stackrel{|x-y|<\varepsilon}{y\in\partial\Omega}}
\frac{z-y}{|z-y|^{n}}\odot\nu(y)\odot (f(y)-f(x))\,d\sigma(y)
\nonumber\\[4pt]
&\quad +\left(\lim_{{\mathcal{O}}_x\ni\,\varepsilon\to
0^+}\lim\limits_{\stackrel{z\in\Gamma_\kappa(x)}{z\to x}}
\frac{1}{\omega_{n-1}}\int\limits_{\stackrel{|x-y|<\varepsilon}{y\in\partial\Omega}}
\frac{z-y}{|z-y|^{n}}\odot\nu(y)\,d\sigma(y)\right)\odot f(x)
\nonumber\\[4pt]
&=:I_1+I_2+I_3.
\end{align}
For each fixed $\varepsilon>0$, Lebesgue's Dominated Convergence
Theorem applies to the limit as $\Gamma_\kappa(x)\ni z\to x$ in
$I_1$ and yields
\begin{eqnarray}\label{TD15}
I_1=\lim_{\varepsilon\to 0^+}\frac{1}{\omega_{n-1}}
\int\limits_{\stackrel{|x-y|>\varepsilon}{y\in\partial\Omega}}
\frac{x-y}{|x-y|^{n}}\odot\nu(y)\odot f(y)\,d\sigma(y)={\mathcal{C}}^{{}^{\rm pv}}\!f(x).
\end{eqnarray}
To handle $I_2$, we first observe that for every
$x,y\in\partial\Omega$ and $z\in\Gamma_\kappa(x)$,
\begin{align}\label{TD16}
|x-y| & \leq |z-y|+|z-x|\leq |z-y|+(1+\kappa)\,{\rm
dist}(z,\partial\Omega)
\nonumber\\[4pt]
& \leq |z-y|+(1+\kappa)|z-y|=(2+\kappa)|z-y|.
\end{align}
Hence, since $f$ is H\"older of order $\alpha$,
\begin{eqnarray}\label{TD17}
\left|\frac{z-y}{|z-y|^{n}}\odot\nu(y)\right||f(y)-f(x)| \leq
[f]_{\dot{\mathscr{C}}^{\alpha}(\partial\Omega)\otimes{\mathcal{C}}\!\ell_{n}}
\frac{(2+\kappa)^{n-1}}{|x-y|^{n-1-\alpha}},
\end{eqnarray}
so that, based on the upper Ahlfors regularity of $\partial\Omega$
and once again Lebesgue's Dominated Convergence Theorem, we obtain
that
\begin{eqnarray}\label{TD18}
I_2=0.
\end{eqnarray}
To treat $I_3$ in \eqref{TD14}, we first claim that, having fixed
$z\in\Omega$, for ${\mathscr{L}}^1$-a.e $\varepsilon>0$ we have
\begin{align}\label{TD19atyY}
\int\limits_{\stackrel{|x-y|<\varepsilon}{y\in\partial\Omega}}
\frac{z-y}{|z-y|^{n}}\odot\nu(y)\,d\sigma(y)
=\int\limits_{\stackrel{|x-y|=\varepsilon}{y\in{\mathbb{R}}^n\setminus\Omega}}
\frac{z-y}{|z-y|^{n}}\odot\nu(y)\,d\sigma(y).
\end{align}
To justify this, pick a large $R>0$ and apply \eqref{eq:TgavFF.iii}
to $D:=B(0,R)\setminus\Omega$ and, for each $j,k\in\{1,\dots,n\}$,
to the vector field
\begin{equation}\label{eq:YYTf.iii}
\vec{F}_{jk}(y):=\Big(0,\dots,0,\frac{z_j-y_j}{|z-y|^n},0,\dots,0\Big),\qquad
\forall\,y\in{\mathbb{R}}^n\setminus\{z\},
\end{equation}
with the non-zero component on the $k$-th slot. Agree to alter each
$\vec{F}_{jk}$ outside a compact neighborhood of $\overline{D}$ to a
vector field
$\vec{G}_{jk}\in{\mathscr{C}}^1_0({\mathbb{R}}^n\setminus\{z\},{\mathbb{R}}^n)$.
Then by reasoning as in \eqref{eq:u54955}-\eqref{eq:Tgahytttr},
formula \eqref{TD19atyY} follows. Consequently, starting with
\eqref{TD19atyY}, then using \eqref{eigrrTY}, and then
\eqref{eq:UUt4ED.iii}, we obtain
\begin{align}\label{TD20}
&\lim_{{\mathcal{O}}_x\ni\,\varepsilon\to
0^+}\lim\limits_{\stackrel{z\in\Gamma_\kappa(x)}{z\to x}}
\frac{1}{\omega_{n-1}}\int\limits_{\stackrel{|x-y|<\varepsilon}{y\in\partial\Omega}}
\frac{z-y}{|z-y|^{n}}\odot\nu(y)\,d\sigma(y)
\nonumber\\[4pt]
&\qquad\qquad\qquad\qquad =\lim_{{\mathcal{O}}_x\ni\,\varepsilon\to
0^+}\frac{1}{\omega_{n-1}}\int\limits_{\stackrel{|x-y|=\varepsilon}
{y\in{\mathbb{R}}^{n}\setminus\Omega}}\frac{x-y}{|x-y|^{n}}\odot\nu(y)\,d\sigma(y)
\nonumber\\[4pt]
&\qquad\qquad\qquad\qquad =\lim_{{\mathcal{O}}_x\ni\,\varepsilon\to
0^{+}} \frac{{\mathcal{H}}^{n-1}\big(\partial
B(x,\varepsilon)\setminus\Omega\big)}
{\omega_{n-1}\varepsilon^{n-1}}=\frac{1}{2}.
\end{align}
A combination of \eqref{TD14}, \eqref{TD15}, \eqref{TD18}, and
\eqref{TD20} shows that the limit in the left hand-side of
\eqref{TD14} exists and matches $(\tfrac{1}{2}I+{\mathcal{C}}^{{}^{\rm pv}}\!)f(x)$.
This proves that formula \eqref{2.3.21BBBB} holds for each
$f\in{\mathscr{C}}^{\alpha}(\partial\Omega)\otimes{\mathcal{C}}\!\ell_{n}$
at every $x\in\partial^\ast\Omega$, hence at $\sigma$-a.e. point in
$\partial\Omega$, by \eqref{2.1.11} and the assumption \eqref{Tay-1}.

To finish the proof of the theorem, there remains to establish formula \eqref{tgVV}
assuming, again, that the set $\Omega$ is open. Suppose this
is the case and bring in the version of the Cauchy reproducing formula from
\cite[Section~3]{IMiMiTa} to the effect that, under the current assumptions on
the set $\Omega$,
\begin{eqnarray}\label{Aav-jva}
\left.
\begin{array}{r}
u:\Omega\to{\mathcal{C}}\!\ell_{n}\,\,\mbox{ continuous, }\,\,
D_Lu=0\,\,\mbox{ in }\,\,\Omega
\\[6pt]
{\mathcal{N}}u\in L^1(\partial\Omega,\sigma),\,\, u\big|^{{}^{\rm
n.t.}}_{\partial\Omega} \,\,\mbox{ exists $\sigma$-a.e. on
}\,\,\partial\Omega
\end{array}
\right\}\Rightarrow\,u={\mathcal{C}}\big(u\big|^{{}^{\rm
n.t.}}_{\partial\Omega}\big) \,\,\mbox{ in }\,\,\Omega.
\end{eqnarray}
Now, given any
$f\in{\mathscr{C}}^{\alpha}(\partial\Omega)\otimes{\mathcal{C}}\!\ell_{n}$,
define $u:={\mathcal{C}}f$ in $\Omega$. Then, by design,
$u\in{\mathscr{C}}^\infty(\Omega)$ and $D_Lu=0$ in $\Omega$. Also,
\eqref{UgagfgPPn} gives that $\sup_{x\in\Omega}\,\big|u(x)\big| \leq
M\|f\|_{{\mathscr{C}}^\alpha(\partial\Omega)\otimes{\mathcal{C}}\!\ell_{n}}$
which, in turn, forces ${\mathcal{N}}u\in
L^\infty(\partial\Omega,\sigma)\subset L^1(\partial\Omega,\sigma)$
given that $\partial\Omega$ has finite measure. Finally, the jump
formula \eqref{2.3.21} for H\"older functions, established earlier
in the proof, yields
\begin{equation}\label{2Jnab}
\big(u\big|^{{}^{\rm n.t.}}_{\partial\Omega}\big)(x)
=\big(\tfrac{1}{2}I+{\mathcal{C}}^{{}^{\rm pv}}\!\big)f(x)\quad\mbox{ for
}\,\,\sigma\text{-a.e. }\,x\in\partial\Omega.
\end{equation}
Granted these, \eqref{Aav-jva} applies and gives
\begin{equation}\label{2JnaTRR}
u={\mathcal{C}}\big(u\big|^{{}^{\rm n.t.}}_{\partial\Omega}\big)\,\,\text{ in }\,\,\Omega.
\end{equation}
Moreover, since $f\in {\mathscr{C}}^\alpha(\partial\Omega)\otimes{\mathcal{C}}\!\ell_{n}$
and ${\mathcal{C}}^{{}^{\rm pv}}$ is a well-defined mapping
in the context of \eqref{Rdac-2BBa.34BBB},
from \eqref{2Jnab} we see that
\begin{equation}\label{2JnaTRR.456}
u\big|^{{}^{\rm n.t.}}_{\partial\Omega}\in
{\mathscr{C}}^\alpha(\partial\Omega)\otimes{\mathcal{C}}\!\ell_{n}.
\end{equation}
Going to the boundary nontangentially in \eqref{2JnaTRR}
and relying on \eqref{2Jnab} and \eqref{2.3.21BBBB} (bearing in mind \eqref{2JnaTRR.456})
then allow us to write
\begin{equation}\label{2Jnab.2}
\big(\tfrac{1}{2}I+{\mathcal{C}}^{{}^{\rm pv}}\!\big)f
=\big(\tfrac{1}{2}I+{\mathcal{C}}^{{}^{\rm pv}}\!\big)
\big(\tfrac{1}{2}I+{\mathcal{C}}^{{}^{\rm pv}}\!\big)f
\,\text{ at $\sigma$-a.e. point on }\,\partial\Omega,
\end{equation}
from which \eqref{tgVV} now readily follows.
\end{proof}

In the last part of this section we briefly consider harmonic layer
potentials. Recall the standard fundamental solution
$E_{\!{}_\Delta}$ for the Laplacian in ${\mathbb{R}}^n$ from
\eqref{RacvTT}. Given a nonempty open proper subset $\Omega$ of
$\mathbb{R}^n$, let $\sigma:=\mathcal{H}^{n-1}\lfloor\partial\Omega$.
Then the {\tt harmonic} {\tt single} {\tt layer} {\tt operator} associated
with $\Omega$ acts on a function $f$ defined on $\partial\Omega$
according to
\begin{equation}\label{iu7grEE}
\mathscr{S}\!f(x):=\int_{\partial\Omega}E_{\!{}_\Delta}\!(x-y)f(y)\,d\sigma(y),
\qquad x\in\Omega.
\end{equation}
Assume that $\Omega$ is a set of locally finite perimeter for which
\eqref{Tay-1} holds and denote by $\nu$ its (geometric measure
theoretic) outward unit normal. In this context, it follows from
\eqref{Dirac}, \eqref{iu7grEE}, \eqref{Cau-C1}, and the fact that
$\nu\odot\nu=-1$ (cf. \eqref{X-sqr}), that the harmonic single layer
operator and the Cauchy-Clifford operator are related via
\begin{equation}\label{eq:Rvav12.3}
D_L{\mathscr{S}}f=-{\mathcal{C}}(\nu\odot f)\,\,\,\mbox{ in
}\,\,\Omega.
\end{equation}
Parenthetically, we wish to note that, in the same setting, the {\tt
harmonic} {\tt double} {\tt layer} {\tt operator} associated with
$\Omega$ is defined as
\begin{equation}\label{dble-layer}
{\mathscr{D}}f(x):=\frac{1}{\omega_{n-1}}\int_{\partial\Omega}
\frac{\langle\nu(y),y-x\rangle}{|x-y|^n}f(y)\,d\sigma(y),\qquad
x\in\Omega,
\end{equation}
where $\langle\cdot,\cdot\rangle$ is the standard inner product of
vectors in $\mathbb{R}^n$. In particular, from \eqref{Cau-C1},
\eqref{Hil-3bb}, \eqref{Hil-aabV}, and \eqref{dble-layer}, it
follows that
\begin{equation}\label{eq:tr4iuhHH}
\mbox{if $f$ is scalar-valued then }\,\,
{\mathscr{D}}f=\big({\mathcal{C}}f\big)_0\,\,\mbox{ in }\,\,\Omega.
\end{equation}
As a consequence of this and \eqref{Rdac-2BBa.34}, we see that if
$\Omega\subset{\mathbb{R}}^{n}$ is a uniform domain whose boundary
is compact, upper Ahlfors regular, and satisfies \eqref{Tay-1} then
for each $\alpha\in(0,1)$ the harmonic double layer operator induces
a well-defined, linear, and bounded mapping
\begin{equation}\label{RdrFVV}
\mathscr{D}:{\mathscr{C}}^{\alpha}(\partial\Omega)\longrightarrow
{\mathscr{C}}^{\alpha}\big(\overline{\Omega}\,\big).
\end{equation}

Returning to the mainstream discussion, make the convention that
$\nabla^2$ is the vector of all second order partial derivatives in
${\mathbb{R}}^n$. Also, once again, recall \eqref{eq:DFFV}.

\begin{lemma}\label{Lemma.4}
Let $\Omega$ be a domain of class ${\mathscr{C}}^{1+\alpha}$ for
some $\alpha\in(0,1)$ with compact boundary. Then
\begin{align}\label{tAAcc.44jjj}
A:=\sup_{x\in\Omega}\,\big|\nabla(\mathscr{S}1)(x)\big|
+\sup_{x\in\Omega}\,\Big\{\rho(x)^{1-\alpha}\big|\nabla^2(\mathscr{S}1)(x)\big|\Big\}<+\infty
\end{align}
and, in fact, this quantity may be estimated in terms of $n$,
$\alpha$, ${\rm diam}\,(\partial\Omega)$,
$\|\nu\|_{{\mathscr{C}}^\alpha(\partial\Omega)}$, and the upper
Ahlfors regularity constant of $\partial\Omega$.
\end{lemma}

\begin{proof}
Via the identification \eqref{embed} we obtain from
\eqref{eq:Rvav12.3} that
\begin{equation}\label{eq:Rvav12.4}
\nabla(\mathscr{S}1)\equiv
D_L{\mathscr{S}}1=-{\mathcal{C}}\nu\,\,\,\mbox{ in }\,\,\Omega.
\end{equation}
Then, keeping in mind that
$\nu\in{\mathscr{C}}^{\alpha}(\partial\Omega)\otimes{\mathcal{C}}\!\ell_{n}$
under the present assumption on $\Omega$, the claim in
\eqref{tAAcc.44jjj} readily follows by combining \eqref{eq:Rvav12.4}
with \eqref{UgagfgPPn}.
\end{proof}

\section{The Proofs of Theorem~\ref{Main-T1aa} and Theorem~\ref{THM-main222}}
\setcounter{equation}{0} \label{S-6}

We start by presenting the proof of Theorem~\ref{Main-T1aa}.

\vskip 0.08in
\begin{proof}[\underline{Proof of {\rm (a)}\,$\Rightarrow$\,{\rm (e)} in Theorem~\ref{Main-T1aa}}]
Let $\Omega$ be a domain of class ${\mathscr{C}}^{1+\alpha}$,
$\alpha\in(0,1)$, with compact boundary (hence, in particular,
$\Omega$ is a {\rm UR} domain). Also, assume $P(x)$ is an odd,
homogeneous, harmonic polynomial of degree $l\geq 1$ in
${\mathbb{R}}^{n}$ and, with it, associate the singular integral
operator
\begin{eqnarray}\label{T-layer.44-GGG}
{\mathbb{T}}f(x):=\int_{\partial\Omega}\frac{P(x-y)}{|x-y|^{n-1+l}}f(y)\,d\sigma(y),
\qquad x\in\Omega.
\end{eqnarray}
In a first stage, the goal is to prove that there exists a constant
$C\in(1,\infty)$, depending only on $n$, $\alpha$, ${\rm
diam}\,(\partial\Omega)$,
$\|\nu\|_{{\mathscr{C}}^\alpha(\partial\Omega)}$, and the upper
Ahlfors regularity constant of $\partial\Omega$ (something we shall
indicate by writing $C=C(n,\alpha,\Omega)$) such that for every
$f\in{\mathscr{C}}^\alpha(\partial\Omega)$ we have
\begin{align}\label{jht6.55}
\sup_{x\in\Omega}\,|\mathbb{T}f(x)|
+\sup_{x\in\Omega}\,\Big\{\rho(x)^{1-\alpha}\big|\nabla(\mathbb{T}f)(x)\big|\Big\}
\leq
C^{l}2^{l^2}\|P\|_{L^1(S^{n-1})}\|f\|_{{\mathscr{C}}^\alpha(\partial\Omega)}.
\end{align}
We shall do so by induction on $l\in 2{\mathbb{N}}-1$, the degree of
the homogeneous harmonic polynomial $P$. When $l=1$ we have
$P(x)=\sum_{j=1}^n a_jx_j$ for each
$x=(x_1,\dots,x_n)\in{\mathbb{R}}^n$, where the $a_j$'s are some
fixed constants. Hence, in this case,
\begin{equation}\label{eq:fr47p}
\max_{1\leq j\leq n}|a_j|\leq\|P\|_{L^\infty(S^{n-1})}\leq
c_n\|P\|_{L^1(S^{n-1})}
\end{equation}
where the last inequality is a consequence of \eqref{eq:rFG} (with
$c_n\in(0,\infty)$ denoting a dimensional constant), and
\begin{equation}\label{eq:jhYH}
{\mathbb{T}}=\omega_{n-1}\sum_{j=1}^n a_j\partial_j{\mathscr{S}}.
\end{equation}
Then \eqref{jht6.55} follows from \eqref{eq:fr47p}, \eqref{eq:jhYH},
Lemma~\ref{Lemma.4}, and Lemma~\ref{Lemma.3}. To proceed, fix some
odd integer $l\geq 3$ and assume that there exists
$C=C(n,\alpha,\Omega)\in(1,\infty)$ such that
\begin{equation}\label{eq:jhYTYT}
\parbox{11.30cm}{the estimate in \eqref{jht6.55} holds whenever ${\mathbb{T}}$
is associated as in \eqref{T-layer.44-GGG} with an odd harmonic homogeneous
polynomial of degree $\leq l-2$ in ${\mathbb{R}}^{n}$.}
\end{equation}
Also, pick an arbitrary odd harmonic homogeneous polynomial $P(x)$
of degree $l$ in ${\mathbb{R}}^{n}$ and let ${\mathbb{T}}$ be as in
\eqref{T-layer.44-GGG} for this choice of $P$. Consider the family
$P_{rs}(x)$, $1\leq r,s\leq n$, of odd harmonic homogeneous
polynomials of degree $l-2$, as well as the family of odd,
${\mathscr{C}}^\infty$ functions
$k_{rs}:{\mathbb{R}}^{n}\setminus\{0\}\to{\mathbb{R}}^{n}\hookrightarrow{\mathcal{C}}\!\ell_{n}$,
associated with $P$ as in Lemma~\ref{L-Semmes}. For each $1\leq
i,j\leq n$ set
\begin{eqnarray}\label{k-up-ij}
k^{rs}(x):=P_{rs}(x)/|x|^{n+l-3}\quad\mbox{for}\quad
x\in{\mathbb{R}}^{n}\setminus\{0\},
\end{eqnarray}
and introduce the integral operator acting on Clifford
algebra-valued functions, $f=\sum_{I}f_I e_I$ with H\"older scalar
components $f_I$ defined on $\partial\Omega$, according to
\begin{align}\label{T-laLk}
{\mathbb{T}}^{rs}f(x)
&:=\int_{\partial\Omega}k^{rs}(x-y)f(y)\,d\sigma(y)
\nonumber\\[4pt]
&\,=\sum_{I}\Big(\int_{\partial\Omega}k^{rs}(x-y)f_I(y)\,d\sigma(y)\Big)e_I,\qquad
x\in\Omega.
\end{align}
Fix such an arbitrary
$f\in{\mathscr{C}}^\alpha(\partial\Omega)\otimes{\mathcal{C}}\!\ell_n$.
Then from the properties of the $P_{rs}$'s and the induction
hypothesis \eqref{eq:jhYTYT} (used component-wise, keeping in mind
that the sum in \eqref{T-laLk} is performed over a set of
cardinality $2^n$) we conclude that for each $1\leq r,s\leq n$ we
have
\begin{align}\label{jht6.aa}
\sup_{x\in\Omega}\,|(\mathbb{T}^{rs}f)(x)|
+&\sup_{x\in\Omega}\,\Big\{\rho(x)^{1-\alpha}\big|\nabla(\mathbb{T}^{rs}f)(x)\big|\Big\}
\nonumber\\[4pt]
&\leq 2^{n/2}C^{l-2}2^{(l-2)^2}\|P_{rs}\|_{L^1(S^{n-1})}
\|f\|_{{\mathscr{C}}^\alpha(\partial\Omega)\otimes{\mathcal{C}}\!\ell_n}
\nonumber\\[4pt]
&\leq c_n\,C^{l-2}2^{(l-2)^2}2^l\|P\|_{L^1(S^{n-1})}
\|f\|_{{\mathscr{C}}^\alpha(\partial\Omega)\otimes{\mathcal{C}}\!\ell_n}.
\end{align}

Moving on, for every $r,s\in\{1,\dots,n\}$ and
$f:\partial\Omega\to{\mathcal{C}}\!\ell_n$ with H\"older scalar
components we set
\begin{eqnarray}\label{T-layer.oyt}
{\mathbb{T}}_{rs}f(x):=\int_{\partial\Omega}k_{rs}(x-y)\odot
f(y)\,d\sigma(y),\qquad x\in\Omega.
\end{eqnarray}
Then, thanks to formula \eqref{pro-1}, whenever the function $f$ is
actually scalar-valued (i.e.,
$f:\partial\Omega\to{\mathbb{R}}\hookrightarrow{\mathcal{C}}\!\ell_n$)
the original operator ${\mathbb{T}}$ from \eqref{T-layer.44-GGG} may
be recovered from the above ${\mathbb{T}}_{rs}$'s by means of the
identity
\begin{equation}\label{Tgc-654}
{\mathbb{T}}f(x)=\sum_{r,s=1}^n\big[{\mathbb{T}}_{rs}f(x)\big]_s
\quad\mbox{for all }\,\,x\in\Omega.
\end{equation}
To proceed, consider first the case when $\Omega$ is unbounded. In
this scenario, fix some $x\in\Omega$ and select
\begin{equation}\label{eq:rrRf}
R_1\in\big(0,{\rm dist}(x,\partial\Omega)\big)\,\,\mbox{ along with
}\,\, R_2>{\rm dist}(x,\partial\Omega)+{\rm diam}(\partial\Omega).
\end{equation}
Set
$\Omega_{R_1,R_2}:=\big(B(x,R_2)\setminus\overline{B(x,R_1)}\,\big)\cap\Omega$
which is a bounded ${\mathscr{C}}^{1+\alpha}$ domain in
${\mathbb{R}}^n$ with the property that
\begin{equation}\label{eq:iyre}
\partial\Omega_{R_1,R_2}=\partial B(x,R_2)\cup\partial B(x,R_1)\cup\partial\Omega.
\end{equation}
We continue to denote by $\nu$ and $\sigma$ the outward unit normal
and surface measure for $\Omega_{R_1,R_2}$. As a consequence of
\eqref{D-IBP} (used with $\Omega_{R_1,R_2}$ in place of $\Omega$,
$u=k_{rs}(x-\cdot)\in{\mathscr{C}}^\infty\big(\overline{\Omega_{R_1,R_2}}\,\big)$,
and $v\equiv 1$) and \eqref{pro-2}, we then obtain that for each
$r,s\in\{1,\dots,n\}$
\begin{align}\label{IBP-k}
\int_{\partial\Omega_{R_1,R_2}}k_{rs}(x-y)\odot\nu(y)\,d\sigma(y)
&=-\int_{\Omega_{R_1,R_2}}(D_R k_{rs})(x-y)\,dy
\nonumber\\[4pt]
&=\frac{l-1}{n+l-3}\int_{\Omega_{R_1,R_2}} \frac{\partial}{\partial
y_r}\left(\frac{P_{rs}(x-y)}{|x-y|^{n+l-3}}\right)\,dy
\nonumber\\[4pt]
&=\frac{l-1}{n+l-3}\int_{\partial\Omega_{R_1,R_2}}k^{rs}(x-y)\nu_r(y)\,d\sigma(y).
\end{align}
Hence,
\begin{align}\label{IBP-k3}
({\mathbb{T}}_{rs}\nu)(x)
&=\int_{\partial\Omega}k_{rs}(x-y)\odot\nu(y)\,d\sigma(y)
\nonumber\\[4pt]
&=\int_{\partial\Omega_{R_1,R_2}}k_{rs}(x-y)\odot\nu(y)\,d\sigma(y)
-\int_{\partial
B(x,R_1)}k_{rs}(x-y)\odot\frac{x-y}{|x-y|}\,d\sigma(y)
\nonumber\\[4pt]
&\quad+\int_{\partial
B(x,R_2)}k_{rs}(x-y)\odot\frac{x-y}{|x-y|}\,d\sigma(y)
\nonumber\\[4pt]
&=\frac{l-1}{n+l-3}\int_{\partial\Omega_{R_1,R_2}}k^{rs}(x-y)\nu_r(y)\,d\sigma(y)
\nonumber\\[4pt]
&\quad-\int_{S^{n-1}}k_{rs}(\omega)\odot\omega\,d\omega
+\int_{S^{n-1}}k_{rs}(\omega)\odot\omega\,d\omega
\nonumber\\[4pt]
&=\frac{l-1}{n+l-3}\int_{\partial\Omega}k^{rs}(x-y)\nu_r(y)\,d\sigma(y)
\nonumber\\[4pt]
&\quad-\frac{l-1}{n+l-3}\int_{\partial
B(x,R_1)}k^{rs}(x-y)\frac{x_r-y_r}{|x-y|}\,d\sigma(y)
\nonumber\\[4pt]
&\quad+\frac{l-1}{n+l-3}\int_{\partial
B(x,R_2)}k^{rs}(x-y)\frac{x_r-y_r}{|x-y|}\,d\sigma(y)
\nonumber\\[4pt]
&=\frac{l-1}{n+l-3}({\mathbb{T}}^{rs}\nu_r)(x)
\nonumber\\[4pt]
&\quad-\frac{l-1}{n+l-3}\int_{S^{n-1}}k^{rs}(\omega)\omega_r\,d\omega
+\frac{l-1}{n+l-3}\int_{S^{n-1}}k^{rs}(\omega)\omega_r\,d\omega
\nonumber\\[4pt]
&=\frac{l-1}{n+l-3}({\mathbb{T}}^{rs}\nu_r)(x).
\end{align}
From \eqref{IBP-k3} and \eqref{jht6.aa} used with
$f=\nu_r\in{\mathscr{C}}^\alpha(\partial\Omega)$, for $1\leq r,s\leq
n$ we obtain
\begin{align}\label{jht6.ew}
\sup_{x\in\Omega}\,|({\mathbb{T}}_{rs}\nu)(x)|
&+\sup_{x\in\Omega}\,\Big\{\rho(x)^{1-\alpha}\big|\nabla({\mathbb{T}}_{rs}\nu)(x)\big|\Big\}
\nonumber\\[4pt]
&\leq\sup_{x\in\Omega}\,|({\mathbb{T}}^{rs}\nu_r)(x)|
+\sup_{x\in\Omega}\,\Big\{\rho(x)^{1-\alpha}\big|\nabla({\mathbb{T}}^{rs}\nu_r)(x)\big|\Big\}
\nonumber\\[4pt]
&\leq c_n\,C^{l-2}2^{(l-2)^2}2^l\|P\|_{L^1(S^{n-1})}
\|\nu\|_{{\mathscr{C}}^\alpha(\partial\Omega)},
\end{align}
in the case when $\Omega$ is an unbounded domain.

When $\Omega$ is a bounded domain, we once again consider
$\Omega_{R_1,R_2}$ as before and carry out a computation similar in
spirit to what we have just done above. This time, however,
$\Omega_{R_1,R_2}=\Omega\setminus\overline{B(x,R_1)}$ and in place
of \eqref{eq:iyre} we have $\partial\Omega_{R_1,R_2}=\partial
B(x,R_1)\cup\partial\Omega$. Consequently, in place of
\eqref{IBP-k3} we now obtain
\begin{align}\label{IBP-k3-BIS}
({\mathbb{T}}_{rs}\nu)(x)
&=\frac{l-1}{n+l-3}({\mathbb{T}}^{rs}\nu_r)(x)
\nonumber\\[4pt]
&\quad-\frac{l-1}{n+l-3}\int_{S^{n-1}}k^{rs}(\omega)\omega_r\,d\omega
-\int_{S^{n-1}}k_{rs}(\omega)\odot\omega\,d\omega.
\end{align}
To estimate the integrals on the unit sphere we note that, in view
of \eqref{k-up-ij}, \eqref{eq:yre36}, \eqref{eq:rFG}, and
\eqref{eq:Esfag}, we have
\begin{align}\label{jht6.ew-BIS}
\|k^{rs}\|_{L^\infty(S^{n-1})}+\|k_{rs}\|_{L^\infty(S^{n-1})} \leq
c_n\,2^{l}\|P\|_{L^1(S^{n-1})}.
\end{align}
Upon observing that
$\|\nu\|_{{\mathscr{C}}^\alpha(\partial\Omega)}\geq 1$, from
\eqref{IBP-k3-BIS} and \eqref{jht6.ew-BIS} we deduce that an
estimate similar to \eqref{jht6.ew} also holds in the case when
$\Omega$ is a bounded domain (this time, replacing the constant
$c_n$ appearing in \eqref{jht6.ew} by $2c_n$, which is
inconsequential for our purposes). In summary, \eqref{IBP-k3-BIS}
may be assumed to hold irrespective of whether $\Omega$ is bounded
or not.

Going further, let $\widetilde{{\mathbb{T}}}_{rs}$ be the version of
${\mathbb{T}}_{rs}$ from \eqref{T-layer.oyt} in which $\nu(y)$ has
been absorbed in the integral kernel. That is, for
$f:\partial\Omega\to{\mathcal{C}}\!\ell_n$ with H\"older scalar
components set
\begin{eqnarray}\label{T-layer.44tilde}
\widetilde{{\mathbb{T}}}_{rs}f(x):=\int_{\partial\Omega}
\big(k_{rs}(x-y)\odot\nu(y)\big)\odot f(y)\,d\sigma(y),\qquad
x\in\Omega,
\end{eqnarray}
for each $r,s\in\{1,\dots,n\}$. Since
$\widetilde{{\mathbb{T}}}_{rs}1={\mathbb{T}}_{rs}\nu$, from
\eqref{jht6.ew} we conclude that for each $r,s\in\{1,\dots,n\}$
\begin{align}\label{jht6.ew4r4r}
\sup_{x\in\Omega}\,|(\widetilde{{\mathbb{T}}}_{rs}1)(x)|
&+\sup_{x\in\Omega}\,\Big\{\rho(x)^{1-\alpha}
\big|\nabla(\widetilde{{\mathbb{T}}}_{rs}1)(x)\big|\Big\}
\nonumber\\[4pt]
&\leq c_n\,C^{l-2}2^{(l-2)^2}2^l\|P\|_{L^1(S^{n-1})}
\|\nu\|_{{\mathscr{C}}^\alpha(\partial\Omega)}.
\end{align}
Given that the integral kernel of $\widetilde{{\mathbb{T}}}_{rs}$
satisfies
\begin{equation}\label{eq:re358}
|k_{rs}(x-y)\odot\nu(y)|\leq\frac{\|k_{rs}\|_{L^\infty(S^{n-1})}}{|x-y|^{n-1}}
\leq\frac{c_{n}2^l\|P\|_{L^1(S^{n-1})}}{|x-y|^{n-1}},
\end{equation}
and
\begin{equation}\label{eq:re358-vbR}
\big|\nabla_x\big[k_{rs}(x-y)\odot\nu(y)\big]\big|\leq
\frac{\|\nabla k_{rs}\|_{L^\infty(S^{n-1})}}{|x-y|^{n}}
\leq\frac{c_{n}2^l\|P\|_{L^1(S^{n-1})}}{|x-y|^{n}},
\end{equation}
we may invoke Lemma~\ref{Lemma.3} with
\begin{equation}\label{eq:y43ii}
A:=c_n\,C^{l-2}2^{(l-2)^2}2^l\|P\|_{L^1(S^{n-1})}
\|\nu\|_{{\mathscr{C}}^\alpha(\partial\Omega)}\,\,\mbox{ and }\,\,
B:=c_{n}2^l\|P\|_{L^1(S^{n-1})}
\end{equation}
in order to conclude that if $1\leq r,s\leq n$ then
\begin{align}\label{jht6.88}
\sup_{x\in\Omega}\, &|\widetilde{\mathbb{T}}_{rs}f(x)|
+\sup_{x\in\Omega}\,\Big\{\rho(x)^{1-\alpha}\big|\nabla(\widetilde{\mathbb{T}}_{rs}f)(x)\big|\Big\}
\\[4pt]
&\leq C_{n,\alpha,\Omega}\Big\{C^{l-2}2^{(l-2)^2}2^l
\|\nu\|_{{\mathscr{C}}^\alpha(\partial\Omega)}+2^l\Big\}\|P\|_{L^1(S^{n-1})}
\|f\|_{{\mathscr{C}}^\alpha(\partial\Omega)\otimes{\mathcal{C}}\!\ell_n}
\nonumber
\end{align}
for every
$f\in{\mathscr{C}}^\alpha(\partial\Omega)\otimes{\mathcal{C}}\!\ell_n$.
Writing \eqref{jht6.88} for $f$ replaced by $\nu\odot f$ then
yields, in light of \eqref{T-layer.44tilde}, \eqref{T-layer.oyt},
and \eqref{M-nu} (bearing in mind that $\nu\odot\nu=-1$), that for
$1\leq r,s\leq n$ we have
\begin{align}\label{gttf5}
\sup_{x\in\Omega}\,|{\mathbb{T}}_{rs}f(x)|
&+\sup_{x\in\Omega}\,\Big\{\rho(x)^{1-\alpha}\big|\nabla({\mathbb{T}}_{rs}f)(x)\big|\Big\}
\nonumber\\[4pt]
&\leq C_{n,\alpha,\Omega}\Big\{C^{l-2}2^{(l-2)^2}2^l
\|\nu\|_{{\mathscr{C}}^\alpha(\partial\Omega)}+2^l\Big\}\times
\nonumber\\[4pt]
&\hskip 0.80in \times
2\|\nu\|_{{\mathscr{C}}^\alpha(\partial\Omega)}\|P\|_{L^1(S^{n-1})}
\|f\|_{{\mathscr{C}}^\alpha(\partial\Omega)\otimes{\mathcal{C}}\!\ell_n}
\end{align}
for every
$f\in{\mathscr{C}}^\alpha(\partial\Omega)\otimes{\mathcal{C}}\!\ell_n$.
In turn, from this and \eqref{Tgc-654} we finally conclude that
\begin{align}\label{gttf5.2}
\sup_{x\in\Omega}\,|{\mathbb{T}}f(x)|
&+\sup_{x\in\Omega}\,\Big\{\rho(x)^{1-\alpha}\big|\nabla({\mathbb{T}}f)(x)\big|\Big\}
\nonumber\\[4pt]
&\leq n^2\,C_{n,\alpha,\Omega}\Big\{C^{l-2}2^{(l-2)^2}2^l
\|\nu\|_{{\mathscr{C}}^\alpha(\partial\Omega)}+2^l\Big\}\times
\nonumber\\[4pt]
&\hskip 0.80in \times
2\|\nu\|_{{\mathscr{C}}^\alpha(\partial\Omega)}
\|P\|_{L^1(S^{n-1})}\|f\|_{{\mathscr{C}}^\alpha(\partial\Omega)}
\end{align}
for every $f\in{\mathscr{C}}^\alpha(\partial\Omega)$. Having
established \eqref{gttf5.2}, we now see that \eqref{jht6.55} holds
provided the constant $C\in(1,\infty)$ is chosen in such a way that
\begin{align}\label{gttf5.3}
n^2\,C_{n,\alpha,\Omega}\Big\{C^{l-2}2^{(l-2)^2}2^l
\|\nu\|_{{\mathscr{C}}^\alpha(\partial\Omega)}+2^l\Big\}
2\|\nu\|_{{\mathscr{C}}^\alpha(\partial\Omega)}\leq C^{l}2^{l^2}
\end{align}
for each odd number $l\in{\mathbb{N}}$, $l\geq 3$. Since
$2^{(l-2)^2}2^l\leq 2\cdot 2^{l^2}$ and $2^l\leq C^{l-2}2^{l^2}$, it
follows that the left-hand side of \eqref{gttf5.3} is $\leq
C(n,\alpha,\Omega) C^{l-2}2^{l^2}$. This, in turn, is majorized by
the right-hand side of \eqref{gttf5.3} granted that
$C\geq\max\big\{1,\,\sqrt{C(n,\alpha,\Omega)}\,\big\}$. In summary,
choosing $C$ in the manner just described, to begin with, ensures
that \eqref{jht6.55} holds.

Next, we aim to show that \eqref{jht6.55} continues to be valid if
the harmonicity condition on $P$ is dropped, that is, when
\begin{equation}\label{eq:123yt}
\mbox{$P(x)$ is a homogeneous polynomial in ${\mathbb{R}}^n$ of
degree $l\in2{\mathbb{N}}-1$.}
\end{equation}
Indeed, a standard fact about arbitrary homogeneous polynomials
$P(x)$ is the decomposition (cf. \cite[\S\,{3.1.2}, p.\,69]{St70})
\begin{equation}\label{eq:Staer}
\begin{array}{c}
P(x)=P_1(x)+|x|^2Q_1(x)\,\,\mbox{ for every $x$ in }\,\,{\mathbb{R}}^n,
\,\,\mbox{ where}
\\[4pt]
\mbox{$P_1,Q_1$ are homogeneous polynomials and $P_1$ is harmonic}.
\end{array}
\end{equation}
Hence, if $P(x)$ is a homogeneous polynomial of degree $l=2N+1$ in
${\mathbb{R}}^{n}$, for some $N\in{\mathbb{N}}_0$, not necessarily
harmonic, then by iterating \eqref{eq:Staer} we obtain
\begin{equation}\label{eq:Staer.2}
\begin{array}{c}
P(x)=\sum\limits_{j=1}^{N+1} |x|^{2(j-1)}P_j(x)\,\,\mbox{ for every $x$ in
}\,\,{\mathbb{R}}^n, \,\,\mbox{ where each $P_j$}
\\[4pt]
\mbox{is a harmonic homogeneous polynomial of degree $l-2(j-1)$}.
\end{array}
\end{equation}
Since the restrictions to the unit sphere of any two homogeneous
harmonic polynomials of different degrees are orthogonal in
$L^2(S^{n-1})$ (cf. \cite[\S\,3.1.1, p.\,69]{St70}), it follows from
\eqref{eq:Staer.2} that
\begin{equation}\label{eq:StaRDS}
\|P\|^2_{L^2(S^{n-1})}=\sum\limits_{j=1}^{N+1}\|P_j\|^2_{L^2(S^{n-1})}.
\end{equation}
In particular, for each $j$, H\"older's inequality and
\eqref{eq:StaRDS} permit us to estimate
\begin{equation}\label{eq:StaRDS.2}
\|P_j\|_{L^1(S^{n-1})}\leq c_n\|P_j\|_{L^2(S^{n-1})}\leq
c_n\|P\|_{L^2(S^{n-1})}.
\end{equation}
Combining \eqref{T-layer.44-GGG} and \eqref{eq:Staer.2}, for any
$x\in\Omega$ and $f\in{\mathscr{C}}^\alpha(\partial\Omega)$ we
obtain
\begin{equation}\label{T-layer.44abb}
{\mathbb{T}}f(x)
=\sum_{j=1}^{N+1}\int_{\partial\Omega}\frac{P_j(x-y)}{|x-y|^{n-1+(l-2(j-1))}}f(y)\,d\sigma(y),
\end{equation}
and each integral operator appearing in the sum above is constructed
according to the same blue-print as the original ${\mathbb{T}}$ in
\eqref{T-layer.44-GGG}, including the property that the intervening
homogeneous polynomial is harmonic. As such, repeated applications
of \eqref{jht6.55} yield
\begin{align}\label{jht-64igd}
\sup_{x\in\Omega}\,|\mathbb{T}f(x)|
+\sup_{x\in\Omega}\,\Big\{\rho(x)^{1-\alpha}\big|\nabla(\mathbb{T}f)(x)\big|\Big\}
\leq c_n l
C^{l}2^{l^2}\|P\|_{L^2(S^{n-1})}\|f\|_{{\mathscr{C}}^\alpha(\partial\Omega)},
\end{align}
for each $f\in{\mathscr{C}}^\alpha(\partial\Omega)$. Since if $C$ is
bigger than a suitable dimensional constant we have $c_n l\leq
C^{l}$ for all $l$'s, by eventually replacing $C$ by $C^2$ in
\eqref{jht-64igd}. Ultimately, with the help of Lemma~\ref{Lem-J1}
(while keeping \eqref{eq:tr55} in mind), we deduce that
\eqref{jhygff8533} holds for $\mathbb{T}_{+}$ in $\Omega_{+}$. That
${\mathbb{T}}_{-}$ also satisfies similar properties follows in a
similar manner, working in $\Omega_{-}$ (in place of $\Omega_{+}$),
which continues to be a domain of class ${\mathscr{C}}^{1+\alpha}$
with compact boundary.
\end{proof}

\begin{proof}[\underline{Proof of {\rm (e)}\,$\Rightarrow$\,{\rm (d)} in Theorem~\ref{Main-T1aa}}]
This is obvious, since the operators ${\mathscr{R}}^{\pm}_j$ from
\eqref{T-pv.4kh445} are particular cases of those considered in
\eqref{T-layer.44}.
\end{proof}

\begin{proof}[\underline{Proof of {\rm (d)}\,$\Rightarrow$\,{\rm (a)} in Theorem~\ref{Main-T1aa}}]
Since we are presently assuming that $\Omega$ is a {\rm UR} domain,
Theorem~\ref{Main-T2} applies in $\Omega_{\pm}$ and yields (bearing
\eqref{URUR} in mind) the following jump-formulas
\begin{align}\label{yt5tr}
\bigg(\mathscr{R}^{\pm}_jf\Big\lvert^{{}^{\rm
n.t.}}_{\partial\Omega_\pm}\bigg)(x)
=\mp\frac{1}{2}\nu_j(x)f(x)+\lim_{\varepsilon \to 0^+}
\int_{\partial\Omega\setminus B(x,\varepsilon)}(\partial_j
E_{\!{}_\Delta})(x-y)f(y)\,d\sigma(y),
\end{align}
for each $f\in L^p(\partial\Omega,\sigma)$ with $p\in[1,\infty)$,
each $j\in\{1,\dots,n\}$, and $\sigma$-a.e. $x\in\partial\Omega$.
Hence, by \eqref{yt5tr} and \eqref{eq:RIEtt}, we have
\begin{align}\label{star}
\nu_j=\mathscr{R}^{-}_j1\Big|_{\partial\Omega_{-}}
-\mathscr{R}^{+}_j1\Big|_{\partial\Omega_{+}}\in{\mathscr{C}}^\alpha(\partial\Omega),
\qquad\forall\,j\in\{1,\dotsc,n\}.
\end{align}
Given the present background assumptions on $\Omega$,
Theorem~\ref{Th-C1} then gives that $\Omega$ is a
${\mathscr{C}}^{1+\alpha}$ domain.
\end{proof}

\begin{proof}[\underline{Proof of {\rm (a)}\,$\Rightarrow$\,{\rm (c)} in Theorem~\ref{Main-T1aa}}]
Assume that $\Omega$ is a domain of class
${\mathscr{C}}^{1+\alpha}$, $\alpha\in(0,1)$, with compact boundary.
Here, the task is to prove that the principal value singular
integral operator $T$, originally defined in \eqref{T-pv.44}, is a
well-defined, linear and bounded mapping from
${\mathscr{C}}^{\alpha}(\partial\Omega)$ into itself. In the
process, we shall also show that \eqref{jhygff8533.2} holds. Since
{\rm (a)}\,$\Rightarrow$\,{\rm (e)} has already been established, we
know that the singular integral operator \eqref{T-layer.44-GGG} maps
${\mathscr{C}}^\alpha(\partial\Omega)$ boundedly into
${\mathscr{C}}^\alpha\big(\overline{\Omega}\,\big)$ with
\begin{align}\label{jhiytr987}
\big\|{\mathbb{T}}f\big\|_{{\mathscr{C}}^\alpha(\overline{\Omega}\,)}
\leq
C^{l}2^{l^2}\|P\|_{L^2(S^{n-1})}\|f\|_{{\mathscr{C}}^\alpha(\partial\Omega)},\qquad
\forall\,f\in{\mathscr{C}}^\alpha(\partial\Omega).
\end{align}

For starters, let us operate under the additional assumption that
the homogeneous polynomial $P$ is also harmonic, and abbreviate
\begin{equation}\label{eq:pytg9}
k(x):=\frac{P(x)}{|x|^{n-1+l}},\qquad\forall\,x\in{\mathbb{R}}^n\setminus\{0\}.
\end{equation}
In such a scenario, \eqref{eq:yre34} gives that
\begin{equation}\label{eq:yre34.rreE}
\widehat{k}(\xi)={\mathcal{F}}_{x\to\xi}\Big(\frac{P(x)}{|x|^{n+l-1}}\Big)
=\gamma_{n,l,1}\,\frac{P(\xi)}{|\xi|^{l+1}},\qquad\forall\,\xi\in{\mathbb{R}}^n\setminus\{0\}.
\end{equation}
Moreover, a direct computation using Stirling's approximation
formula
\begin{equation}\label{eq:resre44}
\sqrt{2\pi}\,m^{m+1/2}e^{-m}\leq m!\leq
e\,m^{m+1/2}e^{-m},\qquad\forall\,m\in{\mathbb{N}},
\end{equation}
shows that
\begin{equation}\label{eq:STIL6433}
\gamma_{n,l,1}=\left\{
\begin{array}{ll}
O(l^{-(n-2)/2}) & \mbox{ if $n$ even},
\\[4pt]
O(l^{-(n-4)/2}) & \mbox{ if $n$ odd},
\end{array}
\right. \qquad\mbox{as }\,\,l\to\infty.
\end{equation}
We continue by observing that, thanks to \eqref{eq:rFG},
\begin{equation}\label{eq:P1aBB.1}
\sup_{x\in\partial\Omega}|P(\nu(x))|\leq\|P\|_{L^\infty(S^{n-1})}\leq
c_n 2^l l^{-1}\|P\|_{L^1(S^{n-1})}.
\end{equation}
Next we note that $|\nu(x)-\nu(y)|\geq 1/2$ forces $|x-y|^\alpha\geq
1/(2\|\nu\|_{{\mathscr{C}}^\alpha(\partial\Omega)})$ which further
implies
\begin{align}\label{eq:P1aBB.2}
\frac{|P(\nu(x))-P(\nu(y))|}{|x-y|^\alpha} &\leq
4\|\nu\|_{{\mathscr{C}}^\alpha(\partial\Omega)}\|P\|_{L^\infty(S^{n-1})}
\nonumber\\[4pt]
&\leq c_n 2^l
l^{-1}\|\nu\|_{{\mathscr{C}}^\alpha(\partial\Omega)}\|P\|_{L^1(S^{n-1})},
\end{align}
by virtue of \eqref{eq:rFG}, while if $|\nu(x)-\nu(y)|\leq 1/2$ the
Mean Value Theorem and \eqref{eq:rFG} permit us to once again
estimate
\begin{align}\label{eq:P1aBB.3}
\frac{|P(\nu(x))-P(\nu(y))|}{|x-y|^\alpha} &\leq
\Big(\sup_{z\in[\nu(x),\nu(y)]}|(\nabla
P)(z)|\Big)\|\nu\|_{{\mathscr{C}}^\alpha(\partial\Omega)}
\nonumber\\[4pt]
&\leq\|\nabla
P\|_{L^\infty(S^{n-1})}\|\nu\|_{{\mathscr{C}}^\alpha(\partial\Omega)}
\nonumber\\[4pt]
&\leq c_n 2^l
l^{-1}\|\nu\|_{{\mathscr{C}}^\alpha(\partial\Omega)}\|P\|_{L^1(S^{n-1})}.
\end{align}
By combining \eqref{eq:yre34.rreE} and
\eqref{eq:STIL6433}-\eqref{eq:P1aBB.3} we therefore arrive at the
conclusion that
\begin{equation}\label{eq:OBS}
\begin{array}{c}
\mbox{the mapping $\partial\Omega\ni
x\mapsto\widehat{k}(\nu(x))\in{\mathbb{C}}$ belongs to
${\mathscr{C}}^\alpha(\partial\Omega)$}
\\[6pt]
\mbox{ and
$\big\|\widehat{k}(\nu(\cdot))\big\|_{{\mathscr{C}}^\alpha(\partial\Omega)}
\leq c_n
2^l\|\nu\|_{{\mathscr{C}}^\alpha(\partial\Omega)}\|P\|_{L^1(S^{n-1})}$}.
\end{array}
\end{equation}
Next, the assumptions on $\Omega$ imply (cf. the discussion in
\S\ref{S-2}) that this is both a {\rm UR} domain and a uniform
domain. As such, Theorem~\ref{Main-T2} applies. Since ${\mathbb{T}}$
from \eqref{T-layer.44-GGG} corresponds to the operator
${\mathcal{T}}$ defined as in \eqref{T-layer} with $k$ as in
\eqref{eq:pytg9}, for each
$f\in{\mathscr{C}}^\alpha(\partial\Omega)$ we obtain from
\eqref{main-jump}, \eqref{eq:OBS}, and \eqref{jhiytr987} that
\begin{align}\label{iutvv}
\|Tf\|_{{\mathscr{C}}^\alpha(\partial\Omega)}
&\leq\big\|\tfrac{1}{2i}\widehat{k}(\nu(\cdot))f+Tf\big\|_{{\mathscr{C}}^\alpha(\partial\Omega)}
+\big\|\tfrac{1}{2i}\widehat{k}(\nu(\cdot))f\big\|_{{\mathscr{C}}^\alpha(\partial\Omega)}
\nonumber\\[4pt]
&\leq \big\|{\mathbb{T}}f\big|^{{}^{\rm
n.t.}}_{\partial\Omega}\big\|_{{\mathscr{C}}^\alpha(\partial\Omega)}
+2^{-1}\big\|\widehat{k}(\nu(\cdot))\big\|_{{\mathscr{C}}^\alpha(\partial\Omega)}
\|f\|_{{\mathscr{C}}^\alpha(\partial\Omega)}
\nonumber\\[4pt]
&=\big\|{\mathbb{T}}f\big|_{\partial\Omega}\big\|_{{\mathscr{C}}^\alpha(\partial\Omega)}
+c_n
2^l\|\nu\|_{{\mathscr{C}}^\alpha(\partial\Omega)}\|P\|_{L^1(S^{n-1})}
\|f\|_{{\mathscr{C}}^\alpha(\partial\Omega)}
\nonumber\\[4pt]
&\leq
\big\|{\mathbb{T}}f\big\|_{{\mathscr{C}}^\alpha(\overline{\Omega}\,)}
+c_n
2^l\|\nu\|_{{\mathscr{C}}^\alpha(\partial\Omega)}\|P\|_{L^2(S^{n-1})}
\|f\|_{{\mathscr{C}}^\alpha(\partial\Omega)}
\nonumber\\[4pt]
&\leq\Big\{C^{l}2^{l^2}+c_n
2^l\|\nu\|_{{\mathscr{C}}^\alpha(\partial\Omega)}\Big\}
\|P\|_{L^2(S^{n-1})}\|f\|_{{\mathscr{C}}^\alpha(\partial\Omega)}
\nonumber\\[4pt]
&\leq(C^2)^{l}2^{l^2}\|P\|_{L^2(S^{n-1})}\|f\|_{{\mathscr{C}}^\alpha(\partial\Omega)},
\end{align}
assuming, without loss of generality, that $C\geq
2+c_n\|\nu\|_{{\mathscr{C}}^\alpha(\partial\Omega)}$ to begin with.
Note that the estimate just derived has the format demanded in
\eqref{jhygff8533.2}.

To treat the general case when $P$ is merely as in \eqref{eq:123yt},
consider the decomposition \eqref{eq:Staer.2} and, for each
$f\in{\mathscr{C}}^\alpha(\partial\Omega)$, write
\begin{equation}\label{T-layer.44abb.i}
Tf(x) =\sum_{j=1}^{N+1}\lim_{\varepsilon\to 0^{+}}
\int\limits_{\stackrel{y\in\partial\Omega}{|x-y|>\varepsilon}}
\frac{P_j(x-y)}{|x-y|^{n-1+(l-2(j-1))}}f(y)\,d\sigma(y),\quad
x\in\partial\Omega.
\end{equation}
Since  every integral operator appearing in the right-hand side of
\eqref{T-layer.44abb.i} is of the same type as the original $T$ in
\eqref{T-pv.44}, with the additional property that the intervening
homogeneous polynomial is harmonic, repeated applications of
\eqref{iutvv} give
\begin{align}\label{jht-64igd-44}
\|Tf\|_{{\mathscr{C}}^\alpha(\partial\Omega)} \leq l
(C^2)^{l}2^{l^2}\|P\|_{L^2(S^{n-1})}\|f\|_{{\mathscr{C}}^\alpha(\partial\Omega)},
\qquad\forall\,f\in{\mathscr{C}}^\alpha(\partial\Omega).
\end{align}
Using $l\leq (C^2)^l$ for all $l$'s if $C$ is sufficiently large and
re-denoting $C^4$ simply as $C$, estimate \eqref{jhygff8533.2}
finally follows.
\end{proof}

\begin{proof}[\underline{Proof of {\rm (c)}\,$\Rightarrow$\,{\rm (b)} in Theorem~\ref{Main-T1aa}}]
Observe that the principal value Riesz transform $R^{{}^{\rm pv}}_j$ from \eqref{T-pv.4khg}
with $\Sigma:=\partial\Omega$ are special cases of the principal value singular integral
operators defined in \eqref{T-pv.44} (corresponding to $P$ as in \eqref{eq:Pjah}).
Hence, on the one hand, $R^{{}^{\rm pv}}_j1\in{\mathscr{C}}^\alpha(\partial\Omega)$.
On the other hand, since $\Omega$ is presently assumed to be a {\rm UR} domain, from \eqref{uytggf-RRR}
it follows that each of the distributional Riesz transforms $R_j$ from \eqref{yrf56f-RRR}-\eqref{yrf56f-RRR.1}
with $\Sigma:=\partial\Omega$ agrees with $R^{{}^{\rm pv}}_j$ on ${\mathscr{C}}^\alpha(\partial\Omega)$.
Combining these, we conclude that \eqref{eq:RIESZ33} holds.
\end{proof}

\begin{proof}[\underline{Proof of {\rm (b)}\,$\Rightarrow$\,{\rm (a)} in Theorem~\ref{Main-T1aa}}]
Granted the background hypotheses on $\Omega$, the assumption made in \eqref{eq:RIESZ33} allows
us to invoke the $T(1)$ Theorem (for operators associated with odd kernels, on spaces of homogeneous type).
Thanks to this, \eqref{TaGv-kgggg.222}, and the Calder\'on-Zygmund machinery mentioned earlier,
we conclude that each of the distributional Riesz transforms $R_j$ from
\eqref{yrf56f-RRR}-\eqref{yrf56f-RRR.1} with $\Sigma:=\partial\Omega$
extends to a bounded linear operator on $L^2(\partial\Omega)$, in the form of the principal
value Riesz transform $R^{{}^{\rm pv}}_j$ from \eqref{T-pv.4khg} with $\Sigma:=\partial\Omega$.
In particular, we presently have
\begin{equation}\label{j6VGV-aVa}
R_j1=R^{{}^{\rm pv}}_j1\,\,\text{ in }\,\,L^2(\partial\Omega).
\end{equation}
Next, observe that since $\nu\odot\nu=-1$ at $\sigma$-a.e. point on $\partial\Omega$ and
$x-y=\sum\limits_{j=1}^n(x_j-y_j)e_j$ for every $x,y\in{\mathbb{R}}^n$,
from \eqref{Cau-C2}, \eqref{T-pv.4khg}, and \eqref{j6VGV-aVa} we obtain
\begin{equation}\label{eq:CHba.11}
{\mathcal{C}}^{{}^{\rm pv}}\!\nu=-\sum_{j=1}^n\big(R^{{}^{\rm pv}}_j1\big)e_j
=\sum_{j=1}^n (R_j1)e_j\,\,\mbox{ at $\sigma$-a.e. point on }\,\,\partial\Omega
\end{equation}
which, on account of \eqref{eq:CsquaTT}, further yields
\begin{equation}\label{eq:Csq22}
\tfrac{1}{4}\nu={\mathcal{C}}^{{}^{\rm pv}}\!({\mathcal{C}}^{{}^{\rm pv}}\!\nu)=-{\mathcal{C}}^{{}^{\rm pv}}\!\Big(\sum_{j=1}^n
(R_j1)e_j\Big) \,\,\mbox{ at $\sigma$-a.e. point on }\,\,\partial\Omega.
\end{equation}
With this in hand, it readily follows from Theorem~\ref{i65r5ED}
that if condition \eqref{eq:RIESZ33} holds then
$\nu\in{\mathscr{C}}^\alpha(\partial\Omega)$. Having
established this, Theorem~\ref{Th-C1} applies and gives that
$\Omega$ is a domain of class ${\mathscr{C}}^{1+\alpha}$.
\end{proof}

This concludes the proof of Theorem~\ref{Main-T1aa}, and we now turn
to the proof of Theorem~\ref{THM-main222}.

\vskip 0.08in
\begin{proof}[Proof of Theorem~\ref{THM-main222}]
This is a direct consequence of Proposition~\ref{treePUb}
and Corollary~\ref{Main-T1aa-CCCa} upon observing that
${\mathfrak{C}}^{{}^{\rm pv}}=iR^{{}^{\rm pv}}_1+R^{{}^{\rm pv}}_2$, where
$R^{{}^{\rm pv}}_j$, $j=1,2$, are the two principal value Riesz
transforms in the plane.
\end{proof}

We finally present the proof of Theorem~\ref{Main-T1aBBB}.

\vskip 0.08in
\begin{proof}[Proof of Theorem~\ref{Main-T1aBBB}]
Let
\begin{equation}\label{eq:yTRf}
k\big|_{S^{n-1}}=\sum_{l=0}^\infty Y_l
\end{equation}
be the decomposition of $k\big|_{S^{n-1}}\in L^2(S^{n-1})$ in
surface spherical harmonics. That is, $\{Y_l\}_{l\in{\mathbb{N}}_0}$
are mutually orthogonal functions in $L^2(S^{n-1})$ with the
property that for each $l\in{\mathbb{N}}_0$ the function
\begin{equation}\label{eq:y44r}
P_l(x):=\left\{
\begin{array}{ll}
|x|^lY_l(x/|x|) & \mbox{ if }\,\,x\in{\mathbb{R}}^n\setminus\{0\},
\\[4pt]
0 & \mbox{ if }\,\,x=0,
\end{array}
\right.
\end{equation}
is a homogeneous harmonic polynomial of degree $l$ in
${\mathbb{R}}^n$. In particular,
\begin{align}\label{a-97t555}
\Delta_{S^{n-1}}Y_l=-l(l+n-2)Y_l\,\,\mbox{ on }\,\,S^{n-1},
\qquad\forall\,l\in{\mathbb{N}}_0.
\end{align}
See, for example, \cite[pp.\,68-70]{St70} for a discussion. Then for
each $l\in{\mathbb{N}}_0$ we may write
\begin{align}\label{a-97tgf}
\big[-l(l+n-2)\big]^{m_l} & \|Y_l\|^2_{L^2(S^{n-1})}
\nonumber\\[4pt]
&=\big[-l(l+n-2)\big]^{m_l}\int_{S^{n-1}}k\overline{Y_l}\,d\omega
\nonumber\\[4pt]
&=\int_{S^{n-1}}k\Delta^{m_l}_{S^{n-1}}\overline{Y_l}\,d\omega
=\int_{S^{n-1}}\big(\Delta^{m_l}_{S^{n-1}}k\big)\overline{Y_l}\,d\omega,
\end{align}
where the first equality uses \eqref{eq:yTRf}, the second one is
based on \eqref{a-97t555}, and the third one follows via repeated
integrations by parts. In turn, from \eqref{a-97tgf} and the
Cauchy-Schwarz inequality we obtain
\begin{align}\label{a-97tgf.2}
\|Y_l\|_{L^2(S^{n-1})}\leq l^{-2m_l}
\big\|\Delta^{m_l}_{S^{n-1}}k\big\|_{L^2(S^{n-1})},
\qquad\forall\,l\in{\mathbb{N}}_0.
\end{align}
We continue by noting that the homogeneity of $k$ together with
\eqref{eq:yTRf} and \eqref{eq:y44r} permit us to express
\begin{align}\label{eq:yr3}
k(x) &=\frac{k(x/|x|)}{|x|^{n-1}}
=\sum_{l=0}^\infty\frac{Y_l(x/|x|)}{|x|^{n-1}}
=\sum_{l=0}^\infty\frac{P_l(x/|x|)}{|x|^{n-1}}
=\sum_{l=0}^\infty\frac{P_l(x)}{|x|^{n-1+l}},
\end{align}
for each $x\in{\mathbb{R}}^n\setminus\{0\}$. For each
$l\in{\mathbb{N}}_0$, let ${\mathbb{T}}_l$, $T_l$ be the integral
operators defined analogously to \eqref{T-layer.44-BaB} and
\eqref{T-pv.44-BaB} in which the kernel $k(x-y)$ has been replaced
by $P_l(x-y)|x-y|^{-(n-1+l)}$. Then for each
$f\in{\mathscr{C}}^\alpha(\partial\Omega)$ we may estimate
\begin{align}\label{a-97tgf.3}
\sum_{l=0}^\infty\|{\mathbb{T}}_lf\|_{{\mathscr{C}}^\alpha(\overline{\Omega})}
&\leq\sum_{l=0}^\infty C^l 2^{l^2}\|P_l\|_{L^2(S^{n-1})}
\|f\|_{{\mathscr{C}}^\alpha(\partial\Omega)}
\nonumber\\[4pt]
&=\sum_{l=0}^\infty C^l 2^{l^2}\|Y_l\|_{L^2(S^{n-1})}
\|f\|_{{\mathscr{C}}^\alpha(\partial\Omega)}
\nonumber\\[4pt]
&\leq\Bigg\{\sum_{l=0}^\infty C^l 2^{l^2}l^{-2m_l}
\big\|\Delta^{m_l}_{S^{n-1}}k\big\|_{L^2(S^{n-1})}\Bigg\}
\|f\|_{{\mathscr{C}}^\alpha(\partial\Omega)},
\end{align}
by invoking \eqref{jhygff8533}, \eqref{a-97tgf.2}, and keeping in
mind that $P\big|_{S^{n-1}}=Y_l$ (cf. \eqref{eq:y44r}). Since for
$l$ large we have $C^l 2^{l^2}\leq 4^{l^2}$, it follows from
\eqref{eq:Reac} that the series in the curly bracket in
\eqref{a-97tgf.3} is convergent to some finite constant $M$. Based
on this and \eqref{eq:yr3} we may then conclude that
$\|{\mathbb{T}}f\|_{{\mathscr{C}}^\alpha(\overline{\Omega})}
\leq\sum_{l=0}^\infty\|{\mathbb{T}}_lf\|_{{\mathscr{C}}^\alpha(\overline{\Omega})}
\leq M\|f\|_{{\mathscr{C}}^\alpha(\partial\Omega)}$. This proves the
boundedness of the first operator in \eqref{maKP-BaB}, and the
second operator in \eqref{maKP-BaB} is treated similarly (making use
of \eqref{jhygff8533.2}).
\end{proof}

\begin{remark}\label{yafv}
We claim that condition \eqref{eq:Reac} is satisfied whenever the
kernel $k$ is of the form $P(x)/|x|^{n-1+l_o}$ for some homogeneous
polynomial $P$ of degree $l_o\in 2{\mathbb{N}}-1$ in
${\mathbb{R}}^n$. Indeed, writing
$P(x)/|x|^{n-1+l_o}=P(x/|x|)/|x|^{n-1}$ and invoking
\eqref{eq:Staer.2}, there is no loss of generality in assuming that
$P$ is also harmonic to begin with. Granted this, it follows that
$k\big|_{S^{n-1}}=P\big|_{S^{n-1}}$ is a surface spherical harmonic
of degree $l_o$, hence {\rm (}cf. \cite[\S\,3.1.4, p.\,70]{St70}{\rm
)}
$\Delta_{S^{n-1}}\big(k\big|_{S^{n-1}}\big)=-l_o(l_o+n-2)\big(k\big|_{S^{n-1}}\big)$.
Choosing $m_l:=l^2$ for each $l\in{\mathbb{N}}_0$ and iterating this
formula then shows that the series in \eqref{eq:Reac} is dominated
by
\begin{equation}\label{eq:Reac22}
\sum_{l=0}^\infty
4^{l^2}l^{-2l^2}\big[l_o(l_o+n-2)\big]^{l^2}\|k\|_{L^2(S^{n-1})}<+\infty.
\end{equation}
\end{remark}

\section{Further Results}
\setcounter{equation}{0} \label{S-7}

We start by recalling some definitions. First, given a compact
Ahlfors regular set $\Sigma\subset{\mathbb{R}}^n$ introduce
$\sigma:={\mathcal{H}}^{n-1}\lfloor\Sigma$ and define the
John-Nirenberg space of functions of bounded mean oscillations on
$\Sigma$ as
\begin{equation}\label{4}
{\rm BMO}(\Sigma):=\big\{f\in L^1(\Sigma,\sigma):\,f^{\#,\,p}\in
L^\infty(\Sigma,\sigma)\big\},
\end{equation}
where $p\in[1,\infty)$ is a fixed parameter and
\begin{equation}\label{5}
f^{\#,\,p}(x):=\sup\limits_{r>0}\Bigg(\frac{1}{\sigma(\Sigma\cap
B(x,r))} \int\limits_{\Sigma\cap
B(x,r)}|f(y)-f_{\Delta(x,r)}|^p\,d\sigma(y)\Bigg)^{\frac1p},
\end{equation}
with $f_{\Delta(x,r)}$ the mean value of $f$ on $\Sigma\cap B(x,r)$.
As is well known, various choices of $p$ give the same space.
Keeping this in mind, we define the seminorm
\begin{equation}\label{6}
[f]_{{\rm BMO}(\Sigma)}:=\|f^{\#,\,p}\|_{L^\infty(\Sigma,\sigma)}.
\end{equation}
We then define the Sarason space ${\rm VMO}(\Sigma)$ of functions of
vanishing mean oscillations on $\Sigma$ as the closure in ${\rm
BMO}(\Sigma)$ of ${\mathscr{C}}^{0}(\Sigma)$, the space of
continuous functions on $\Sigma$. Alternatively, given any
$\alpha\in(0,1)$, the space ${\rm VMO}(\Sigma)$ may be described
(cf. \cite[Proposition~2.15, p.\,2602]{HoMiTa10}) as the closure in
${\rm BMO}(\Sigma)$ of ${\mathscr{C}}^{\alpha}(\Sigma)$. Hence, in
the present context,
\begin{equation}\label{eq:12123}
\bigcup_{0\leq\alpha<1}{\mathscr{C}}^{\alpha}(\Sigma)
\hookrightarrow {\rm VMO}(\Sigma)\hookrightarrow
{\rm BMO}(\Sigma)\hookrightarrow\bigcap_{0<p<\infty}
L^p(\Sigma,\sigma).
\end{equation}

\begin{proposition}\label{MPjab}
If $\Omega\subseteq\mathbb{R}^n$ is a {\rm UR} domain with compact
boundary then the principal value Cauchy-Clifford operator
${\mathcal{C}}^{{}^{\rm pv}}$ from \eqref{Cau-C2} is bounded both on
${\rm BMO}(\partial\Omega)\otimes{\mathcal{C}}\!\ell_{n}$, as well as on
${\rm VMO}(\partial\Omega)\otimes{\mathcal{C}}\!\ell_{n}$. Moreover,
$\big({\mathcal{C}}^{{}^{\rm pv}}\big)^2=\tfrac{1}{4}I$ both on
${\rm BMO}(\partial\Omega)\otimes{\mathcal{C}}\!\ell_{n}$ and on
${\rm VMO}(\partial\Omega)\otimes{\mathcal{C}}\!\ell_{n}$. Hence, in
particular, ${\mathcal{C}}^{{}^{\rm pv}}$ is an isomorphism when acting
on either of these spaces.
\end{proposition}

\begin{proof}
To begin with, observe that in the present setting
\eqref{eq:CsqASd.2tTT} ensures that ${\mathcal{C}}^{{}^{\rm pv}}$
is well-defined on ${\rm BMO}(\partial\Omega)\otimes{\mathcal{C}}\!\ell_{n}$.
Fix now $f\in{\rm BMO}(\partial\Omega)\otimes{\mathcal{C}}\!\ell_{n}$
and pick some $x_0\in\partial\Omega$ and $r>0$. For each $R>0$, let
us agree to abbreviate $\Delta_{R}:=\partial\Omega\cap B(x_0,R)$.
Denote by $\nu$ the geometric measure theoretic outward unit normal
to $\Omega$ and, with $\sigma:={\mathcal{H}}^{n-1}\lfloor\partial\Omega$,
introduce
\begin{equation}\label{eq:}
A(x_0,r):=\frac{1}{\omega_{n-1}}
\int\limits_{\stackrel{y\in\partial\Omega}{|x_0-y|\geq 2r}}
\frac{x_0-y}{|x_0-y|^n}\odot\nu(y)\odot\big(f(y)-f_{\Delta_{2r}}\big)\,d\sigma(y)
\pm\frac{1}{2}f_{\Delta_{2r}}
\end{equation}
where the sign is chosen to be plus if $\Omega$ is bounded and minus
if $\Omega$ is unbounded, and where $f_{\Delta_{2r}}$ stands for the
integral average of $f$ over $\Delta_{2r}$. For $x\in\Delta_{r}$ use
\eqref{eq:CsqASd.2tTT} to split
\begin{align}\label{eq:3rf5f}
{\mathcal{C}}^{{}^{\rm pv}}\!f(x)&=\lim_{\varepsilon\to 0^{+}}
\frac{1}{\omega_{n-1}}\int\limits_{\stackrel{y\in\partial\Omega\setminus
B(x,\varepsilon)}{|x_0-y|<2r}}
\frac{x-y}{|x-y|^n}\odot\nu(y)\odot\big(f(y)-f_{\Delta_{2r}}\big)\,d\sigma(y)
\nonumber\\[4pt]
&\quad+\frac{1}{\omega_{n-1}}\int\limits_{\stackrel{y\in\partial\Omega}{|x_0-y|\geq
2r}} \Big(\frac{x-y}{|x-y|^n}-\frac{x_0-y}{|x_0-y|^n}\Big)
\odot\nu(y)\odot\big(f(y)-f_{\Delta_{2r}}\big)\,d\sigma(y)
\nonumber\\[4pt]
&\quad+A(x_0,r),
\end{align}
then employ this representation (and Minkowski's inequality) in
order to estimate
\begin{align}\label{eq:3rf5f.2}
\Bigg(\frac{1}{\sigma(\Delta_r)}
\int_{\Delta_r}\big|{\mathcal{C}}^{{}^{\rm pv}}\!f(x)-A(x_0,r)\big|^2\,d\sigma(x)\Bigg)^{\frac12}
\leq c\big(I+II\big),
\end{align}
where $c\in(0,\infty)$ depends only on $\Omega$ and
\begin{equation*}
I:=\Bigg(\frac{1}{\sigma(\Delta_r)}
\int_{\partial\Omega}\big|{\mathcal{C}}^{{}^{\rm pv}}\!\big((f-f_{\Delta_{2r}})
{\mathbf{1}}_{\Delta_{2r}}\big)\big|^2\,d\sigma\Bigg)^{\frac12}
\end{equation*}
and
\begin{align*}
II:=r^{-\frac{n-1}{2}}\!\!\!\!
\int\limits_{\stackrel{y\in\partial\Omega}{|x_0-y|\geq 2r}}
\Bigg(\int_{\Delta_r}\Big|\frac{x-y}{|x-y|^n}-\frac{x_0-y}{|x_0-y|^n}\Big|^2
\,d\sigma(x)\Bigg)^{\frac12}\big|f(y)-f_{\Delta_{2r}}\big|\,d\sigma(y).
\end{align*}
Now, the boundedness of ${\mathcal{C}}^{{}^{\rm pv}}$ on
$L^2(\partial\Omega,\sigma)\otimes{\mathcal{C}}\!\ell_{n}$ from
Proposition~\ref{iTRfc} gives (bearing in mind that $\sigma$ is
doubling)
\begin{align}\label{eq:3rf5f.5}
I\leq c\Bigg(\frac{1}{\sigma(\Delta_{2r})}
\int_{\Delta_{2r}}\big|f-f_{\Delta_{2r}}\big|^2\,d\sigma\Bigg)^{\frac12}
\leq c[f]_{{\rm BMO}(\partial\Omega)\otimes{\mathcal{C}}\!\ell_{n}},
\end{align}
which suits our purposes. Next, we write
\begin{align}\label{eq:3rf5f.7}
II & \leq c\int\limits_{\stackrel{y\in\partial\Omega}{|x_0-y|\geq
2r}} \frac{r}{|x_0-y|^n}\big|f(y)-f_{\Delta_{2r}}\big|\,d\sigma(y)
\nonumber\\[4pt]
& \leq
c\sum_{j=1}^{\infty}\int_{\Delta_{2^{j+1}r}\setminus\Delta_{2^jr}}
\frac{r}{(2^jr)^n}|f(y)-f_{\Delta_{2r}}|\,d\sigma(y)
\nonumber\\[4pt]
&\leq c\sum_{j=1}^{\infty}\frac{1}{2^j}\meanint_{\Delta_{2^{j+1}r}}
|f-f_{\Delta_{2r}}|\,d\sigma
\nonumber\\[4pt]
&\leq c\sum_{j=1}^{\infty}\frac{1}{2^j}\meanint_{\Delta_{2^{j+1}r}}
\Bigl[|f-f_{\Delta_{2^{j+1}r}}|
+\sum_{k=1}^j|f_{\Delta_{2^{k+1}r}}-f_{\Delta_{2^kr}}|\Bigr]\,d\sigma
\nonumber\\[4pt]
&\leq c\sum_{j=1}^{\infty}\frac{1}{2^j}(1+j)f^{\#,\,1}(x_0) \leq
cf^{\#,\,1}(x_0) \leq c[f]_{{\rm
BMO}(\partial\Omega)\otimes{\mathcal{C}}\!\ell_{n}}.
\end{align}
Above, the first inequality follows from the Mean Value Theorem,
while the second inequality is a consequence of writing the integral
over $\partial\Omega\setminus\Delta_{2r}$ as the telescopic sum over
$\Delta_{2^{j+1}r}\setminus\Delta_{2^jr}$, $j\in{\mathbb{N}}$ and
the fact that $|x_0-y|\geq 2^jr$ for
$y\in\Delta_{2^{j+1}r}\setminus\Delta_{2^jr}$. The third inequality
is a result of enlarging the domain of integration from
$\Delta_{2^{j+1}r}\setminus\Delta_{2^jr}$ to $\Delta_{2^{j+1}r}$ and
using $\sigma\bigl(\Delta_{2^{j+1}r}\bigr)\approx(2^jr)^{n-1}$. The
fourth inequality follows from the triangle inequality after writing
\begin{equation}\label{taffee}
f-f_{\Delta_{2r}}=f-f_{\Delta_{2^{j+1}r}}+\sum\limits_{k=1}^j
(f_{\Delta_{2^{k+1}r}}-f_{\Delta_{2^kr}}).
\end{equation}
The fifth inequality is a consequence of the fact that, for each
$k$, we have
\begin{align}\label{taffee-2}
\bigl|f_{\Delta_{2^{k+1}r}}-f_{\Delta_{2^kr}}\bigr|
&=\Bigl|\meanint_{\Delta_{2^kr}}\big(f-f_{\Delta_{2^{k+1}r}}\big)\,d\sigma\Bigr|
\nonumber\\[4pt]
&\leq
c\meanint_{\Delta_{2^{k+1}r}}\big|f-f_{\Delta_{2^{k+1}r}}\big|\,d\sigma
\leq c\,f^{\#,\,1}(x_0).
\end{align}
The sixth inequality is a consequence of $\sum\limits_{j=1}^\infty
2^{-j}(1+j)<+\infty$ and, finally, the last inequality is seen from
\eqref{6}.

From \eqref{eq:3rf5f.2}-\eqref{eq:3rf5f.7} we eventually obtain
$\big\|\big({\mathcal{C}}^{{}^{\rm pv}}\!f\big)^{\#,\,2}\big\|
_{L^\infty(\partial\Omega,\sigma)\otimes{\mathcal{C}}\!\ell_{n}}
\leq c[f]_{{\rm BMO}(\partial\Omega)\otimes{\mathcal{C}}\!\ell_{n}}$,
hence
\begin{equation}\label{6trf}
\big[{\mathcal{C}}^{{}^{\rm pv}}\!f\big]_{{\rm BMO}(\partial\Omega)
\otimes{\mathcal{C}}\!\ell_{n}} \leq c[f]_{{\rm BMO}(\partial\Omega)
\otimes{\mathcal{C}}\!\ell_{n}}
\end{equation}
from which we conclude that the operator
\begin{equation}\label{6trf.11}
{\mathcal{C}}^{{}^{\rm pv}}:{\rm BMO}(\partial\Omega)
\otimes{\mathcal{C}}\!\ell_{n}\longrightarrow
{\rm BMO}(\partial\Omega)\otimes{\mathcal{C}}\!\ell_{n}
\end{equation}
is well-defined and bounded. Next, that
\begin{equation}\label{6trf.12}
{\mathcal{C}}^{{}^{\rm pv}}:{\rm VMO}(\partial\Omega)
\otimes{\mathcal{C}}\!\ell_{n}\longrightarrow
{\rm VMO}(\partial\Omega)\otimes{\mathcal{C}}\!\ell_{n}
\end{equation}
is also well-defined and bounded follows from \eqref{6trf.11}, the
characterization of ${\rm
VMO}(\partial\Omega)\otimes{\mathcal{C}}\!\ell_{n}$ as the closure
in ${\rm BMO}(\partial\Omega)\otimes{\mathcal{C}}\!\ell_{n}$ of
${\mathscr{C}}^{\alpha}(\partial\Omega)\otimes{\mathcal{C}}\!\ell_{n}$
for each $\alpha\in(0,1)$, and Theorem~\ref{i65r5ED}.

Finally, the claims in the last part of the statement of the
proposition are direct consequences of what we have proved so far,
\eqref{eq:12123}, and \eqref{eq:CsquaTT}.
\end{proof}

When $\Omega\subseteq\mathbb{R}^n$ is a {\rm UR} domain with compact
boundary, it follows from \eqref{uytggf-RRR} and \eqref{maKP56-BBN}
in Theorem~\ref{Main-T2-BBBB} that $R_j$ maps ${\mathscr{C}}^\alpha(\partial\Omega)$
into ${\rm BMO}(\partial\Omega)$ for each $j\in\{1,\dots,n\}$. Hence, in this case,
we have $R_j1\in{\rm BMO}(\partial\Omega)$ for each $j\in\{1,\dots,n\}$.
Remarkably, the proximity of the ${\rm BMO}$ functions $R_j1$, $1\leq j\leq n$,
to the space ${\rm VMO}(\partial\Omega)$ controls how close the outward
unit normal $\nu$ to $\Omega$ is to being in ${\rm VMO}(\partial\Omega)$.
Specifically, we have the following result.

\begin{theorem}\label{Kbagg}
Let $\Omega\subseteq\mathbb{R}^n$ be a {\rm UR} domain with compact
boundary and denote by $\nu$ the geometric measure theoretic outward
unit normal to $\Omega$. Also, let $\big\|{\mathcal{C}}^{{}^{\rm pv}}\big\|_\ast$ stand for
the operator norm of the Cauchy-Clifford singular integral operator
acting on the space ${\rm BMO}(\partial\Omega)\otimes{\mathcal{C}}\!\ell_{n}$.
Then, with distances considered in ${\rm BMO}(\partial\Omega)$, one has
\begin{align}\label{eq:iugf.AA}
& {\rm dist}\big(\nu\,,\,{\rm VMO}(\partial\Omega)\big)
\leq 4\big\|{\mathcal{C}}^{{}^{\rm pv}}\big\|_\ast\Big(\sum_{j=1}^n{\rm dist}
\big(R_j1\,,\,{\rm VMO}(\partial\Omega)\big)^2\Big)^{1/2},
\\[4pt]
& \Big(\sum_{j=1}^n{\rm dist}
\big(R_j1\,,\,{\rm VMO}(\partial\Omega)\big)^2\Big)^{1/2}
\leq\big\|{\mathcal{C}}^{{}^{\rm pv}}\big\|_\ast\,{\rm dist}
\big(\nu\,,\,{\rm VMO}(\partial\Omega)\big).
\label{eq:iugf.BB}
\end{align}
\end{theorem}

\begin{proof}
On the one had, based on \eqref{eq:Csq22},
Proposition~\ref{MPjab}, and the fact that each $R^{{}^{\rm pv}}_j$
agrees with $R_j$ on $L^2(\partial\Omega)$, we may estimate
\begin{align}\label{eq:Csq22.BBss}
{\rm dist}\big(\nu\,,\,{\rm VMO}(\partial\Omega)\big)
&=\inf_{\eta\in{\rm VMO}(\partial\Omega)}
[\nu-\eta]_{{\rm BMO}(\partial\Omega)}
\nonumber\\[4pt]
&=\inf_{\eta\in{\rm
VMO}(\partial\Omega)\otimes{\mathcal{C}}\!\ell_{n}} [\nu-\eta]_{{\rm
BMO}(\partial\Omega)\otimes{\mathcal{C}}\!\ell_{n}}
\nonumber\\[4pt]
&=\inf_{\eta\in{\rm
VMO}(\partial\Omega)\otimes{\mathcal{C}}\!\ell_{n}}
4\Big[{\mathcal{C}}^{{}^{\rm pv}}
\Big(\sum_{j=1}^n\big(R^{{}^{\rm pv}}_j1\big)e_j+{\mathcal{C}}^{{}^{\rm pv}}\!\eta\Big)
\Big] _{{\rm BMO}(\partial\Omega)\otimes{\mathcal{C}}\!\ell_{n}}
\nonumber\\[4pt]
&\leq 4\big\|{\mathcal{C}}^{{}^{\rm pv}}\big\|_\ast \inf_{\eta\in{\rm
VMO}(\partial\Omega)\otimes{\mathcal{C}}\!\ell_{n}}
\Big[\sum_{j=1}^n(R_j1)e_j+{\mathcal{C}}^{{}^{\rm pv}}\!\eta\Big]_{{\rm
BMO}(\partial\Omega)\otimes{\mathcal{C}}\!\ell_{n}}
\nonumber\\[4pt]
&=4\big\|{\mathcal{C}}^{{}^{\rm pv}}\big\|_\ast \inf_{\xi\in{\rm
VMO}(\partial\Omega)\otimes{\mathcal{C}}\!\ell_{n}}
\Big[\sum_{j=1}^n(R_j1)e_j-\xi\Big]_{{\rm
BMO}(\partial\Omega)\otimes{\mathcal{C}}\!\ell_{n}}
\nonumber\\[4pt]
&=4\big\|{\mathcal{C}}^{{}^{\rm pv}}\big\|_\ast
\inf_{\xi\in{\rm VMO}(\partial\Omega)}
\Big[\sum_{j=1}^n(R_j1)e_j-\xi\Big]_{{\rm BMO}(\partial\Omega)}
\nonumber\\[4pt]
&=4\big\|{\mathcal{C}}^{{}^{\rm pv}}\big\|_\ast \Big(\sum_{j=1}^n{\rm
dist}\big(R_j1\,,\,{\rm VMO}(\partial\Omega)\big)^2\Big)^{1/2},
\end{align}
yielding \eqref{eq:iugf.AA}. On the other hand, from
\eqref{eq:CHba.11} and Proposition~\ref{MPjab} we deduce
(once again by bearing in mind that each $R^{{}^{\rm pv}}_j$
agrees with $R_j$ on $L^2(\partial\Omega)$) that
\begin{align}\label{eq:CHbbab645}
\Big(\sum_{j=1}^n{\rm dist}
\big(R_j1\,,\,{\rm VMO}(\partial\Omega)\big)^2\Big)^{1/2}
&=\inf_{\xi\in{\rm VMO}(\partial\Omega)\otimes{\mathcal{C}}\!\ell_{n}}
\Big[\sum_{j=1}^n(R_j1)e_j-\xi\Big]
_{{\rm BMO}(\partial\Omega)\otimes{\mathcal{C}}\!\ell_{n}}
\nonumber\\[4pt]
&=\inf_{\xi\in{\rm VMO}(\partial\Omega)\otimes{\mathcal{C}}\!\ell_{n}}
\Big[\sum_{j=1}^n\big(R^{{}^{\rm pv}}_j1\big)e_j-\xi\Big]
_{{\rm BMO}(\partial\Omega)\otimes{\mathcal{C}}\!\ell_{n}}
\nonumber\\[4pt]
&=\inf_{\xi\in{\rm VMO}(\partial\Omega)\otimes{\mathcal{C}}\!\ell_{n}}
\Big[{\mathcal{C}}^{{}^{\rm pv}}\nu-\xi\Big]
_{{\rm BMO}(\partial\Omega)\otimes{\mathcal{C}}\!\ell_{n}}
\nonumber\\[4pt]
&=\inf_{\eta\in{\rm
VMO}(\partial\Omega)\otimes{\mathcal{C}}\!\ell_{n}}
\Big[{\mathcal{C}}^{{}^{\rm pv}}(\nu-\eta)\Big]
_{{\rm BMO}(\partial\Omega)\otimes{\mathcal{C}}\!\ell_{n}}
\nonumber\\[4pt]
&\leq\big\|{\mathcal{C}}^{{}^{\rm pv}}\big\|_\ast
\inf_{\eta\in{\rm VMO}(\partial\Omega)\otimes{\mathcal{C}}\!\ell_{n}}
\big[\nu-\eta\big]_{{\rm BMO}(\partial\Omega)\otimes{\mathcal{C}}\!\ell_{n}}
\nonumber\\[4pt]
&=\big\|{\mathcal{C}}^{{}^{\rm pv}}\big\|_\ast
\inf_{\eta\in{\rm VMO}(\partial\Omega)}[\nu-\eta]_{{\rm BMO}(\partial\Omega)}
\nonumber\\[4pt]
&=\big\|{\mathcal{C}}^{{}^{\rm pv}}\big\|_\ast\,
{\rm dist}\big(\nu\,,\,{\rm VMO}(\partial\Omega)\big),
\end{align}
finishing the justification of \eqref{eq:iugf.BB}.
\end{proof}

Having established Theorem~\ref{Kbagg}, we are now in a position to
present the proof of Theorem~\ref{Kbagg.A}.

\vskip 0.08in
\begin{proof}[Proof of Theorem~\ref{Kbagg.A}]
For the left-to-right implication in \eqref{eq:iugf.6r4}, start by
observing that $\Omega$ is a {\rm UR} domain (cf.
Definition~\ref{Def-UB}). As such, Theorem~\ref{Kbagg} applies and
\eqref{eq:iugf.BB} gives that $R_j1\in{\rm VMO}(\partial\Omega)$ for
each $j\in\{1,\dots,n\}$. For the right-to-left implication in
\eqref{eq:iugf.6r4}, use \eqref{Mabb88} and the background
assumptions on $\Omega$ to conclude that $\Omega$ is a {\rm UR}
domain, then invoke \eqref{eq:iugf.AA} from Theorem~\ref{Kbagg} to
conclude that $\nu\in{\rm VMO}(\partial\Omega)$.
\end{proof}

Moving on, we record the following definition.
\begin{definition}\label{Def-John}
Let $\Omega\subset{\mathbb{R}}^{n}$ be an open set with compact
boundary. Then $\Omega$ is said to satisfy a {\tt John} {\tt
condition} if there exist $\theta\in(0,1)$ and $R\in(0,\infty)$,
called the John constants of $\Omega$, with the following
significance. For every $p\in\partial\Omega$ and $r\in(0,R)$ one can
find $p_r\in B(p,r)\cap\Omega$ such that $B(p_r,\theta
r)\subset\Omega$ and with the property that for each $x\in
B(p,r)\cap\partial\Omega$ there exists a rectifiable path
$\gamma_x:[0,1]\to\overline{\Omega}$, whose length is
$\leq\theta^{-1}r$ and
\begin{eqnarray}\label{Loc-John1}
\gamma_x(0)=x,\quad\gamma_x(1)=p_r,\,\,\mbox{ and }\,\, {\rm
dist}\,(\gamma_x(t),\partial\Omega)>\theta\,|\gamma_x(t)-x|
\quad\forall\,t\in (0,1].
\end{eqnarray}
Furthermore, $\Omega$ is said to satisfy a {\tt two-sided John
condition} if both $\Omega$ and
${\mathbb{R}}^{n}\setminus\overline{\Omega}$ satisfy a John
condition.
\end{definition}

The above definition appears in \cite{HoMiTa10}, where it has been
noted that any {\rm NTA} domain (in the sense of D. Jerison and C.
Kenig; \cite{JeKe82}) with compact boundary satisfies a John
condition.

Next, we recall the concept of $\delta$-Reifenberg flat domain,
following \cite{KeTo99}, \cite{KeTo03}. As a preamble, the reader is
reminded that the Pompeiu-Hausdorff distance between two sets
$A,B\subseteq{\mathbb{R}}^{n}$ is given by
\begin{eqnarray}\label{R-3}
D[A,B]:=\max\Bigl\{\sup\{{\rm dist}(a,B):\,a\in A\}\,,\, \sup\{{\rm
dist}(b,A):\,b\in B\}\Bigr\}.
\end{eqnarray}

\begin{definition}\label{Def-R1}
Let $\Sigma\subset{\mathbb{R}}^{n}$ be a compact set and let
$\delta\in(0,\frac{1}{4\sqrt{2}})$. Call $\Sigma$ a $\delta$-{\tt
Reifenberg} {\tt flat} {\tt set} if there exists $R>0$ such that for
every $x\in\Sigma$ and every $r\in(0,R]$ there exists an
$(n-1)$-dimensional plane $L(x,r)$ which contains $x$ and such that
\begin{eqnarray}\label{R-2}
D\big[\Sigma\cap B(x,r)\,,\,L(x,r)\cap B(x,r)\big]\leq\delta r.
\end{eqnarray}
\end{definition}

\begin{definition}\label{Def-R3}
Say that a bounded open set $\Omega\subset{\mathbb{R}}^{n}$ has the
{\tt separation} {\tt property} if there exists $R>0$ such that for
every $x\in\partial\Omega$ and $r\in(0,R]$ there exists an
$(n-1)$-dimensional plane ${\mathcal{L}}(x,r)$ containing $x$ and a
choice of unit normal vector to ${\mathcal{L}}(x,r)$, call it
$\vec{n}_{x,r}$, satisfying
\begin{eqnarray}\label{RR-1}
\begin{array}{l}
\{y+t\,\vec{n}_{x,r}\in B(x,r):\,y\in{\mathcal{L}}(x,r),\,\,
t<-{\textstyle{\frac{r}{4}}}\}\subset\Omega,
\\[10pt]
\{y+t\,\vec{n}_{x,r}\in B(x,r):\,y\in{\mathcal{L}}(x,r),\,\,
t>{\textstyle{\frac{r}{4}}}\}\subset{\mathbb{R}}^{n}\setminus\Omega.
\end{array}
\end{eqnarray}
Moreover, if $\Omega$ is unbounded, it is also required that
$\partial\Omega$ divides ${\mathbb{R}}^{n}$ into two distinct
connected components and that ${\mathbb{R}}^{n}\setminus\Omega$ has
a non-empty interior.
\end{definition}

\begin{definition}\label{Def-R4}
Let $\Omega\subset{\mathbb{R}}^{n}$ be a bounded open set and
$\delta\in(0,\delta_n)$. Call $\Omega$ a $\delta$-{\tt Reifenberg}
{\tt flat} {\tt domain} if $\Omega$ has the separation property and
$\partial\Omega$ is a $\delta$-Reifenberg flat set.

The notion of {\tt Reifenberg} {\tt flat} {\tt domain} {\tt with}
{\tt vanishing} {\tt constant} is introduced in a similar fashion,
this time allowing the constant $\delta$ appearing in \eqref{R-2} to
depend on $r$, say $\delta=\delta(r)$, and demanding that
$\lim\limits_{r\to 0^{+}}\delta(r)=0$.
\end{definition}

As our next result shows, under appropriate background assumptions
(of a ``large" geometry nature) the proximity of the vector-valued
function $(R_11,R_21,\dots,R_n1)$ to the space ${\rm
VMO}(\partial\Omega)$, measured in ${\rm
BMO}(\partial\Omega)$, can be used to quantify
Reifenberg flatness.

\begin{theorem}\label{KbAAbb}
Assume $\Omega\subseteq\mathbb{R}^n$ is an open set with a compact
Ahlfors regular boundary, satisfying a two-sided John condition
{\rm (}hence, $\Omega$ is a {\rm UR} domain which further entails
that $R_j1\in{\rm BMO}(\partial\Omega)$ for each $j${\rm )}.
If, with distances considered in ${\rm BMO}(\partial\Omega)$,
\begin{equation}\label{eq:YYjygg}
\sum_{j=1}^n{\rm dist}\big(R_j1\,,\,{\rm
VMO}(\partial\Omega)\big)<\varepsilon
\end{equation}
then $\Omega$ is a $\delta$-Reifenberg flat domain for
$\delta=C_o\cdot\varepsilon$, where $C_o\in(0,\infty)$ depends only
on the Ahlfors regularity and John constants of $\Omega$.

As a consequence, if $R_j1\in{\rm VMO}(\partial\Omega)$ for every
$j\in\{1,\dots,n\}$ then actually $\Omega$ is a Reifenberg flat
domain with vanishing constant.
\end{theorem}

\begin{proof}
It is known that if $\Omega\subseteq\mathbb{R}^n$ is an open set
with a compact Ahlfors regular boundary, satisfying a two-sided John
condition, and such that
\begin{equation}\label{eq:YYjygg.ttt}
{\rm dist}\big(\nu\,,\,{\rm VMO}(\partial\Omega)\big)<\varepsilon
\end{equation}
(with the distance considered in ${\rm BMO}(\partial\Omega)$), then $\Omega$ is a
$\delta$-Reifenberg flat domain for the choice $\delta=C_o\cdot\varepsilon$,
where the constant $C_o\in(0,\infty)$ is as in the statement of the
theorem. See \cite[Definition~4.7 p.\,2690 and Corollary~4.20
p.\,2710]{HoMiTa10} in this regard. Granted this, the desired
conclusion follows by invoking Theorem~\ref{Kbagg}, since our
assumptions on $\Omega$ guarantee that this is a {\rm UR} domain
(cf. \eqref{eq:rTG866}).
\end{proof}

In this last part of this section we discuss a (partial) extension
of Theorem~\ref{Main-T1aa} in the context of Besov spaces. We begin
by defining this scale, and recalling some of its most basic
properties.

\begin{definition}\label{ndyay}
Assume that $\Sigma\subset{\mathbb{R}}^n$ is an Ahlfors regular set
and let $\sigma:={\mathcal{H}}^{n-1}\lfloor{\Sigma}$. Then, given
$1\leq p\leq\infty$ and $0<s<1$, define the Besov space
\begin{eqnarray}\label{b-UU.5}
B^{p,p}_s(\Sigma):=\big\{f\in
L^p(\Sigma,\sigma):\,\|f\|_{B^{p,p}_s(\Sigma)}<+\infty\big\}
\end{eqnarray}
where
\begin{eqnarray}\label{b-UU.6}
\|f\|_{B^{p,p}_s(\Sigma)}:=\|f\|_{L^p(\Sigma,\sigma)}+\Bigl(\int_{\Sigma}\int_{\Sigma}
\frac{|f(x)-f(y)|^p}{|x-y|^{n-1+sp}}\,d\sigma(x)\,d\sigma(y)\Bigr)^{1/p},
\end{eqnarray}
with the convention that
\begin{equation}\label{Ugagb64dc}
B^{\infty,\infty}_s(\Sigma):={\mathscr{C}}^s(\Sigma)\,\,\mbox{ and
}\,\,
\|f\|_{B^{\infty,\infty}_s(\Sigma)}:=\|f\|_{{\mathscr{C}}^s(\Sigma)}.
\end{equation}

Finally, denote by $B^{p,p}_{s,\,{\rm loc}}(\Sigma)$ the space of
functions whose truncations by smooth and compactly supported
functions belong to $B^{p,p}_s(\Sigma)$.
\end{definition}

Consider $\Sigma$ as in Definition~\ref{ndyay} and suppose $1\leq
p_0,p_1\leq\infty$ and $s_0,s_1\in(0,1)$ are such that
\begin{eqnarray}\label{akjhlu-2}
\frac{1}{p_1}-\frac{s_1}{n-1}=\frac{1}{p_0}-\frac{s_0}{n-1}\quad\mbox{and}\quad
s_0\geq s_1.
\end{eqnarray}
Then \cite[Proposition~5, p.\,213]{JoWa84} gives that
\begin{eqnarray}\label{akjhlu-1}
B^{p_0,p_0}_{s_0}(\Sigma)\hookrightarrow
B^{p_1,p_1}_{s_1}(\Sigma)\,\,\mbox{ continuously}.
\end{eqnarray}
In particular,
\begin{eqnarray}\label{akjhlu-eW}
B^{p,p}_{s}(\Sigma)\hookrightarrow{\mathscr{C}}^\alpha(\Sigma)\,\,\mbox{
if }\,\,
p\in[1,\infty],\,\,s\in(0,1),\,\,sp>n-1,\,\,\alpha:=s-\tfrac{n-1}{p}.
\end{eqnarray}
In turn, from \eqref{b-UU.5}-\eqref{b-UU.6} and \eqref{akjhlu-eW}
one may easily deduce that
\begin{eqnarray}\label{akjhlu-eW.2}
\mbox{if }\,\,p\in[1,\infty],\,\,s\in(0,1)\,\,\mbox{ satisfy
}\,\,sp>n-1 , \mbox{ then }\,\,B^{p,p}_{s}(\Sigma)\,\,\mbox{ is an
algebra},
\end{eqnarray}
and
\begin{eqnarray}\label{akjhlu-eW.2bb}
\begin{array}{c}
f/g\in B^{p,p}_{s}(\Sigma)\,\,\mbox{ whenever $f,g\in
B^{p,p}_{s}(\Sigma)$ }
\\[4pt]
\mbox{and $|g|\geq c>0$ $\sigma$-a.e. on $\Sigma$}.
\end{array}
\end{eqnarray}
Another useful simple property is that, given any $p\in[1,\infty]$
and $s\in(0,1)$, if $F:{\mathbb{R}}\to{\mathbb{R}}$ is a bounded
Lipschitz function then
\begin{eqnarray}\label{akjhlu-eW.3}
F\circ f\in B^{p,p}_{s,\,{\rm loc}}(\Sigma)\,\,\mbox{ for every
}\,\,f\in B^{p,p}_{s}(\Sigma).
\end{eqnarray}
Finally, we note that in the case when $\Sigma$ is the graph of a
Lipschitz function $\varphi:{\mathbb{R}}^{n-1}\to{\mathbb{R}}$, from
\cite[Proposition~2.9, p.\,33]{MM} and real interpolation we obtain
that, for each $p\in(1,\infty)$ and $s\in(0,1)$,
\begin{equation}\label{eq:tgGV}
f\in B^{p,p}_s(\Sigma)\Longleftrightarrow f(\cdot,\varphi(\cdot))\in
B^{p,p}_s({\mathbb{R}}^{n-1}).
\end{equation}

\begin{proposition}\label{bd-Bes}
Assume $\Omega\subset{\mathbb{R}}^{n}$ is a Lebesgue measurable set
whose boundary is compact, Ahlfors regular, and satisfies
\eqref{Tay-1}. Then
\begin{equation}\label{mcskF}
{\mathcal{C}}^{{}^{\rm pv}}:B^{p,p}_s(\partial\Omega)
\otimes{\mathcal{C}}\!\ell_{n}\longrightarrow
B^{p,p}_s(\partial\Omega)\otimes{\mathcal{C}}\!\ell_{n}
\end{equation}
is well-defined and bounded for each $p\in[1,\infty]$ and
$s\in(0,1)$.
\end{proposition}

\begin{proof}
One way to see this is via real interpolation (cf.
\cite[S~8.1]{HaMuYa08} for a version suiting the current setting)
between the boundedness result proved in Theorem~\ref{i65r5ED}
(corresponding to \eqref{mcskF} when $p=\infty$; cf.
\eqref{Ugagb64dc}), and the fact that the operator ${\mathcal{C}}^{{}^{\rm pv}}$
in \eqref{mcskF} with $p=1$ is also bounded (which follows from the
atomic/molecular theory for the Besov scale on spaces of homogeneous
type from \cite{HaYa03}).
\end{proof}

In order to present the extension of Theorem~\ref{Main-T1aa}
mentioned earlier to the scale of Besov spaces, we make the
following definition.

\begin{definition}\label{lipdom.BES}
Given $p\in[1,\infty]$ and $s\in(0,1)$, call a nonempty, open,
proper subset $\Omega$ of ${\mathbb{R}}^n$ a $B^{p,p}_{s+1}$-{\tt
domain} provided it may be locally identified\,\footnote{in the sense
described in Definition~\ref{lipdom}} near boundary points with the
upper-graph of a real-valued function $\varphi$ defined in
${\mathbb{R}}^{n-1}$ with the property that $\partial_j\varphi\in
B^{p,p}_s({\mathbb{R}}^{n-1})$ for each $j\in\{1,\dots,n-1\}$.
\end{definition}

The stage has been set for stating and proving the following result.

\begin{theorem}\label{Thm-Besov}
Assume $\Omega\subseteq\mathbb{R}^n$ is an Ahlfors regular domain with a compact
boundary, satisfying $\partial\Omega=\partial(\overline{\Omega})$.
Then for each $s\in(0,1)$ and $p\in[1,\infty]$ with the property
that $sp>n-1$ the following claims are equivalent:
\begin{enumerate}
\item[{\rm (a)}] $\Omega$ is a $B^{p,p}_{s+1}$-domain;
\item[{\rm (b)}] the distributional Riesz transforms associated
with $\partial\Omega$ satisfy
\begin{equation}\label{eq:RIESZ33-BB}
R_j1\in B^{p,p}_s(\partial\Omega)\,\,\,\mbox{ for each
}\,\,j\in\{1,\dots,n\};
\end{equation}
\end{enumerate}
\end{theorem}

\begin{proof}
Consider the implication ${\rm (b)}\Rightarrow{\rm (a)}$. The
starting point is the observation that \eqref{eq:RIESZ33-BB} and
\eqref{akjhlu-eW} imply \eqref{eq:RIESZ33} for
$\alpha:=s-\tfrac{n-1}{p}\in(0,1)$. As such, Theorem~\ref{Main-T1aa}
applies and gives that $\Omega$ is a domain of class
${\mathscr{C}}^{1+\alpha}$. Hence, locally, the outward unit normal
$\nu$ to $\Omega$ has components $(\nu_j)_{1\leq j\leq n}$ of the
form
\begin{equation}\label{ndfidedo}
\nu_j(x',\varphi(x'))= \left\{
\begin{array}{ll}
\frac{\partial_j\varphi(x')}{\sqrt{1+|\nabla\varphi(x')|^2}} &\mbox{
if }\,1\leq j\leq n-1,
\\[8pt]
-\frac{1}{\sqrt{1+|\nabla\varphi(x')|^2}} &\mbox{ if }\,j=n,
\end{array}
\right.
\end{equation}
where $\varphi\in{\mathscr{C}}^{1+\alpha}({\mathbb{R}}^{n-1})$ is a
real-valued function whose upper-graph locally describes $\Omega$.
Without loss of generality it may be assumed that $\varphi$ has
compact support.

On the other hand, from the assumption \eqref{eq:RIESZ33-BB},
Proposition~\ref{bd-Bes}, and \eqref{eq:Csq22} we may conclude that
\begin{equation}\label{eq:ttyy}
\nu\in B^{p,p}_s(\partial\Omega).
\end{equation}
On account of this membership and \eqref{eq:tgGV} we obtain
\begin{equation}\label{eq:ttyy.2}
\nu_j(\cdot,\varphi(\cdot))\in B^{p,p}_s({\mathbb{R}}^{n-1})
\,\,\mbox{ for each }\,\,j\in\{1,\dots,n\}.
\end{equation}
Upon recalling \eqref{akjhlu-eW.2}-\eqref{akjhlu-eW.2bb}, this
further yields
\begin{equation}\label{kdhgw}
\partial_j\varphi=\frac{\nu_j(\cdot,\varphi(\cdot))}{\nu_n(\cdot,\varphi(\cdot))}
\in B^{p,p}_s({\mathbb{R}}^{n-1})\,\,\mbox{ for each
}\,\,j\in\{1,\dots,n-1\},
\end{equation}
proving that $\Omega$ is a $B^{p,p}_{s+1}$-domain.

Concerning the implication ${\rm (a)}\Rightarrow{\rm (b)}$, assume
that $\Omega$ is a $B^{p,p}_{s+1}$-domain with $s,p$ as before. From
definitions and \eqref{akjhlu-eW} (used with
$\Sigma:={\mathbb{R}}^{n-1}$) it follows that $\Omega$ is a domain
of class ${\mathscr{C}}^{1+\alpha}$ with $\alpha:=s-\tfrac{n-1}{p}$.
Hence, in particular, $\Omega$ is a Lipschitz domain. We claim that
\eqref{eq:ttyy} holds. Thanks to \eqref{eq:tgGV}, justifying this
claim comes down to proving that \eqref{eq:ttyy.2} holds, where
$\varphi$ is a real-valued function defined in ${\mathbb{R}}^{n-1}$
satisfying $\partial_j\varphi\in B^{p,p}_s({\mathbb{R}}^{n-1})$ for
each $j\in\{1,\dots,n-1\}$, and whose upper-graph locally describes
$\Omega$ (again, without loss of generality it may be assumed that
$\varphi$ has compact support). To this end, consider the function
$F:{\mathbb{R}}\to{\mathbb{R}}$ given by
$F(t):=\tfrac{1}{\sqrt{1+|t|}}$ for each $t\in{\mathbb{R}}$, and
note that $F$ is both bounded and Lipschitz. Since by
\eqref{akjhlu-eW.2}
\begin{equation}\label{eq:rtgV}
|\nabla\varphi|^2=\sum\limits_{j=1}^{n-1}(\partial_j\varphi)(\partial_j\varphi)\in
B^{p,p}_s({\mathbb{R}}^{n-1}),
\end{equation}
it follows from \eqref{akjhlu-eW.3} that
\begin{equation}\label{eq:rtgV.2}
\nu_n(\cdot,\varphi(\cdot))=-F\circ|\nabla\varphi|^2\in
B^{p,p}_{s,\,{\rm loc}}({\mathbb{R}}^{n-1}).
\end{equation}
Granted this, another reference to \eqref{akjhlu-eW.2} gives that
for each $j\in\{1,\dots,n-1\}$
\begin{equation}\label{ndfidedo.222}
\nu_j(\cdot,\varphi(\cdot))=\frac{\partial_j\varphi}{\sqrt{1+|\nabla\varphi|^2}}
=-\partial_j\varphi\cdot\nu_n(\cdot,\varphi(\cdot))\in
B^{p,p}_s({\mathbb{R}}^{n-1}).
\end{equation}
This finishes the proof of \eqref{eq:ttyy.2}, hence completing the
justification of \eqref{eq:ttyy}. Having established this, bring in
identity \eqref{eq:CHba.11} in order to conclude on account of
Proposition~\ref{bd-Bes} that
\begin{equation}\label{eq:CHba.11.pauh}
\sum_{j=1}^n(R_j1)e_j=\sum_{j=1}^n
\big(R^{{}^{\rm pv}}_j1\big)e_j=-{\mathcal{C}}^{{}^{\rm pv}}\!
\nu\in B^{p,p}_s(\partial\Omega)\otimes{\mathcal{C}}\!\ell_{n}.
\end{equation}
Since this readily implies \eqref{eq:RIESZ33-BB}, the implication
${\rm (a)}\Rightarrow{\rm (b)}$ is established.
\end{proof}

Lastly, we remark that the limiting case $s=1$ of
Theorem~\ref{Thm-Besov} also holds provided $p\in(n-1,\infty)$ and
the Besov space intervening in \eqref{eq:RIESZ33-BB} is replaced by
$L^p_1(\partial\Omega)$, the $L^p$-based Sobolev space of order one
on $\partial\Omega$ considered in \cite{HoMiTa10} (in which scenario
$\Omega$ is an $L^p_2$-domain, in a natural sense). The proof
follows the same blue-print, and makes use of the fact that
${\mathcal{C}}^{{}^{\rm pv}}$ is a bounded operator from
$L^p_1(\partial\Omega)\otimes{\mathcal{C}}\!\ell_{n}$ into itself
(cf. \cite{IMiMiTa.1}, \cite{IMiMiTa} in this regard).

\vskip 0.20in
\begin{minipage}[t]{7.5cm}

\noindent {\tt Dorina Mitrea}

\noindent Department of Mathematics

\noindent University of Missouri at Columbia

\noindent Columbia, MO 65211, USA

\vskip 0.08in

\noindent {\tt e-mail}: {\it mitread\@@missouri.edu}

\vskip 0.15in

\noindent {\tt Marius Mitrea}

\noindent Department of Mathematics

\noindent University of Missouri at Columbia

\noindent Columbia, MO 65211, USA

\vskip 0.08in

\noindent {\tt e-mail}: {\it mitream\@@missouri.edu}

\end{minipage}
\hfill
\begin{minipage}[t]{7.5cm}

\noindent {\tt Joan Verdera}

\noindent Department de Matem\`atiques

\noindent Universitat Aut\`onoma de Barcelona

\noindent 08193 Bellaterra, Barcelona, Catalonia

\vskip 0.08in

\noindent {\tt e-mail}: {\it jvm\@@mat.uab.cat}
\end{minipage}
\end{document}